\documentclass{amsart}
\usepackage{amssymb}
\usepackage[utf8]{inputenc}
\usepackage{amsmath,amsthm,amstext,amssymb,bbm}
\usepackage{amscd,  amsfonts,  amsbsy,  mathrsfs, upgreek, mathtools, stmaryrd, bbm}
\usepackage{todonotes} 
\usepackage{pdfsync}
\usepackage{hyperref}

\usepackage{geometry}
\newtheorem{theorem}{Theorem}[section]
\newtheorem{proposition}[theorem]{Proposition}
\newtheorem{lemma}[theorem]{Lemma}
\newtheorem{corollary}[theorem]{Corollary}

\newtheorem*{claim*}{Claim}
\newtheorem*{subclaim*}{Subclaim}
\theoremstyle{definition}
\newtheorem{definition}[theorem]{Definition}
\newtheorem{remark}[theorem]{Remark}

\newcommand{\Ce}{{\mathcal{C}}}

\newcommand{\Ee}{{\mathcal{E}}}

\newcommand{\SR}{\mathrm{SR}}
\newcommand{\HSR}{\mathrm{HSR}}
\newcommand{\VSR}{\mathrm{VSR}}
\newcommand{\WSR}{\mathrm{WSR}}
\newcommand{\PR}{\mathrm{PSR}}
\newcommand{\WPR}{\mathrm{WPSR}}
\newcommand{\PRP}{\mathrm{PRP}}

\renewcommand{\emptyset}{\varnothing}
\newcommand{\dom}[1]{{{\rm{dom}}(#1)}}
\newcommand{\length}[1]{{\rm{lh}}({#1})}

\newcommand{\otp}[1]{{{\rm{otp}}\left(#1\right)}}

\newcommand{\lub}{{\rm{lub}}}
\newcommand{\POT}[1]{{\mathcal{P}}({#1})}

\newcommand{\seq}[2]{\langle{#1}~\vert~{#2}\rangle}
\newcommand{\map}[3]{{#1}:{#2}\longrightarrow{#3}}
\newcommand{\Map}[5]{{\map{#1}{#2}{#3}}; ~ {#4}\mapsto{#5}}
\newcommand{\ran}[1]{{{\rm{ran}}(#1)}}
\newcommand{\id}{{\rm{id}}}
\newcommand{\crit}[1]{{{\rm{crit}}\left({#1}\right)}}
\newcommand{\calL}{\mathcal{L}}
\newcommand{\rank}[1]{{\rm{rnk}}({#1})}
\newcommand{\cof}[1]{{{\rm{cof}}(#1)}}
\newcommand{\Set}[2]{\{{#1}~ \vert~{#2}\}}

\newcommand{\anf}[1]{{\text{``}\hspace{0.3ex}{#1}\hspace{0.3ex}\text{''}}}

\newcommand{\ZFC}{{\rm{ZFC}}}
\newcommand{\On}{{\rm{Ord}}}
\newcommand{\Lim}{{\rm{Lim}}}

\newcommand{\betrag}[1]{\vert{#1}\vert}

\newcommand{\calE}{\mathcal{E}}

\newcommand{\calS}{\mathcal{S}}

\newcommand{\goedel}[2]{{\prec}{#1},{#2}{\succ}}

\title{Patterns of Structural Reflection in the large-cardinal hierarchy}

\author{Joan Bagaria}
\address{ICREA (Instituci\'o Catalana de Recerca i Estudis Avan\c{c}ats) and
\newline \indent Departament de Matem\`atiques i Inform\`atica, Universitat de Barcelona. 
Gran Via de les Corts Catalanes, 585,
08007 Barcelona, Catalonia.}
\email{joan.bagaria@icrea.cat}

\author{Philipp L\"ucke}
\address{Fachbereich Mathematik, Universit\"at Hamburg, Bundesstra{\ss}e 55, Hamburg, 20146, Germany}
\email{philipp.luecke@uni-hamburg.de}

\thanks{The authors would like to thank the anonymous referee for the careful reading of the manuscript and numerous helpful comments. 
The research of both authors was supported by the Spanish Government under grant EUR2022-134032. 
In addition, the first author was supported by the Generalitat de Catalunya (Catalan Government) under grant 2021 SGR 00348, and by the Spanish Government under grant MTM-PID2020-116773GB-I00. 
The second author was supported by the Deutsche Forschungsgemeinschaft (Project number 522490605). 
Finally, this project has received funding from the European Union’s Horizon 2020 research and innovation programme under the Marie Sk{\l}odowska-Curie grant agreement No 842082 (Project \emph{SAIFIA: Strong Axioms of Infinity -- Frameworks, Interactions and Applications}).}


\begin{document}

\begin{abstract}
We unveil new patterns of Structural Reflection in the large-cardinal hierarchy below the first measurable cardinal. Namely, we give two different characterizations of strongly unfoldable and subtle cardinals in terms of a weak form of the principle of Structural Reflection, and also in terms of weak product structural reflection. Our analysis prompts the introduction of the new notion of $C^{(n)}$-strongly unfoldable cardinal for every natural number $n$, and we show that these cardinals form a natural hierarchy between strong unfoldable and subtle cardinals  analogous to the known hierarchies of $C^{(n)}$-extendible  and $\Sigma_n$-strong cardinals. These results show that the relatively low region of the large-cardinal hierarchy comprised between the first strongly unfoldable and the first subtle cardinals is completely analogous to the much higher region between the first strong and the first Woodin cardinals, and also to the much further upper region of the hierarchy ranging between the first supercompact and the first Vop\v{e}nka cardinals.
\end{abstract}

\maketitle

%%%%%%%%%%%%%%%%%%%%
%%%%%%%%%%%%%%%%%%%%

\section{Introduction}

Large cardinals are transfinite cardinal numbers with associated properties that make them very large, so much so that their existence cannot be proved in ZFC. Since the weakest large cardinals, the \emph{weakly inaccessible}, were first defined and studied by Hausdorff over a century ago, in 1908, a plethora of different and much stronger large cardinals have  since then been identified  in a great variety of contexts and  taking many different forms. The book of Kanamori \cite{MR1994835} gives a comprehensive overview of the rich world of large cardinals, the world of the ``\emph{Higher Infinite}" as the book's title reads. 
Since the book's first edition, published in 1994 in the wake of the groundbreaking results of Martin-Steel \cite{MS:PPD} and Woodin \cite{Woodin:PNAS} establishing the tight connections between  large cardinals and the determinacy of sets of reals, the theory of large cardinals has been expanding  in multiple directions, yielding  solutions to well-known set-theoretic problems -- e.g., in the arithmetic of singular cardinals (see \cite{GitPrikry}) or in the combinatorial properties of small uncountable cardinals (see \cite{For:denseideal} and  \cite{FoMa:MutStat}) -- as well as fertile applications to other areas, from general topology (see \cite{BM2} and \cite{Fl:NMSC}) to algebraic topology and homotopy theory (see \cite{BCMR} and \cite{BagariaWilsonRefl}), to abelian groups (see \cite{Ba-Ma:GR} and \cite{MR4474130}), etc. 
The use of large cardinals is most effective in conjunction with the forcing technique (e.g., Prikry-type forcing), and also via the construction and analysis of canonical inner models in which large cardinals exist -- the so-called \emph{inner model program}. Indeed, to settle a given  statement $\varphi$ one typically assumes, on the one hand, the existence of some suitable large cardinal and, by forcing, produces a model of set theory in which  $\varphi$ holds. On the other hand, one assumes $\varphi$ and shows that some large cardinal (or even the same kind of large cardinal) exists in a canonical inner model, thereby showing that $\varphi$ is (equi)consistent with the existence of that large cardinal.

However, in spite of their enormous success, large-cardinal axioms, i.e., the axioms asserting that such and such large cardinals exist, have remained a long-standing mystery from a foundational point of view. Indeed, a precise definition of \emph{large cardinal}, which would encompass all the variety of known large cardinals into a single notion, is still lacking. Moreover, the fact that all known large cardinals line up into a well-ordered hierarchy of consistency strength remains unexplained. Furthermore, no convincing intrinsic justification of their naturalness, or their status as true axioms of set theory has been yet put forward, and their acceptance has so far been supported mainly by their fruitful consequences.

In recent years, a new program of reformulating large cardinals in terms of a general reflection principle called \emph{Structural Reflection} ($\SR$) (see \cite{Bag-SR}), has succeeded at characterizing well-known large cardinals at several regions of the large-cardinal hierarchy in terms of this principle.
The program started, still in a veiled form, with  \cite{zbMATH06029795} and  \cite{BCMR}, in which large cardinals in the region spanning from supercompact  and extendible cardinals to Vop\v{e}nka's Principle are characterized in terms of a strong form of L\"owenheim-Skolem type of reflection, equivalent to stratified forms of Vop\v{e}nka's Principle restricted to    classes of structures of a given definitional complexity. The work continued in \cite{BagariaWilsonRefl}, in which a similar characterization in terms of $\SR$, this time using products and homomorphisms instead of elementary embeddings, was given for large cardinals in the region between strong and Woodin cardinals. Further work along the same lines was done in the lower regions of the large-cardinal hierarchy, below the first measurable cardinal. In \cite{MR3519447} (see also \cite{SRminus} and \cite{Bag-SR}),   reformulations in terms of $\SR$ are given for the weakest of large cardinals, namely those between weakly inaccessible and weakly compact cardinals. Furthermore, in \cite{BGS}, a generic form of $\SR$ is used to characterize cardinals that lie between almost-remarkable and virtually extendible. Several other similar $\SR$ characterizations of large cardinals lying in other regions of the large-cardinal hierarchy are given in \cite{Bag-SR} -- e.g., for globally superstrong cardinals -- and even large-cardinal principles such as ``\emph{$0^\sharp$ exists}" or ``\emph{$0^\dagger$ exists}" are shown to be equivalent to $\SR$ for definable classes of structures belonging to canonical  inner models. Finally, in \cite{BALU1}, the authors gave characterizations of large cardinals in the uppermost regions of the large-cardinal hierarchy,  between Vop\v{e}nka's Principle and ${\rm I}1$-cardinals. Moreover, our analysis of those cardinals in terms of  $\SR$ allowed  to formulate a new hierarchy of large cardinal notions that has the potential to go beyond all known large cardinals not known to be inconsistent with $\ZFC$.

In the present article, we show that the same pattern of $\SR$ that holds  between a supercompact  and a Vop\v{e}nka cardinal, and between a strong and a Woodin cardinal, also holds   between a strongly unfoldable and a subtle cardinal. These correlations can be informally expressed by the following equations: 

\begin{equation*}
 \frac{\textit{Vop\v{e}nka}}{\textit{supercompact}} ~ {} ~ {} ~ = ~ {} ~ {} ~ \frac{\textit{Woodin}}{\textit{strong}} ~ {} ~ {} ~  = ~ {} ~ {} ~ \frac{\textit{subtle}}{\textit{strongly unfoldable}} 
\end{equation*}
\smallskip

As strongly unfoldable and subtle cardinals lie below the first measurable cardinal and can exist in G\"odel's constructible universe $L$, our results show that the large-cardinal hierarchy is highly homogeneous, repeating the same pattern in the upper, central, and lower regions of the hierarchy. The reformulation of large cardinals in terms of $\SR$ does explain the empirical fact that they form a well-ordered hierarchy, and also attests to their naturalness and intrinsic justification as true axioms of set theory (see \cite{Bag-SR}). Moreover, each new $\SR$ characterization, like the ones given in the present article, constitutes a further step towards the ultimate goal of yielding a uniform formulation of all known large cardinals in terms of a single unifying principle. Furthermore, what could be termed as a ``side effect" is that  every $\SR$ reformulation of a known large-cardinal notion suggests and gives rise to new definitions of large cardinals,  filling in the gaps in the large-cardinal hierarchy, e.g., $C^{(n)}$-extendible cardinals, which lie between supercompact cardinals and Vop\v{e}nka's Principle. Let us  describe next more precisely our results and explain how they fit in and contribute to the $\SR$ program.

%%%%%%%%%%%%
%%%%%%%%%%%%
%%%%%%%%%%%%

\subsection{Structural Reflection}

Recall that, given a  natural number $n$, $C^{(n)}$ is the $\Pi_n$-definable closed unbounded class of ordinals $\alpha$ that are $\Sigma_n$-correct in $V$, that is, $V_\alpha$ is a $\Sigma_n$-elementary substructure of $V$, written $V_\alpha \prec_{\Sigma_n}V$. Thus, $C^{(0)}$ is the class of all ordinal numbers and $C^{(1)}$ is the class of all uncountable cardinals $\kappa$ such that $V_\kappa =H_\kappa$.

Let us now recall the following different variants of the principle of Structural Reflection:

\begin{definition}\label{defsr}
 Let $\Ce$ be a  class\footnote{We work in $\ZFC$. Therefore,  when considering proper classes, we shall always mean classes that are definable, possibly using sets as parameters.} of structures\footnote{In this paper, the term \emph{structure} refers to structures for countable first-order languages. The cardinality of a structure is defined as the cardinality of its domain.} of the same type and let $\kappa$ be an infinite cardinal. 
\begin{enumerate}
    \item (Bagaria-V\"a\"an\"anen, ~ \cite{MR3519447})
  $\SR_\Ce(\kappa)$ denotes the statement that for every structure $B$ in $\Ce$, there exists a structure $A$ in $\Ce$ of cardinality less than $\kappa$  and an elementary embedding of $A$ into $B$.
 \item (Bagaria, ~ \cite{zbMATH06029795}) 
  $\HSR_\Ce(\kappa)$ denotes the statement that for every structure $B$ in $\Ce$, there exists a structure $A$ in $\Ce \cap H_\kappa$    and an elementary embedding of $A$ into $B$.
 \item (Bagaria, ~ \cite{Bag-SR})
 $\VSR_\Ce(\kappa)$ denotes the statement that for every structure $B$ in $\Ce$, there exists a structure $A$ in $\Ce \cap V_\kappa$  and an elementary embedding of $A$ into $B$.
 \end{enumerate}
\end{definition}

Clearly, for a given class $\Ce$ and a cardinal $\kappa$, the principle $\HSR_\Ce(\kappa)$ implies both the principle $\SR_\Ce(\kappa)$ and the principle  $\VSR_\Ce(\kappa)$. Also, if $\kappa$ is an element of $C^{(1)}$, then we have $H_\kappa=V_\kappa$, and therefore the principles $\HSR_\Ce(\kappa)$ and $\VSR_\Ce(\kappa)$ are identical. Moreover, for classes $\Ce$ that are closed under isomorphic images, the principle $\SR_\Ce(\kappa)$ implies the principle $\HSR_\Ce(\kappa)$. Thus, for a class $\Ce$ closed under isomorphic images and $\kappa \in C^{(1)}$, the three principles $\SR_\Ce(\kappa)$, $\HSR_\Ce(\kappa)$ and $\VSR_\Ce(\kappa)$ are the same.

The following theorem, proved in {\cite[Corollary 4.10]{zbMATH06029795}} using Magidor's characterization of the first supercompact cardinal as the first cardinal that reflects the class of all structures of the form $\langle V_\alpha,\in\rangle$ for  ordinals $\alpha$ (see {\cite[Theorem 1]{MR295904}}),  gives a reformulation of  supercompact cardinals in terms of $\SR$ and it may be considered the first result of the $\SR$ program.

\begin{theorem}[\cite{zbMATH06029795}]\label{theorem:JoanBoldSupercompactSigma2}
 The following statements are equivalent for every cardinal $\kappa$: 
 \begin{enumerate}
  \item The cardinal $\kappa$ is either supercompact or a limit of supercompact cardinals. 
  
  \item The principle $\SR_\Ce(\kappa)$ (equivalently, $\HSR_\Ce(\kappa)$, or $\VSR_\Ce(\kappa)$) holds for every class $\Ce$ of structures of the same type that is definable by a $\Sigma_2$-formula with parameters in $V_\kappa$. 
 \end{enumerate}
\end{theorem}

\begin{corollary}
 The following statements are equivalent for every cardinal $\kappa$: 
 \begin{enumerate}
  \item The cardinal $\kappa$ is the least supercompact cardinal. 
  
  \item The cardinal $\kappa$ is the least cardinal with the property that $\SR_\Ce(\kappa)$ (equivalently, $\HSR_\Ce(\kappa)$, or $\VSR_\Ce(\kappa)$) holds for every class $\Ce$ of structures of the same type that is definable by a $\Sigma_2$-formula with parameters in $V_\kappa$. \qed 
 \end{enumerate}
\end{corollary}

For classes of structures of higher definitional complexity an analogous equivalence is given in the theorem below, proved in  {\cite[Theorem 4.18]{zbMATH06029795}}, which characterizes extendible and $C^{(n)}$-extendible cardinals in terms of $\SR$. 
Recall that a cardinal $\kappa$ is \emph{$C^{(n)}$-extendible} if for every $\lambda > \kappa$, there is an ordinal $\mu$ and  an elementary embedding $\map{j}{V_\lambda}{V_\mu}$ with $\crit{j}=\kappa$, $j(\kappa)>\lambda$ and  $j(\kappa)\in C^{(n)}$. A cardinal is \emph{extendible} if and only if it is $C^{(1)}$-extendible.

\begin{theorem}[\cite{zbMATH06029795}]\label{theorem:JoanBoldSupercompactSigma22}
 The following statements are equivalent for every cardinal $\kappa$: 
 \begin{enumerate}
  \item $\kappa$ is either $C^{(n)}$-extendible or a limit of $C^{(n)}$-extendible cardinals. 
  
  \item The principle $\SR_\Ce(\kappa)$ (equivalently, $\HSR_\Ce(\kappa)$, or $\VSR_\Ce(\kappa)$) holds for every class $\Ce$ of structures of the same type that is definable by a $\Sigma_{n+2}$-formula with parameters in $V_\kappa$. 
 \end{enumerate}
\end{theorem}

\begin{corollary}
 The following statements are equivalent for every cardinal $\kappa$: 
 \begin{enumerate}
  \item $\kappa$ is the least $C^{(n)}$-extendible cardinal. 
  
  \item $\kappa$ is the least cardinal with the property that $\SR_\Ce(\kappa)$ (equivalently, $\HSR_\Ce(\kappa)$, or $\VSR_\Ce(\kappa)$) holds for every class $\Ce$ of structures of the same type that is definable by a $\Sigma_{n+2}$-formula with parameters in $V_\kappa$. \qed 
 \end{enumerate}
\end{corollary}

 Remember that \emph{Vop\v{e}nka's Principle} is the scheme of axioms stating that for every proper class $\Ce$ of structures of the same type, there exist $A,B\in\Ce$ with $A\neq B$ and an elementary embedding from $A$ to $B$. In \cite{zbMATH06029795}, the first author obtains the following characterization of this principle:

 \begin{theorem}[\cite{zbMATH06029795}]\label{theorem:CharVP}
  The following schemes of axioms are equivalent over $\ZFC$: 
  \begin{enumerate}
      \item Vop\v{e}nka's Principle. 
      
      \item For every natural number $n$, there exists a $C^{(n)}$-extendible cardinal. 
      
      \item For every natural number $n$, there exists a proper class of $C^{(n)}$-extendible cardinals.  
      
      \item For every class $\Ce$ of structures of the same type, there exists a cardinal $\kappa$ such that $\SR_\Ce(\kappa)$  (equivalently, $\HSR_\Ce(\kappa)$, or $\VSR_\Ce(\kappa)$) holds.
  \end{enumerate}
 \end{theorem}

The next result, which shall be proved in Section \ref{section2} below, uses the \emph{set}-version of this equivalence to provide a canonical characterization of \emph{Vop\v{e}nka cardinals}, i.e., inaccessible cardinals $\delta$ with the property that for every set $\Ce$ of structures of the same type with $\Ce\in V_{\delta+1}\setminus V_\delta$, there exist $A,B\in\Ce$ with $A\neq B$ and an elementary embedding from $A$ to $B$ (see, for example, \cite{hamkins2016vopvenka} and \cite{MR3324126}).

\begin{theorem}\label{main:Vopenka}
 The following statements are equivalent for every uncountable cardinal $\delta$: 
 \begin{enumerate}
  \item\label{item:Vopenka1} The cardinal $\delta$ is a Vop\v{e}nka cardinal. 
  
  \item\label{item:Vopenka2} For every set $\Ce$ of structures of the same type with $\Ce\in V_{\delta+1}\setminus V_\delta$, there exists a cardinal $\kappa<\delta$ with the property that the principle $\SR_\Ce(\kappa)$ (equivalently, $\HSR_\Ce(\kappa)$, or $\VSR_\Ce(\kappa)$)  holds. 
 \end{enumerate}
\end{theorem}

%%%%%%%%%%%%%%%%%%%%
%%%%%%%%%%%%%%%%%%%%

\subsection{Weak Structural Reflection}

We now use weaker forms of $\SR$, defined below, to obtain canonical characterizations of smaller large cardinals, lying below the first measurable cardinal.

\begin{definition}
\label{defminus}
 Let $\Ce$ be a non-empty class of structures of the same type and let $\kappa$ be an infinite cardinal. 
 \begin{enumerate}
  \item (Bagaria--V\"a\"an\"anen, ~ \cite{MR3519447})  $\SR^-_\Ce(\kappa)$ is the statement that for every structure $B$ in $\Ce$ of cardinality $\kappa$, there exists a structure $A$ in $\Ce$ of cardinality less than $\kappa$ and an elementary embedding of $A$ into $B$.
  
  \item (Bagaria--V\"a\"an\"anen, ~ \cite{MR3519447}) Let $\SR^{--}_\Ce(\kappa)$ denote the statement that $\Ce$ contains a structure of cardinality less than $\kappa$. 
  
  \item $\WSR_\Ce(\kappa)$ denotes the conjunction of $\SR^-_\Ce(\kappa)$ and $\SR^{--}_\Ce(\kappa)$. 
  \item Let $\HSR_\Ce^-(\kappa)$ and $\VSR_\Ce^-(\kappa)$ be the same as $\HSR_\Ce(\kappa)$ and $\VSR_\Ce(\kappa)$, respectively (as given in Definition \ref{defsr}), but restricted to structures $B$ of cardinality $\kappa$.
 \end{enumerate}
\end{definition}

For a given class $\Ce$ and a cardinal $\kappa$, the principle $\HSR^-_\Ce(\kappa)$ clearly implies both $\SR^-_\Ce(\kappa)$ and $\VSR^-_\Ce(\kappa)$; and if $\kappa$ is an element of $C^{(1)}$, then the principles $\HSR^-_\Ce(\kappa)$ and $\VSR^-_\Ce(\kappa)$ are equivalent. Also, if $\Ce$ is closed under isomorphic images, then the principle $\SR^-_\Ce(\kappa)$ implies the principle  $\HSR_\Ce(\kappa)$. Hence, for  $\Ce$  closed under isomorphic images and $\kappa \in C^{(1)}$, the principles  $\HSR^-_\Ce(\kappa)$, $\SR^-_\Ce(\kappa)$ and $\VSR^-_\Ce(\kappa)$ are the same.

Next, we recall the large-cardinal notion of strong unfoldability, introduced by Villaveces in model-theoretic investigations of models of set theory.

\begin{definition}[Villaveces, ~ \cite{MR1649079}]
 An inaccessible cardinal $\kappa$ is \emph{strongly unfoldable} if for every ordinal $\lambda$ and every transitive $\mathrm{ZF}^-$-model $M$ of cardinality $\kappa$ with $\kappa\in M$ and ${}^{{<}\kappa}M\subseteq M$, there is a transitive set $N$ with $V_\lambda\subseteq N$ and an elementary embedding $\map{j}{M}{N}$ with $\crit{j}=\kappa$ and $j(\kappa)\geq\lambda$. 
\end{definition}

Every strongly unfoldable cardinal is a totally indescribable element  of $C^{(2)}$ (see {\cite[Theorem 2]{MR2279655}} and {\cite[Section 3]{luecke2021strong}}). 
In particular, the first strongly unfoldable cardinal is bigger than the first subtle cardinal (see Definition \ref{defsubtle} below) and smaller than the first strong cardinal. 
%and smaller than the first Ramsey cardinal (\cite{MR1649079}). 
 The following characterization of strongly unfoldable cardinals in terms of weak $\SR$ will be proved in Section \ref{section3}:

\begin{theorem}\label{theorem:CharSr-Sr--}
 The following statements are equivalent for every  cardinal $\kappa$: 
 \begin{enumerate}
   \item\label{item:EiterSunfOrSC} $\kappa$ is either strongly unfoldable or a limit of supercompact cardinals. 
   
  \item\label{item:BothWeakSR} The principle $\WSR_\Ce(\kappa)$ holds for every class $\Ce$ of structures of the same type that is definable by a $\Sigma_2$-formula with parameters in $V_\kappa$. 
  
  \item\label{item:BothWeakHSR} $\kappa$ is an element of $C^{(2)}$ and the principle $\HSR_\Ce^- (\kappa)$ holds for every class $\Ce$ of structures of the same type that is definable by a $\Sigma_2$-formula with parameters in $V_\kappa$. 
 \end{enumerate}
\end{theorem}

The fact that supercompact cardinals are strongly unfoldable now directly yields a characterization of strongly unfoldable cardinals through weak principles of structural reflection:

\begin{corollary}\label{corollary:LeastSUNF}
 The following statements are equivalent for every cardinal $\kappa$:
 \begin{enumerate}
  \item $\kappa$ is the least strongly unfoldable cardinal. 
  
  \item $\kappa$ is the least cardinal with the property that the principle $\WSR_\Ce(\kappa)$ holds for every class $\Ce$ of structures of the same type that is definable by a $\Sigma_2$-formula with parameters in $V_\kappa$. 
  \item $\kappa$ is the least element of the class $C^{(2)}$ with the property that the principle $\HSR_\Ce^-(\kappa)$ holds for every class $\Ce$ of structures of the same type that is definable by a $\Sigma_2$-formula with parameters in $V_\kappa$. \qed 
 \end{enumerate} 
\end{corollary}

For classes of structures of higher definitional complexity, we shall prove analogous equivalences, using the following natural strengthening of strong unfoldability:

\begin{definition}
\label{defCnsf}
 Given a natural number $n$, an inaccessible cardinal $\kappa$ is \emph{$C^{(n)}$-strongly unfoldable} if for every ordinal $\lambda\in C^{(n)}$ greater than $\kappa$ and every transitive $\mathrm{ZF}^-$-model $M$ of cardinality $\kappa$ with $\kappa\in M$ and ${}^{{<}\kappa}M\subseteq M$, there is a transitive set $N$ with $V_\lambda\subseteq N$ and an elementary embedding $\map{j}{M}{N}$ with $\crit{j}=\kappa$,  $j(\kappa)>\lambda$ and $V_\lambda\prec_{\Sigma_n}V_{j(\kappa)}^N$. 
\end{definition}

By definition, a cardinal is strongly unfoldable if and only if it is $C^{(0)}$-strongly unfoldable. A short argument (see Proposition \ref{proposition:SunfC1Sunf} below) shows that strong unfoldability also coincides with $C^{(1)}$-strong unfoldability. Moreover, it is also easy to see that all $C^{(n)}$-extendible cardinals are $C^{(n+1)}$-strongly unfoldable (see Proposition \ref{proposition:CnExtendsibleCn+1Sunf} below). 
The following result, which  shall be proved in Section \ref{section9}, characterizes $C^{(n)}$-strongly unfoldable cardinals in terms of weak $\SR$. For every natural number $n>0$, the theorem shows that $C^{(n+1)}$-strongly unfoldable cardinals are related, via weak $\SR$,  to strongly unfoldable cardinals, as $C^{(n)}$-extendible cardinals are related, via $\SR$, to supercompact cardinals.

\begin{theorem} \label{C(n)StronglyUnfoldable}
 Given a natural number $n>1$, the following statements are equivalent for every  cardinal $\kappa$: 
 \begin{enumerate} 
  \item $\kappa$ is either $C^{(n)}$-strongly unfoldable or a limit of $C^{(n-1)}$-extendible cardinals. 
   
  \item  The principle $\WSR_\Ce(\kappa)$ holds for every class $\Ce$ of structures of the same type that is definable by a $\Sigma_{n+1}$-formula with parameters in $V_\kappa$. 
  
   \item $\kappa$ is an element of $C^{(n+1)}$ and the principle $\HSR_\Ce^-(\kappa)$ holds for every class $\Ce$ of structures of the same type that is definable by a $\Sigma_{n+1}$-formula with parameters in $V_\kappa$. 
 \end{enumerate}
\end{theorem}

 By combining this theorem with Corollary \ref{corollary:LeastSUNF} and the fact that all $C^{(n)}$-extendible cardinals are $C^{(n+1)}$-strongly unfoldable, we now obtain a uniform characterization of the least $C^{(n)}$-strongly unfoldable cardinal through weak principles of structural reflection:

\begin{corollary}\label{corollary:LeastSUNF2}
 Given a natural number $n>0$, the following statements are equivalent for every cardinal $\kappa$:
 \begin{enumerate}
  \item $\kappa$ is the least $C^{(n)}$-strongly unfoldable cardinal. 
  
  \item $\kappa$ is the least cardinal with the property that the principle $\WSR_\Ce(\kappa)$ holds for every class $\Ce$ of structures of the same type that is definable by a $\Sigma_{n+1}$-formula with parameters in $V_\kappa$. 
  \item $\kappa$ is the least cardinal in $C^{(n+1)}$ with the property that the principle $\HSR_\Ce^-(\kappa)$ holds for every class $\Ce$ of structures of the same type that is definable by a $\Sigma_{n+1}$-formula with parameters in $V_\kappa$. \qed 
 \end{enumerate} 
\end{corollary}

We now want to use weak $\SR$ to isolate the canonical variations of Vop\v{e}nka's Principle and Vop\v{e}nka cardinals for the considered region of the large cardinal hierarchy. This analysis turns out to be closely related to the notion of \emph{subtle cardinals}, introduced by Jensen and Kunen in their work on the validity of strong diamond principles in the constructible universe $L$. Remember that, given a set $A$ of ordinals, a sequence $\seq{E_\alpha}{\alpha\in A}$ is called  an \emph{$A$-list} if $E_\alpha\subseteq\alpha$ holds for every $\alpha\in A$.

\begin{definition}[Jensen-Kunen, ~ \cite{jensennotes}]
\label{defsubtle}
 An infinite cardinal $\delta$ is \emph{subtle} if for every $\delta$-list $\seq{E_\gamma}{\gamma<\delta}$ and every closed unbounded subset $C$ of $\delta$, there exist $\beta<\gamma$ in $C$ with $E_\beta=E_\gamma\cap\beta$. 
\end{definition}

Subtle cardinals are strongly inaccessible,\footnote{Note that it is not necessary to include regularity in the definition of subtleness, because, if we assume that a singular cardinal $\delta$ possesses the above property, then we could pick a closed cofinal subset $C\subseteq\delta$ of order-type $\cof{\delta}$ with $\min(C)>\cof{\delta}$ and define $E_\gamma=\otp{C\cap\gamma}\subseteq\gamma$ for all $\gamma<\delta$.} and below the first subtle cardinal there are many totally indescribable cardinals. However, the first subtle cardinal is $\Pi^1_1$-describable, and therefore not weakly compact. Moreover, if $\delta$ is a subtle cardinal, then the set of $\kappa<\delta$ that are strongly unfoldable in $V_\delta$ is stationary in $\delta$ (see {\cite[Theorem 3]{MR2279655}}).

The notion of subtleness has a natural \emph{class}-version (``\emph{$\On$ is subtle}") that postulates, as a schema, that for every closed unbounded class $C$ of ordinals and every class sequence $\seq{E_\gamma}{\gamma\in\On}$, there exist $\beta<\gamma$ in $C$ with $E_\beta=E_\gamma\cap\beta$ (see {\cite[Section 5]{boney2022model}}). It turns out that for our purposes, a slight strengthening of this principle is needed. The formulation of this principle is motivated by the trivial observation that the Axiom of Choice allows us to show that a cardinal $\delta$ is subtle if and only if for every closed unbounded subset $C$ of $\delta$ and every sequence $\seq{\calE_\gamma}{\gamma<\delta}$ with $\emptyset\neq\calE_\gamma\subseteq\POT{\gamma}$ for all $\gamma<\delta$, there exist $\beta<\gamma$ in $C$ and $E\in\calE_\gamma$ with $E\cap\beta\in\calE_\beta$.

\begin{definition}
 Let ``\emph{$\On$ is essentially subtle}"  denote the axiom schema stating that for every closed unbounded class $C$ of ordinals and every class function $\calE$ on the ordinals with the property that $\emptyset\neq \calE(\gamma)\subseteq\POT{\gamma}$ for all $\gamma\in\On$, there exist $\beta<\gamma$ in $C$ and $E\in\calE(\gamma)$ with $E\cap\beta\in\calE(\beta)$.
\end{definition}

Note that this principle implies that $\On$ is subtle, and, in the presence of a definable well-ordering of $V$, both principles are equivalent. Moreover, the assumption that $\On$ is essentially subtle is obviously downwards absolute to $L$, and hence both principles have the same consistency strength over $\ZFC$. In contrast, the proof of {\cite[Theorem 5.6]{boney2022model}} produces a model of set-theory that witnesses that both principles are not equivalent over $\ZFC$. 
 The next result, proved in Section \ref{section9}, provides a direct analog of Theorem \ref{theorem:CharVP} by showing that strongly unfoldable and $C^{(n+1)}$-strongly unfoldable cardinals are related, via weak $\SR$,  to the assumption that $\On$ is essentially subtle, as supercompact and $C^{(n)}$-extendible cardinals are related, via $\SR$, to Vop\v{e}nka's Principle.

\begin{theorem}\label{theorem:OrdSubtle1}
 The following schemes of axioms are equivalent over $\ZFC$: 
 \begin{enumerate}
  \item $\On$ is essentially subtle. 

  \item For every natural number $n$, there exists a $C^{(n)}$-strongly unfoldable cardinal. 
  
  \item For every natural number $n$, there exists a proper class of $C^{(n)}$-strongly unfoldable cardinals. 
    
  \item  For every natural number $n$ and every class $\Ce$ of structures of the same type, there exists a cardinal $\kappa\in C^{(n)}$ with the property that   $\HSR^-_\Ce(\kappa)$ holds. 
 \end{enumerate}
\end{theorem}

Finally, we will also show that an analog of Theorem \ref{main:Vopenka} holds for this region of the large cardinal hierarchy. 
The following theorem, which will be proved in Section \ref{section5}, shows that subtle cardinals are related, via weak $\SR$, to strongly unfoldable and $C^{(n+1)}$-strongly unfoldable cardinals, as Vop\v{e}nka cardinals are related, via $\SR$, to supercompact and $C^{(n)}$-extendible cardinals.

\begin{theorem}\label{theorem:SubtleChar}
 The following statements are equivalent for every uncountable cardinal $\delta$: 
 \begin{enumerate}
  \item\label{item:SubtleLimSubtle} $\delta$ is either subtle or a limit of subtle cardinals. 
  
  \item\label{item:SR-SR--} For every set $\Ce$ of structures of the same type with $\Ce\in V_{\delta+1}\setminus V_\delta$, there exists a cardinal $\kappa<\delta$ with the property that  the principle $\WSR_\Ce(\kappa)$ holds. 
 \end{enumerate}
\end{theorem}

\begin{corollary}
  The following statements are equivalent for every uncountable cardinal $\delta$: 
  \begin{enumerate}
   \item $\delta$ is the least subtle cardinal. 
   
   \item $\delta$ is the least cardinal with the property that for every set $\Ce$ of structures of the same type with $\Ce\in V_{\delta+1}\setminus V_\delta$, there exists a cardinal $\kappa<\delta$ such that  the principle $\WSR_\Ce(\kappa)$ holds. \qed 
  \end{enumerate}
\end{corollary}

%%%%%%%%%%%%%%%%
%%%%%%%%%%%%%%%%

\subsection{Product Structural Reflection}

In analogy with the characterization (given essentially in \cite{BagariaWilsonRefl}, but outlined below) of strong, $\Sigma_n$-strong, and Woodin cardinals, in terms of  \emph{product} $\SR$ (see Definition \ref{defpsr} below), we work towards  additional characterizations of  strongly unfoldable, $C^{(n)}$-strongly unfoldable, and subtle cardinals in terms of \emph{weak product} $\SR$ (see Definition \ref{wpsr} below).

Let us recall that for any set $S$ of structures of the same type,  the set-theoretic product $\prod S$  is the structure whose universe is the  set of all functions $f$ with domain $S$ such that $f(\mathcal{A})\in A$ for every $\mathcal{A}=\langle A, \ldots \rangle\in S$, and whose interpretation of constant, function and relation symbols is defined point-wise. In addition, note that, if $X$ is a substructure of $\prod S$ and $\mathcal{A}=\langle A, \ldots \rangle\in S$, then the canonical projection map $$\Map{p}{X}{A}{f}{f(\mathcal{A})}$$ is a homomorphism from $X$ into $\mathcal{A}$.\footnote{As defined in {\cite[Section 1.2]{MR1221741}}.} 

The following \emph{Product Reflection Principle} ($\PRP$) was introduced in \cite[Definition 3.1]{BagariaWilsonRefl}:

\begin{definition}[Bagaria-Wilson, \cite{BagariaWilsonRefl}]\label{defPRP}
 Given  a class $\Ce$ of structures of the same type, the principle $\PRP_\Ce$ asserts that 
there is a subset $S$ of $\Ce$ such that for $B\in\Ce$, there is a homomorphism from $\prod S$ to $B$.
\end{definition}

The principle $\PRP_\Ce$ holds for every class $\Ce$ that is definable by a $\Sigma_1$-formula with parameters, as witnessed by sets of the form $V_\kappa \cap \Ce$, where $\kappa \in C^{(1)}$ has the property that  $V_\kappa$ contains the parameters used in the definition of $\Ce$ (see \cite[Proposition 3.2]{BagariaWilsonRefl}). 
Also, if $\kappa$ is either a strong cardinal or a limit of strong cardinals, then  $\PRP_\Ce$ holds for every class $\Ce$ that is defined by a $\Sigma_2$-formula with parameters in $V_\kappa$, as  witnessed by $\Ce \cap V_\kappa$  (see \cite[Proposition 3.3]{BagariaWilsonRefl}). 
Moreover, the proof of  \cite[Theorem 4.1]{BagariaWilsonRefl} shows that if $\kappa$ is a cardinal such that $\Ce \cap V_\kappa$ is non-empty and witnesses $\PRP_\Ce$ for every class $\Ce$ that is defined by a $\Pi_1$-formula without parameters, then  there exists a strong cardinal less than or equal to $\kappa$. If the same holds for every class $\Ce$ that is defined by a $\Pi_1$-formula with parameters in $V_\kappa$, then $\kappa$ is either a strong cardinal or a limit of strong cardinals. Hence,  the following statements are equivalent for every infinite cardinal $\kappa$:

\begin{enumerate}
 \item $\kappa$ is either a strong cardinal or a limit of strong cardinals. 
 
 \item $\Ce \cap V_\kappa$ witnesses $\PRP_\Ce$ for all classes $\Ce$ of structures of the same type that are definable by a $\Pi_1$-formula  (equivalently, by a $\Sigma_2$-formula) with parameters in $V_\kappa$. 
\end{enumerate}

In view of this equivalence, we may reformulate the principle $\PRP$  as a principle of Structural Reflection for products, as follows:

\begin{definition}\label{defpsr}
 Given an infinite cardinal $\kappa$ and a class $\Ce$ of structures of the same type, we let $\PR_\Ce(\kappa)$ denote the statement that $\Ce\cap V_\kappa\neq\emptyset$ and for every $B\in\Ce$, there exists a homomorphism from $\prod(\Ce\cap V_\kappa)$ to $B$.  
\end{definition}

The above equivalence can now be stated in the following way:

\begin{theorem}[\cite{BagariaWilsonRefl}]\label{theorem:StrongProductSigma2}
 The following statements are equivalent for every cardinal $\kappa$: 
 \begin{enumerate}
  \item $\kappa$ is either a strong cardinal or a limit of strong cardinals. 
  
  \item The principle $\PR_\Ce(\kappa)$ holds for every class $\Ce$ of structures of the same type that is definable by a $\Sigma_2$-formula with parameters in $V_\kappa$. 
 \end{enumerate}
\end{theorem}

\begin{corollary}
 The following statements are equivalent for every cardinal $\kappa$: 
 \begin{enumerate}
  \item $\kappa$ is the least strong cardinal. 
  
  \item $\kappa$ is the least cardinal with the property that $\PR_\Ce(\kappa)$ holds for every class $\Ce$ of structures of the same type that is definable by a $\Sigma_2$-formula with parameters in $V_\kappa$. 
 \end{enumerate}
\end{corollary}

More generally, recall the following strengthening of strongness introduced in \cite[Definition 5.1]{BagariaWilsonRefl}:

\begin{definition}[Bagaria-Wilson, \cite{BagariaWilsonRefl}]\label{defSigmaStrong}
 Let $n>0$ be a natural number. 
 \begin{enumerate}
     \item  Given an ordinal $\lambda$, a cardinal $\kappa$  is \emph{$\lambda$-$\Sigma_n$-strong}  if for every class $A$ that is definable by a $\Sigma_n$-formula without parameters, there is a transitive class $M$ with $V_\lambda \subseteq M$ and an elementary embedding $\map{j}{V}{M}$ with $\crit{j}=\kappa$, $j(\kappa)>\lambda$ and $A\cap V_\lambda \subseteq j(A)$. 
     
     \item A cardinal $\kappa$ is \emph{$\Sigma_n$-strong} if it is $\lambda$-$\Sigma_n$-strong for every ordinal $\lambda$. 
 \end{enumerate}
\end{definition}

Every strong cardinal is $\Sigma_2$-strong (see \cite[Proposition 5.2]{BagariaWilsonRefl}). Arguments contained in the proofs of  {\cite[Claim 5.12]{BagariaWilsonRefl}} and  {\cite[Theorem 5.13]{BagariaWilsonRefl}} now yield the following equivalence:

\begin{theorem}[\cite{BagariaWilsonRefl}]\label{theorem:StrongProductSigman}
 Given a natural number $n>1$, the  following statements are equivalent for every cardinal $\kappa$: 
 \begin{enumerate}
  \item $\kappa$ is either a $\Sigma_n$-strong cardinal or a limit of $\Sigma_n$-strong cardinals. 
  
  \item The principle $\PR_\Ce(\kappa)$ holds for every class $\Ce$ of structures of the same type that is definable by a $\Sigma_n$-formula with parameters in $V_\kappa$. 
 \end{enumerate}
\end{theorem}

\begin{corollary}
 Given a natural number $n>1$, the  following statements are equivalent for every cardinal $\kappa$: 
 \begin{enumerate}
  \item $\kappa$ is the least $\Sigma_n$-strong cardinal. 
  
  \item $\kappa$ is the least cardinal with the property that $\PR_\Ce(\kappa)$ holds for every class $\Ce$ of structures of the same type that is definable by a $\Sigma_n$-formula with parameters in $V_\kappa$. 
 \end{enumerate}
\end{corollary}

The results of \cite{BagariaWilsonRefl} also show that the canonical variation of Vop\v{e}nka's Principle corresponding to the above large cardinal notions is the assumption that ``\emph{$\On$ is Woodin}". This  axiom schema asserts that for every class $A$, there exists some cardinal $\kappa$ which is $A$-strong, i.e., for every $\lambda$, there is an elementary embedding $\map{j}{V}{M}$ with $\crit{j}=\kappa$, $j(\kappa)>\lambda$, $V_\lambda \subseteq M$ and  $A\cap V_\lambda =j(A)\cap V_\lambda$.

\begin{theorem}[\cite{BagariaWilsonRefl}]\label{main:Woodin2}
 The following statements are equivalent: 
 \begin{enumerate}
  \item  $\On$ is  Woodin. 
  
  \item For every class $\Ce$ of structures of the same type  there exists a cardinal $\kappa$ with the property that the principle $\PR_\Ce(\kappa)$ holds. 
  
  \item For every natural number $n$, there exists a $\Sigma_n$-strong cardinal. 
  
  \item For every natural number $n$, there exists a proper class of $\Sigma_n$-strong cardinals. 
 \end{enumerate}
\end{theorem}

Using the characterization of a cardinal $\delta$ being a Woodin cardinal in terms of the existence of cardinals $\kappa <\delta$ that are $A$-strong, for every subset $A$ of $V_\delta$ (see \cite[Theorem 26.14]{MR1994835}), 
we will extend the argument given in \cite[Theorem 5.13]{BagariaWilsonRefl}  to prove a version of Theorem \ref{main:Vopenka} for Woodinness and obtain the following characterization of Woodin cardinals in terms of $\PR$ in Section \ref{section:7}:

\begin{theorem}\label{main:Woodin}
 The following statements are equivalent for every uncountable cardinal $\delta$: 
 \begin{enumerate}
  \item  $\delta$ is a Woodin cardinal. 
  
  \item For every set $\Ce$ of structures of the same type with $\Ce\subseteq V_\delta$, there exists a cardinal $\kappa<\delta$ with the property that the principle $\PR_\Ce(\kappa)$ holds. 
 \end{enumerate}
\end{theorem}

%%%%%%%%%%%%%%%
%%%%%%%%%%%%%%%

\subsection{Weak Product Structural Reflection}

We now introduce a weakening of $\PR$ and show that this principle produces  analogous characterizations of strongly unfoldable, $C^{(n)}$-strongly unfoldable, and subtle cardinals.

\begin{definition}\label{wpsr}
 Given an infinite cardinal $\kappa$ and a class $\Ce$ of structures of the same type, we let $\WPR_\Ce(\kappa)$ denote the  statement that  $\Ce\cap V_\kappa\neq\emptyset$ and for every substructure $X$ of $\prod(\Ce\cap V_\kappa)$ of cardinality at most  $\kappa$ and every $B\in\Ce$, there exists a homomorphism from $X$ to $B$.  
\end{definition}

Note that, if $X$ is a substructure of $\prod(\Ce\cap V_\kappa)$, then the identity map $\map{\id}{X}{\prod(\Ce\cap V_\kappa)}$ is a homomorphism. In particular, for every  class $\Ce$ and every cardinal $\kappa$, the principle $\PR_\Ce(\kappa)$ implies the principle $\WPR_\Ce(\kappa)$. 
In Section \ref{section7}, we will prove the following theorem that yields another canonical characterization of strong unfoldability:

\begin{theorem}\label{theorem:CharWPR}
 The following statements are equivalent for every  cardinal $\kappa$: 
 \begin{enumerate}
   \item\label{item:EiterSunfOrStrong}  $\kappa$ is either strongly unfoldable or a limit of strong  cardinals. 
   
  \item\label{item:BothWeakPSR} The principle $\WPR_\Ce(\kappa)$ holds for every class $\Ce$ of structures of the same type that is definable by a $\Sigma_2$-formula with parameters in $V_\kappa$. 
 \end{enumerate}
\end{theorem}

The fact that strong cardinals are strongly unfoldable now yields  the following corollary:

\begin{corollary}\label{corollary:LeastSUNF-WPR}
 The following statements are equivalent for every cardinal $\kappa$:
 \begin{enumerate}
  \item $\kappa$ is the least strongly unfoldable cardinal. 

  \item $\kappa$ is the least cardinal with the property that the principle $\WPR_\Ce(\kappa)$  holds for every class $\Ce$ of structures of the same type that is definable by a $\Sigma_2$-formula with parameters in $V_\kappa$. \qed 
 \end{enumerate} 
\end{corollary}

The  generalization of the last theorem to $C^{(n)}$-strongly unfoldable  cardinals is given by the following  theorem, which shall  be proved in Section \ref{section10}:

\begin{theorem}\label{WPSRC(n)}
  Given a natural number $n>1$, the following statements are equivalent for every cardinal $\kappa$:
  \begin{enumerate}
   \item $\kappa$ is either $C^{(n)}$-strongly unfoldable or a limit of $\Sigma_{n+1}$-strong cardinals. 
   
   \item The principle $\WPR_\Ce(\kappa)$ holds for every class $\Ce$ of structures of the same type that is definable by a $\Sigma_{n+1}$-formula with parameters in $V_\kappa$.
  \end{enumerate}
\end{theorem}

The next result shows that the above pattern also holds for the class-version of $\WPR$. The proof will be given in Section \ref{section10}.

\begin{theorem}\label{theorem:Ordsubtle2}
 The following schemes of axioms are equivalent over $\ZFC$: 
 \begin{enumerate}
  \item $\On$ is essentially subtle. 
  
  \item  For every non-empty class $\Ce$ of structures of the same type, there exists a cardinal $\kappa$ such that $\WPR_\Ce(\kappa)$ holds. 
 \end{enumerate}
\end{theorem}

Finally, we also have the analog characterization of subtle cardinals  in terms of weak product $\SR$. The proof of this result is contained  in Section \ref{section8}.

\begin{theorem}\label{main:ProdSubtle}
 The following statements are equivalent for every uncountable cardinal $\delta$: 
 \begin{enumerate}
  \item $\delta$ is a subtle cardinal. 
  
  \item For every set $\Ce$ of structures of the same type with $\Ce\subseteq V_\delta$, there exists a cardinal $\kappa<\delta$ with the property that the principle $\WPR_\Ce(\kappa)$ holds. 
 \end{enumerate}
\end{theorem}

%%%%%%%%%%%%%%%%%%%%%%%%%%
%%%%%%%%%%%%%%%%%%%%%%%%%%

\section{Vop\v{e}nka cardinals}\label{section2}

In this section, we present the proof of Theorem \ref{main:Vopenka}.  Recall that an inaccessible cardinal $\delta$ is a \emph{Vop\v{e}nka cardinal} if for every set $\Ce\in V_{\delta+1}\setminus V_\delta$ of structures of the same type, there exist distinct $A,B\in\Ce$ with an elementary embedding from $A$ to $B$.

\begin{lemma}\label{lemma:VopenkeSR}
 If $\delta$ is a Vop\v{e}nka cardinal and $\Ce\subseteq V_\delta$  is a set of structures of the same type, then there exists a cardinal $\kappa<\delta$ with the property that $\HSR_\Ce(\kappa)$ (and hence also $\SR_\Ce(\kappa)$ and  $\VSR_\Ce(\kappa)$) holds. 
\end{lemma}

\begin{proof}
 %With the help of {\cite[Theorem 5.4]{MR3324126}} and {\cite[Theorem 5.10]{MR3324126}}, we can find  a cardinal $\kappa<\delta$ with the property that for every $\gamma<\delta$, there exists an inner model $M$ with ${}^\gamma M\subseteq M$ and an elementary embedding $\map{j}{V}{M}$ with $\crit{j}=\kappa$, $j(\kappa)>\gamma$ and $\Ce\cap V_\gamma=j(\Ce)\cap V_\gamma$. 
 %
 Assume that $\delta$ is a Vop\v{e}nka cardinal and $\Ce\subseteq V_\delta$ is a set of structures of the same type. By \cite[Exercise 24.19]{MR1994835}, there is  a cardinal $\kappa<\delta$   that is \emph{$\eta$-extendible for $\Ce$}, for every $\eta<\delta$, i.e., for every ordinal $\kappa<\eta<\delta$, there is an ordinal $\zeta$ and an elementary embedding
   $$\map{j}{\langle V_\eta ,\in, \Ce \cap V_\eta\rangle}{\langle V_\zeta,\in ,\Ce\cap V_\zeta \rangle}$$
   with $\crit{j}=\kappa$ and $j(\kappa)>\eta$. 
   Now, fix a structure $B$ in $\Ce$ and let $\kappa<\eta<\delta$ be a limit ordinal such that $B\in V_\eta$. Since $B$ is an element of $\Ce\cap V_{j(\kappa)}$ and $j$ induces an elementary embedding of $B$ into $j(B)$, elementarity yields an element $A$ of $\Ce\cap V_\kappa$ and an elementary embedding of $A$ into $B$. Moreover, since $\kappa$ is inaccessible, we know that $A$ is contained in $H_\kappa$. These computations show that $\HSR_\Ce(\kappa)$ holds.   
\end{proof}

\begin{lemma}\label{lemma:BadC1singular}
 If $\delta$ is a singular cardinal, then there exists a class $\Ce$ of structures of the same type with $\Ce\subseteq V_\delta$ and the property that the principles $\SR_\Ce(\kappa)$ and $\VSR_\Ce(\kappa)$ fail for every cardinal $\kappa<\delta$. 
\end{lemma}

\begin{proof}
  Fix a strictly increasing sequence $\seq{\delta_\xi}{\xi<\cof{\delta}}$ of cardinals greater than $\cof{\delta}$ that is cofinal in $\delta$. 
  Define $\calL$ to be the first-order language that extends $\calL_\in$ by two constant symbols. 
 Given $\xi<\cof{\delta}$, define $A^\xi$ to be the $\calL$-structure $\langle V_{\delta_\xi},\in,\cof{\delta},\xi\rangle$. Set  $\Ce=\Set{A^\xi}{\xi<\cof{\delta}}\subseteq V_\delta$.

 \begin{claim*}
  If $\xi,\zeta<\cof{\delta}$ have the property that there exists an elementary embedding from $A^\xi$ to $A^\zeta$, then $\xi=\zeta$. 
 \end{claim*}
 
 \begin{proof}[Proof of the Claim]
  Let $\map{j}{V_{\delta_\xi}}{V_{\delta_\zeta}}$ be an elementary embedding with $j(\cof{\delta})=\cof{\delta}$ and $j(\xi)=\zeta$. 
  Then $j(V_{\cof{\delta}+2})=V_{\cof{\delta}+2}$ and the map $\map{j\restriction V_{\cof{\delta}+2}}{V_{\cof{\delta}+2}}{V_{\cof{\delta}+2}}$ is an elementary embedding. 
  The \emph{Kunen Inconsistency} then ensures that $j\restriction V_{\cof{\delta}+2}=\id_{V_{\cof{\delta}+2}}$ and we can conclude that $\xi=j(\xi)=\zeta$. 
 \end{proof}
 
 The above claim directly shows that the principle $\SR_\Ce(\kappa)$, and also $\VSR_\Ce(\kappa)$,  fail for every cardinal $\kappa<\delta$. 
 \end{proof}

\begin{proposition}\label{proposition:Sr--C1}
 If $\delta$ is an uncountable cardinal that is not a limit of inaccessible cardinals, then there exists a class $\Ce$ of structures of the same type, which is $\Pi_1$-definable in  $V_\delta$ with an ordinal as a parameter, and such that the principles $\SR_\Ce(\kappa)$ and  $\VSR_\Ce(\kappa)$ fail for every cardinal $\kappa < \delta$. 
\end{proposition}

\begin{proof}
 Fix an ordinal $\lambda<\delta$ with the property that  there are no inaccessible cardinals between $\lambda$ and $\delta$. 
 Let $\calL$ denote the first-order language that extends $\calL_\in$  by a constant symbol and define $\Ce$ to be the set of all $\calL$-structures of the form $\langle V_{\gamma+2},\in,\lambda\rangle$ with $\lambda<\gamma<\delta$.
 Then $\emptyset\neq\Ce\subseteq V_\delta$ and $\Ce$ is definable in $V_\delta$ by a $\Pi_1$-formula with parameter $\lambda$. 
 Assume, towards a contradiction, that there is a cardinal $\kappa<\delta$ with the property that either  $\SR_\Ce(\kappa)$ or $\VSR_\Ce(\kappa)$ holds. Then there is an ordinal $\lambda<\gamma<\kappa$ with the property that there exists an elementary embedding $\map{j}{V_{\gamma+2}}{V_{\kappa+2}}$ with $j(\lambda)=\lambda$. Since elementarity ensures that $j(\gamma)=\kappa$, we know that $j$ is non-trivial. Moreover, we have $j(V_{\lambda+2})=V_{\lambda+2}$ and the \emph{Kunen Inconsistency} ensure that the elementary embedding $\map{j\restriction V_{\lambda+2}}{V_{\lambda+2}}{V_{\lambda+2}}$ is trivial. This allows us to conclude that $\crit{j}$ is an inaccessible cardinal greater than $\lambda$, a contradiction.   
\end{proof}

\begin{proof}[Proof of Theorem \ref{main:Vopenka}]
    The implication $\eqref{item:Vopenka1}\Rightarrow \eqref{item:Vopenka2}$ is given by Lemma \ref{lemma:VopenkeSR}. 
    In order to prove the implication $\eqref{item:Vopenka2}\Rightarrow \eqref{item:Vopenka1}$, first notice that  Lemma \ref{lemma:BadC1singular} and Proposition \ref{proposition:Sr--C1}  show that $\eqref{item:Vopenka2}$ implies that  $\delta$ is inaccessible. The implication $\eqref{item:Vopenka2}\Rightarrow \eqref{item:Vopenka1}$ then directly follows from the simple observation that for every set $\Ce\in V_{\delta+1}\setminus V_\delta$ of structures of the same type and every cardinal $\kappa<\delta$, both $\SR_\Ce(\kappa)$ and $\VSR_\Ce(\kappa)$ imply that there exist distinct $A,B\in\Ce$ with an elementary embedding from $A$ to $B$.  
 \end{proof}

%%%%%%%%%%%%%%%%%%%%%%%
%%%%%%%%%%%%%%%%%%%%%%%

\section{Strongly unfoldable reflection}\label{section3}

We shall next demonstrate the connection between  strong unfoldability and the validity of the principle $\WSR$ for $\Sigma_2$-definable classes by proving Theorem \ref{theorem:CharSr-Sr--}. 
The starting point  is the following characterization of the elements of the class $C^{(n)}$-cardinals through the principle $\mathrm{SR}^{--}$.

\begin{lemma}\label{lemma:Sigma2Reflex}
 Given a natural number $n>1$, the following statements are equivalent for every cardinal $\kappa$: 
 \begin{enumerate}
     \item\label{item:Cn} The cardinal $\kappa$ is an element of $C^{(n)}$. 
     
      \item\label{item:SRminus} The principle $\SR^{--}_\Ce(\kappa)$ holds for every non-empty class $\Ce$ of structures of the same type that is definable by a $\Sigma_n$-formula with parameters in $V_\kappa$. 
     
     \item\label{item:ElementsInVkappa} Every non-empty class of structures of the same type that is definable by a $\Sigma_n$-formula with parameters in $H_\kappa$ contains an element of $V_\kappa$. 
     
    \item\label{item:ElementsInCardLesskappa}  Every non-empty class of structures of the same type that is definable by a $\Sigma_n$-formula with parameters in $H_\kappa$ contains a structure of cardinality less than $\kappa$. 
     
     \item\label{item:ElementsInHkappa} Every non-empty class of structures of the same type that is definable by a $\Sigma_n$-formula with parameters in $V_\kappa$ contains an element of $H_\kappa$. 
      \item\label{item:ElementsInVkappa2} Every non-empty class of structures of the same type that is definable by a $\Sigma_n$-formula with parameters in $V_\kappa$ contains an element of $V_\kappa$. 
 \end{enumerate}
\end{lemma}

\begin{proof}
 First, assume that \eqref{item:Cn} holds. 
 Fix a $\Sigma_n$-formula $\varphi(v_0,v_1)$,  and $a\in V_\kappa$ with the property that the class $\Ce=\Set{A}{\varphi(A,a)}$ is non-empty and consists of structures of the same type. 
 Then  the $\Sigma_n$-statement $\exists v_0   \varphi(v_0,a)$ holds in $V$ and our assumptions imply that it also holds in $V_\kappa$. Hence, there is $A\in V_\kappa$ with the property that $\varphi(A,a)$ holds in $V_\kappa$. 
 Our assumption now implies that $\varphi(A,a)$ holds in $V$. But our assumption also ensures that  $H_\kappa=V_\kappa$ and hence we know that $\Ce\cap H_\kappa\neq\emptyset$.  
 This shows  that \eqref{item:SRminus},\eqref{item:ElementsInVkappa}, \eqref{item:ElementsInCardLesskappa}, \eqref{item:ElementsInHkappa} and \eqref{item:ElementsInVkappa2} all hold in this case.

 Next, assume that either \eqref{item:SRminus}, or \eqref{item:ElementsInVkappa}, or \eqref{item:ElementsInCardLesskappa}, or \eqref{item:ElementsInHkappa}, or \eqref{item:ElementsInVkappa2} holds, and we shall prove \eqref{item:Cn}.

  \begin{claim*}
  If $\rho$ is an ordinal with the property that the set $\{\rho\}$ is definable by a $\Sigma_n$-formula with parameters in $H_\kappa$, then $\rho<\kappa$.
 \end{claim*}
 
 \begin{proof}[Proof of the Claim]
  Let $\calL$ denote the trivial first-order language and let $B$ be the unique $\calL$-structure with domain $\rho$. Then the class $\Ce=\{B\}$ is definable by a $\Sigma_n$-formula with parameters in $H_\kappa$ and hence each of our assumptions allows us to conclude that   $\rho<\kappa$. 
 \end{proof}

 The above claim  implies that $\kappa$ is a limit point of $C^{(n-1)}$, because for every $\alpha<\kappa$, the least cardinal in $C^{(n-1)}$ greater than $\alpha$ is definable by a $\Sigma_n$-formula with parameter $\alpha$. %, and so $\kappa$ is a limit of cardinals in $C^{(n-1)}$. 
 To show \eqref{item:Cn}, fix a $\Sigma_n$-formula $\exists x \psi(x,y)$, with $\psi$ being $\Pi_{n-1}$, and some $a\in V_\kappa$ with the property that $\exists x \psi(x,a)$ holds. 
  Let $\rho$ denote the least element of $C^{(n-1)}$ such that $a\in V_\rho$ and for some $b\in V_\rho$,  $\psi(b,a)$ holds in $V_\rho$. 
  Then the set $\{\rho\}$ is definable by a $\Sigma_n$-formula with parameter $a\in H_\kappa$ and the previous claim shows that $\rho<\kappa$. 
  Since $\kappa\in C^{(n-1)}$ and $\Pi_{n-1}$-formulas are upwards absolute between $V_\rho$ and $V_\kappa$, we can now conclude that $\psi(b,a)$ holds in $V_\kappa$, for some $b$. This shows that    \eqref{item:Cn} holds. 
\end{proof}

Let us also observe that Clause \eqref{item:Cn} of the above lemma  is also equivalent to each of \eqref{item:SRminus}, \eqref{item:ElementsInVkappa}, \eqref{item:ElementsInCardLesskappa}, \eqref{item:ElementsInHkappa}, \eqref{item:ElementsInVkappa2}, restricted to $\Sigma_n$-definable classes $\Ce$ that are closed under isomorphic images. The reason is that the above claim also holds under this restriction by taking $\Ce$ in the proof to be the class of all $\mathcal{L}$-structures  of cardinality $\rho$, which is closed under isomorphic images, instead of the singleton $\{B\}$. 
%
%Using this fact, the following proposition shows that, when restricted to classes $\Ce$ that are $\Sigma_n$-definable with $n>1$ and closed under isomorphic images, the principle $\WSR_\Ce(\kappa)$ is equivalent to $\VSR^-_\Ce(\kappa)$ (see Definition \ref{defminus}).

 We are now ready to prove the equivalences stated in Theorem \ref{theorem:CharSr-Sr--}. The following arguments rely on Magidor's characterization of supercompactness in \cite{MR295904} and the analysis of strong unfoldability provided by the results of \cite{luecke2021strong} and \cite{SRminus}.

\begin{proof}[Proof of Theorem \ref{theorem:CharSr-Sr--}]
\label{proof1.7}
  First, note that Lemma \ref{lemma:Sigma2Reflex} directly yields the implication $\eqref{item:BothWeakHSR} \Rightarrow \eqref{item:BothWeakSR}$. 
 
  Next, assume that $\kappa$ is a strongly unfoldable cardinal. 
 Then $\kappa$ is an element of $C^{(2)}$ (see {\cite[Section 3]{luecke2021strong}} or Proposition \ref{proposition:CnSUNF-Cn+1} below). 
 Fix  a class $\Ce$ of structures of the same type that is definable by a $\Sigma_2$-formula with parameters in $V_\kappa$ and a structure $B$ in $\Ce$ of cardinality $\kappa$.
 Pick a $\Sigma_2$-formula $\varphi(v_0,v_1)$ and an element $z$ of $H_\kappa$  with $\Ce=\Set{A}{\varphi (A, z)}$, and let $\theta >\kappa$ be a cardinal with the property that $B \in H_\theta$ and $\varphi (B, z)$ holds in $H_\theta$. 
 Using {\cite[Theorem 1.3]{luecke2021strong}} and {\cite[Lemma 2.1]{SRminus}} (see also Theorem \ref{theorem:SigmaNsunf} below), we can find cardinals $\bar{\kappa}<\bar{\theta}<\kappa$, an
elementary submodel $X$ of $H_{\bar{\theta}}$ with $\bar{\kappa} +1\subseteq X$ and an elementary
embedding $\map{j}{X}{H_\theta}$ with $j\restriction \bar{\kappa} = \id_{\bar{\kappa}}$, $j(\bar{\kappa}) = \kappa$, and $B, z \in  \ran{j}$. 
 We then have $j \restriction (H_{\bar{\kappa}} \cap X) = id_{H_{\bar{\kappa}}\cap X}$, and therefore $z \in H_{\bar{\kappa}}$ 
and $j(z) = z$. Pick $A \in X$ with $j(A) = B$. Then elementarity and $\Sigma_1$-absoluteness implies that $\varphi(A, z)$ holds and hence $A$ is an element of $\Ce$. Since  $A$ belongs to $X$, and therefore to  $H_{\bar{\theta}}$, it also belongs to $H_\kappa$. Moreover,  since  $\bar{\kappa}$ is a subset of $X$, and $A$ has cardinality $\bar{\kappa}$ in $X$,  the restriction of $j$ to $A$  yields an elementary embedding of $A$ into $B$. This shows that $\HSR^-_\Ce(\kappa)$ holds in this case.  

 Now, assume that $\kappa$ is a limit of supercompact cardinals. Then $\kappa$ is an element of $C^{(2)}$. Moreover, Theorem \ref{theorem:JoanBoldSupercompactSigma2} shows that $\HSR_\Ce(\kappa)$ holds for every  class $\Ce$ of structures of the same type that is definable by a $\Sigma_2$-formula with parameters in $V_\kappa$.
 In combination with the above computations, this yields the implication $\eqref{item:EiterSunfOrSC} \Rightarrow \eqref{item:BothWeakHSR}$.

Finally,  assume that  $\kappa$ is a cardinal that is not strongly unfoldable and the principle $\WSR_\Ce(\kappa)$ holds for every class $\Ce$ of structures of the same type that is definable by a $\Sigma_2$-formula with parameters in $V_\kappa$. 
By Lemma \ref{lemma:Sigma2Reflex}, we know that $\kappa$ is an element of $C^{(2)}$. In particular, $\kappa$ is a limit cardinal and the set $V_\kappa$ has cardinality $\kappa$.

 \begin{claim*}
  If $\theta>\kappa$ is a cardinal, $y\in V_\kappa$ and $z\in H_\theta$, then there are cardinals $\bar{\kappa}<\bar{\theta}<\kappa$ with $y\in V_{\bar{\kappa}}$, an elementary submodel $X$ of $H_{\bar{\theta}}$ with $V_{\bar{\kappa}}\cup\{\bar{\kappa}\}\subseteq X$ and an elementary embedding $\map{j}{X}{H_\theta}$ with $j(\bar{\kappa})=\kappa$, $j(y)=y$ and $z\in\ran{j}$. 
 \end{claim*}
 
 \begin{proof}[Proof of the Claim]
  Let $\calL$ denote the first-order language that extends $\calL_\in$ by three constant symbols and let $\Ce$ denote the class of all $\calL$-models of the form $\langle M,\in,\mu,a,y\rangle$ such that $\mu$ is a cardinal in $C^{(1)}$, $y\in V_\mu$ and there exists a cardinal $\nu>\mu$ and  an elementary submodel $X$ of $H_{\nu}$ with $V_\mu\cup\{\mu\}\subseteq X$ and the property that $M$ is the transitive collapse of $X$. 
  It is then easy to see that the class $\Ce$ is definable by a $\Sigma_2$-formula with parameter $y$. 
  Now, let $Y$ be an elementary submodel of $H_{\theta}$ of cardinality $\kappa$ with $V_\kappa\cup\{\kappa,z\}\subseteq Y$ and let $\map{\tau}{Y}{N}$ denote the corresponding transitive collapse. Then $\theta$ and $Y$ witness that $B=\langle N,\in,\kappa,\tau(z),y\rangle$ is an element of $\Ce$ of cardinality $\kappa$. 
  Our assumptions then yield cardinals $\bar{\kappa}<\bar{\theta}$ with $\bar{\kappa}\in C^{(1)}\cap\kappa$, an elementary submodel $X$ of $H_{\bar{\theta}}$ of cardinality less than $\kappa$ with $V_{\bar{\kappa}}\cup\{\bar{\kappa}\}\subseteq X$ and an elementary embedding $\map{i}{M}{N}$ with $i(\bar{\kappa})=\kappa$, $i(y)=y$ and $\tau(z)\in\ran{i}$, where $\map{\pi}{X}{M}$ denotes the corresponding transitive collapse.  
  Since $\kappa\in C^{(2)}$ and $M\in V_\kappa$, we may assume that $\bar{\theta}<\kappa$. Define $$\map{j ~ = ~ \tau^{{-}1}\circ i\circ\pi}{X}{H_{\theta}}.$$ Then $j$ is an elementary embedding with $j(\bar{\kappa})=\kappa$, $j(y)=y$ and $z\in\ran{j}$.  
 \end{proof}

 \begin{claim*}
  If $\theta>\kappa$ is a cardinal, $y\in V_\kappa$ and $z\in H_\theta$, then there are cardinals $\bar{\kappa}<\bar{\theta}<\kappa$ with $y\in V_{\bar{\kappa}}$, an elementary submodel $X$ of $H_{\bar{\theta}}$ with $V_{\bar{\kappa}}\cup\{\bar{\kappa}\}\subseteq X$ and an elementary embedding $\map{j}{X}{H_\theta}$ with $j(\bar{\kappa})=\kappa$, $j(y)=y$, $z\in\ran{j}$ and $j\restriction\bar{\kappa}\neq\id_{\bar{\kappa}}$. 
 \end{claim*}
 
 \begin{proof}[Proof of the Claim]
  %First, assume that $\kappa$ is singular. By the previous claim, there exist cardinals $\bar{\kappa}<\bar{\theta}<\kappa$ with $\cof{\kappa},y\in V_{\bar{\kappa}}$, an elementary submodel $X$ of $H_{\bar{\theta}}$ with $V_{\bar{\kappa}}\cup\{\bar{\kappa}\}\subseteq X$ and an elementary embedding $\map{j}{X}{H_\theta}$ with $j(\bar{\kappa})=\kappa$, $j(\cof{\kappa})=\cof{\kappa}$, $j(y)=y$ and $z\in\ran{j}$. 
  %
  %Then elementarity implies that $\cof{\kappa}=\cof{\bar{\kappa}}$ and $X$ contains a cofinal function $\map{c}{\cof{\kappa}}{\bar{\kappa}}$. 
  %
  %We then know that $\map{j(c)}{\cof{\kappa}}{\kappa}$ is cofinal and there exists $\xi<\cof{\kappa}$ with $j(c)(\xi)>\bar{\kappa}$. But then either $j(\xi)\neq \xi$ or $j(c(\xi))\neq c(\xi)$, because otherwise we would have $$j(c)(\xi) ~ = ~ j(c)(j(\xi)) ~ = ~ j(c(\xi)) ~ = ~ c(\xi) ~ < ~ \bar{\kappa}.$$ In particular, we know that $j\restriction\bar{\kappa}\neq\id_{\bar{\kappa}}$ in this case. 

  %Now, assume that $\kappa$ is regular. The fact that $\kappa$ is in $C^{(2)}$ implies that $\kappa^{{<}\kappa}=\kappa$. 
  %
  Since $\kappa$ is not strongly unfoldable, we can combine {\cite[Theorem 1.3]{luecke2021strong}} with {\cite[Lemma 2.1]{SRminus}} (see also Theorem \ref{theorem:SigmaNsunf}.\eqref{item:Magidor} below and the Remark \ref{remarkCnsf} that follows) to find a cardinal $\vartheta>\theta$ and $z'\in V_\vartheta$ such that for all cardinals $\bar{\kappa}<\bar{\vartheta}$ and all elementary submodels $X$ of $H_{\bar{\vartheta}}$ with $V_{\bar{\kappa}}\cup \{ \bar{\kappa}\}\subseteq X$, there is no elementary embedding $\map{j}{X}{H_\vartheta}$ with $j\restriction\bar{\kappa}=\id_{\bar{\kappa}}$,  $j(\bar{\kappa})=\kappa$, and $z, z', \theta \in \ran{j}$.
  %, because otherwise {\cite[Lemma 3.1]{SRminus}} would show that $\kappa$ is a weakly shrewd cardinal with $\kappa^{{<}\kappa}=\kappa$, a combination of {\cite[Proposition 3.4]{SRminus}} and  {\cite[Lemma 3.5]{SRminus}} would then show that $\kappa$ is  \emph{shrewd} (see {\cite[Definition 1.4]{SRminus}}) and this would allow us to apply {\cite[Theorem 1.3]{luecke2021strong}} to show that $\kappa$ is strongly unfoldable. 
 %
  An application of our previous claim now yields cardinals $\bar{\kappa}<\bar{\vartheta}<\kappa$ with $y\in V_{\bar{\kappa}}$, an elementary submodel $Y$ of $H_{\bar{\vartheta}}$ with $V_{\bar{\kappa}}\cup\{\bar{\kappa}\}\subseteq Y$ and an elementary embedding $\map{i}{Y}{H_\vartheta}$ with $i(\bar{\kappa})=\kappa$, $i(y)=y$ and $z,z',\theta\in\ran{i}$. 
 Therefore, we must have $i\restriction\bar{\kappa}\neq\id_{\bar{\kappa}}$. Pick $\bar{\theta}\in Y$ with $i(\bar{\theta})=\theta$. Then elementarity implies that $\bar{\theta}$ is a cardinal. Set $X=Y\cap H_{\bar{\theta}}$ and $j=i\restriction X$. 
 In this situation, we can conclude that $\bar{\kappa}<\bar{\theta}<\kappa$, $X$  is an elementary submodel of $H_{\bar{\theta}}$ with $V_{\bar{\kappa}}\cup\{\bar{\kappa}\}\subseteq X$ and $\map{j}{X}{H_\theta}$ is an elementary embedding  with $j(\bar{\kappa})=\kappa$, $j(y)=y$, $z\in\ran{j}$ and $j\restriction\bar{\kappa}\neq\id_{\bar{\kappa}}$. 
  \end{proof}

 \begin{claim*}
  There are unboundedly many cardinals below $\kappa$ that are $\alpha$-supercompact for every $\alpha<\kappa$. 
 \end{claim*}
 
 \begin{proof}[Proof of the Claim]
  Fix an uncountable regular cardinal $\rho<\kappa$ and a cardinal $\theta$ in $C^{(1)}$ above $\kappa$. By our previous claim, we can find cardinals $\rho<\bar{\kappa}<\bar{\theta}<\kappa$, an elementary submodel $X$ of $H_{\bar{\theta}}$ with $V_{\bar{\kappa}}\cup\{\bar{\kappa}\}\subseteq X$ and an elementary embedding $\map{i}{X}{H_\theta}$ with $i(\bar{\kappa})=\kappa$, $j(\rho)=\rho$ and $i\restriction\bar{\kappa}\neq\id_{\bar{\kappa}}$. 
  Set $\map{j=i\restriction V_{\bar{\kappa}}}{V_{\bar{\kappa}}}{V_\kappa}$. Our setup then ensures that $j$ is a non-trivial elementary embedding. Since the \emph{Kunen Inconsistency} implies that $i\restriction V_\rho=\id_\rho$, we know that $\crit{j}>\rho$. 
  Moreover, {\cite[Lemma 2]{MR295904}} directly shows that $\crit{j}$ is $\alpha$-supercompact for all $\alpha<\bar{\kappa}$. 
  Since $\theta$ is an element of $C^{(1)}$, we know that $\theta$ is a limit cardinal with $H_\theta=V_\theta$. 
  Therefore, elementarity implies that the Power Set Axiom holds in $H_{\bar{\theta}}$ and for every ordinal $\gamma<\bar{\theta}$, the set $H_{\bar{\theta}}\cap V_\gamma$ is an element of $H_{\bar{\theta}}$. Since $H_{\bar{\theta}}$ computes power sets correctly, it follows that $H_{\bar{\theta}}=V_{\bar{\theta}}$ and we now know that $\bar{\theta}$ is an element of $C^{(1)}$.
  This implies that $H_{\bar{\theta}}$ contains all ultrafilters witnessing that $\crit{j}$ is $\alpha$-supercompact for all $\alpha<\bar{\kappa}$ and hence it follows that, in $X$, the cardinal  $\crit{j}$ is $\alpha$-supercompact for all $\alpha<\bar{\kappa}$. 
  But this allows us to conclude that $j(\crit{j})$ is a cardinal in the interval $(\rho,\kappa)$ that is $\alpha$-supercompact for all $\alpha<\kappa$.  
 \end{proof}

 \begin{claim*}
  Every cardinal below $\kappa$ that is $\alpha$-supercompact for every $\alpha<\kappa$ is supercompact. 
 \end{claim*}
  
 \begin{proof}[Proof of the Claim]
  Let $\mu<\kappa$ be a cardinal that is $\alpha$-supercompact for all $\alpha<\kappa$, let $\lambda>\kappa$ be an ordinal and let $\theta>\lambda$ be an element of $C^{(1)}$. 
  By our first claim, there exist cardinals $\mu<\bar{\kappa}<\bar{\theta}<\kappa$, an elementary submodel $X$ of $H_{\bar{\theta}}$ with $V_{\bar{\kappa}}\cup\{\bar{\kappa}\}\subseteq X$ and an elementary embedding $\map{j}{X}{H_\theta}$ with $j(\bar{\kappa})=\kappa$, $j(\mu)=\mu$ and $\lambda\in\ran{j}$. 
  Pick $\bar{\lambda}\in X$ with $j(\bar{\lambda})=\lambda$. Then $\mu<\bar{\lambda}<\kappa$ and $\mu$ is $\bar{\lambda}$-supercompact. 
  Since elementarity ensures that $\bar{\theta}$ is an element of $C^{(1)}$, we now know that $\mu$ is $\bar{\lambda}$-supercompact in $X$ and this shows that $\mu$ is $\lambda$-supercompact. 
 \end{proof}

 The combination of the above claims now shows that $\kappa$ is a limit of supercompact cardinals in this case. 
 In particular, these arguments prove the implication $\eqref{item:BothWeakSR} \Rightarrow \eqref{item:EiterSunfOrSC}$. 
\end{proof}

 In addition to Corollary \ref{corollary:LeastSUNF}, Theorem \ref{theorem:CharSr-Sr--} can directly be used to derive several interesting equivalences. 
 For example, it shows that for cardinals that are not strongly unfoldable, the validity of the principle $\mathrm{SR}$ for $\Sigma_2$-definable classes is equivalent to the  validity of the principle $\WSR$, and also to $\HSR^-$, for these classes. In particular, this equivalence holds for all singular cardinals.

\begin{corollary}
 The following statements are equivalent for every cardinal $\kappa$ that is not strongly unfoldable: 
 \begin{enumerate}
  \item The cardinal $\kappa$ is a limit of supercompact cardinals. 
 
  %\item The principle $\SR_\Ce(\kappa)$ holds for every class $\Ce$ of structures of the same type that is definable by a $\Sigma_2$-formula with parameters in $V_\kappa$. 
  
  \item The principle $\WSR_\Ce(\kappa)$  holds for every class $\Ce$ of structures of the same type that is definable by a $\Sigma_2$-formula with parameters in $V_\kappa$. 
  
  \item The principle $\HSR_\Ce^-(\kappa)$  holds for every class $\Ce$ of structures of the same type that is definable by a $\Sigma_2$-formula with parameters in $V_\kappa$. 
   \qed 
 \end{enumerate}
\end{corollary}

%%%%%%%%%%%%%%%%%%%%%%%%%%
%%%%%%%%%%%%%%%%%%%%%%%%%%

\section{Subtle reflection}\label{section5}

We will next give a proof of Theorem \ref{theorem:SubtleChar}. 
%As a side product of the developed techniques, we will obtain a combinatorial characterization of cardinals that are either subtle or limits of subtle cardinals (see Corollary \ref{corollary:NewCharClSubtle} below). 
Recall that, given an ordinal $\delta$, a sequence $\seq{E_\gamma}{\gamma <\delta}$ is a \emph{$\delta$-list} if $E_\gamma\subseteq\gamma$ holds for every $\gamma <\delta$.

\begin{lemma}\label{lemma:SubtleSR-SR--}
 Let $\delta$ be an element of $C^{(1)}$ with the property that for every $\delta$-list $\seq{E_\gamma}{\gamma<\delta}$ and every $\rho<\delta$, there exist cardinals $\rho<\mu<\nu<\delta$ with $E_\mu\subseteq E_\nu$. 
 %is a subtle cardinal and 
 If $\Ce$ is a non-empty set of structures of the same type with $\Ce\subseteq V_\delta$, then there exists a cardinal $\kappa<\delta$ with the property that  $\WSR_\Ce(\kappa)$ holds. 
\end{lemma}

\begin{proof}
 Let $\calL$ denote the signature of $\Ce$ and let $\rho$ denote the minimal cardinality of structures in $\Ce$. Since $\delta\in C^{(1)}$, we know that $\rho<\delta$. 
 In addition, let $\mathsf{Sat}_\calL$ denote the formalized satisfaction relation for $\calL$ and let $\mathsf{Fml}_\calL$ denote the set of G\"odel numbers of formalized $\calL$-formulas. 
 Assume, towards a contradiction, that the principle $\SR^-_\Ce(\kappa)$ fails for every cardinal $\rho<\kappa<\delta$. 
 Given a cardinal $\rho<\kappa<\delta$, we can now fix a structure $A_\kappa\in\Ce$ of cardinality $\kappa$ with the property that there is no elementary embedding of a structure in $\Ce$ of cardinality less than $\kappa$ into $A_\kappa$. Let $b_\kappa$ be a bijection between $\kappa$ and the domain of $A_\kappa$. 
 Let $\seq{E_\gamma}{\gamma<\delta}$ be a $\delta$-list with the property that for every cardinal  $\rho<\kappa<\delta$, the set $E_\kappa$ consists of all ordinals of the form\footnote{Here, we let $\map{\goedel{\cdot}{\ldots,\cdot}}{\On^n}{\On}$ denote the \emph{iterated G\"odel pairing function}.} $$\goedel{\ell}{\alpha_0,\ldots,\alpha_{k-1}}$$ for some $\ell\in\mathsf{Fml}_\calL$ that codes a formula with $k$ free variables and $\alpha_0,\ldots,\alpha_{k-1}<\kappa$ with $$\mathsf{Sat}_\calL(A_\kappa,\ell,\langle b_\kappa(\alpha_0),\ldots,b_\kappa(\alpha_{k-1})\rangle).$$ 
 
 By our assumptions, there exist cardinals $\rho<\mu<\nu<\delta$ with the property that $E_\mu\subseteq E_\nu$. 
 
 \begin{claim*}
  The map $b_\nu\circ b_\mu^{{-}1}$ is an elementary embedding of $A_\mu$ into $A_\nu$. 
 \end{claim*}
 
 \begin{proof}[Proof of the Claim]
  Fix an $\calL$-formula $\varphi(v_0,\ldots,v_{k-1})$ and $\alpha_0,\ldots,\alpha_{k-1}<\mu$ such that $$A_\mu\models\varphi(b_\mu(\alpha_0),\ldots,b_\mu(\alpha_{k-1})).$$ If $\ell_\varphi$ is the canonical element of $\mathsf{Fml}_\calL$ corresponding to $\varphi$, then we have $$\goedel{\ell_\varphi}{\alpha_0,\ldots,\alpha_{k-1}} ~ \in ~ E_\mu ~ \subseteq ~ E_\nu$$ and this implies that $$A_\nu\models\varphi(b_\nu(\alpha_0),\ldots,b_\nu(\alpha_{k-1})).$$ By also considering negated formulas, the derived  implication yields the statement of the claim. 
 \end{proof}

 Since the above claim yields a contradiction, we now know that the principle $\SR^-_\Ce(\kappa)$ holds for some cardinal $\rho<\kappa<\delta$. Moreover, our setup also ensures that $\SR_\Ce^{--}(\kappa)$ holds. 
\end{proof}

The above lemma directly yields a proof of the forward implication of Theorem \ref{theorem:SubtleChar}.

\begin{corollary}\label{corollary:SubtleYieldsSR-SR--}
 If $\delta$ is either a subtle cardinal or a limit of subtle cardinals and $\Ce$ is a non-empty class of structures of the same type with $\Ce\subseteq V_\delta$, then there exists a cardinal $\kappa<\delta$ with the property that  $\WSR_\Ce(\kappa)$ holds. 
\end{corollary}

\begin{proof}
 Let $\seq{E_\gamma}{\gamma<\delta}$ be a $\delta$-list and let $\rho$ be an ordinal smaller than $\delta$. 
 If $\delta$ is a subtle cardinal, then we can consider the closed unbounded subset of all cardinals in the interval $(\rho,\delta)$ and find elements $\mu<\nu$ of this set with the property that $E_\mu=E_\nu\cap\mu$. 
 In the other case, if $\delta$ is a limit of subtle cardinals, then we can fix a subtle cardinal $\delta_0$ with $\rho<\delta_0<\delta$ and repeat the above argument with the $\delta_0$-list $\seq{E_\gamma}{\gamma<\delta_0}$ to find cardinals $\mu<\nu$ with $\rho<\mu<\nu<\delta_0$ and $E_\mu=E_\nu\cap\mu$.  
 Since our assumptions imply that $\delta$ is an element of $C^{(1)}$, these computations allow us to apply Lemma \ref{lemma:SubtleSR-SR--} to derive the desired conclusion. 
\end{proof}

\begin{lemma}\label{lemma:SubtleFromSR-SR--}
 Let $\rho$ be an ordinal with the property that there exists a cardinal $\varepsilon$ greater than $\rho$ such that for every $A\subseteq V_\varepsilon$ and every sufficiently large cardinal $\theta$, there exist cardinals $\rho<\bar{\kappa}<\kappa<\varepsilon$, an elementary submodel $X$ of $H_\theta$ with $\bar{\kappa}\cup\{\bar{\kappa},A\}\subseteq X$ and an elementary embedding $\map{j}{X}{H_\theta}$ with $j\restriction\bar{\kappa}=\id_{\bar{\kappa}}$, $j(\bar{\kappa})=\kappa$ and $j(A)=A$. 
 If $\delta$ is the least cardinal greater than $\rho$ with this property, then $\delta$ is a subtle cardinal. 
\end{lemma}

\begin{proof}
 We start our proof with the following claim. 
 
 \begin{claim*}
  $\delta$ is a limit cardinal. 
 \end{claim*}
 
 \begin{proof}[Proof of the Claim]
  First, note that our assumptions imply that $\delta$ is an uncountable cardinal. Next, assume towards a contradiction that $\delta=\gamma^+$ for some infinite cardinal $\gamma$. Let $A$ be a subset of $V_\gamma$ and let $\theta$ be a cardinal that is sufficiently large with respect to $\delta$ and the above property. 
  Then we can find cardinals $\rho<\bar{\kappa}<\kappa<\delta$, an elementary submodel $X$ of $H_\theta$ with $\bar{\kappa}\cup\{\gamma,\delta,\bar{\kappa},A\}\subseteq X$ and an elementary embedding $\map{j}{X}{H_\theta}$ with $j\restriction\bar{\kappa}=\id_{\bar{\kappa}}$, $j(\bar{\kappa})=\kappa$, $j(\gamma)=\gamma$, $j(\delta)=\delta$ and $j(A)=A$. 
 Since $\kappa\leq\gamma$, $j(\gamma)=\gamma$ and $j(\bar{\kappa})=\kappa$, we know that $\kappa<\gamma$. 
 But this shows that $\gamma$ also possesses the relevant property, contradicting the minimality of $\delta$.  
 \end{proof}
 
 Let $\vec{E}=\seq{E_\gamma}{\gamma<\delta}$ be a $\delta$-list and let $C$ be a closed unbounded subset of $\delta$ that consists of cardinals. 
 Now, pick a well-ordering $\lhd$ of $V_\delta$ and a cardinal $\theta>\beth_\delta$ such that $\theta$ is sufficiently large and $H_\theta$ is sufficiently elementary in $V$. 
 In this situation, our assumptions yields cardinals $\rho<\bar{\kappa}<\kappa<\delta$, an elementary submodel $X$ of $H_\theta$ with $\bar{\kappa}\cup\{\bar{\kappa},C,\vec{E},\lhd\}\subseteq X$ and an elementary embedding $\map{j}{X}{H_\theta}$ with $j\restriction\bar{\kappa}=\id_{\bar{\kappa}}$, $j(\bar{\kappa})=\kappa$, $j(C)=C$, $j(\vec{E})=\vec{E}$ and $j(\lhd)=\lhd$.

 \begin{claim*}
  $\bar{\kappa}\in C$. 
 \end{claim*}
 
 \begin{proof}[Proof of the Claim]
  Assume, towards a contradiction, that $\bar{\kappa}$ is not an element of $C$. 
  Define $$\gamma ~ = ~ \min(C\setminus\bar{\kappa}) ~ \in ~ X\cap(\bar{\kappa},\delta).$$ 
  If $C\cap\bar{\kappa}=\emptyset$, then $\gamma=\min(C)$ and the ordinal $\gamma$ is definable in $H_\theta$ by a formula with parameter $C$. 
  In the other case, if $C\cap\bar{\kappa}\neq\emptyset$, then $\sup(C\cap\bar{\kappa})<\bar{\kappa}$, $\gamma=\min(C\setminus(\sup(C\cap\bar{\kappa})+1))$ and therefore the ordinal $\gamma$ is definable in $H_\theta$ by a formula with parameters in $\bar{\kappa}\cup\{C\}$. 
  In both cases, the ordinal $\gamma$ is definable in $H_\theta$ by a formula whose parameters are elements of $X$ and that are fixed by the embedding $j$. 
  This shows that $\gamma$ is a cardinal greater than $\bar{\kappa}$ that is  fixed by $j$. Since $j(\bar{\kappa})=\kappa$, this also shows that $\gamma$ is bigger than $\kappa$.

  Using the minimality of $\delta$, we can now find a cardinal $\vartheta>\beth_\gamma$ and a subset $A$ of $V_\gamma$, with the property that for all cardinals $\rho<\bar{\mu}<\mu<\gamma$ and every elementary submodel $Y$ of $H_\vartheta$ with $\bar{\mu}\cup\{\bar{\mu},A\}\subseteq Y$, there is no elementary embedding $\map{i}{Y}{H_\vartheta}$ with $i\restriction\bar{\mu}=\id_{\bar{\mu}}$, $i(\bar{\mu})=\mu$ and $j(A)=A$. 
  Let $\vartheta$ denote the least cardinal greater than $\beth_\gamma$ such that there exists a subset of $V_\gamma$ with these properties and let $A$ denote the $\lhd$-least subset of $V_\gamma$ witnessing this statement with respect to $\gamma$. 
  Since $H_\theta$ was chosen sufficiently elementary in $V$, it follows that  $\vartheta$ and $A$ are elements of $H_\theta$ and both sets are definable in $H_\theta$ from the parameters $\gamma$ and $\lhd$. 
  But this implies that $\vartheta,A\in X$, $j(\vartheta)=\vartheta$ and $j(A)=A$. 
  Set $Y=H_\vartheta\cap X$ and $\map{i=j\restriction Y}{Y}{H_\vartheta}$. 
  Since our assumptions on $\theta$ ensure that $H_\vartheta\in X$, it follows that $Y$ is an elementary submodel of $H_\vartheta$ with $\bar{\kappa}\cup\{\bar{\kappa},A\}\subseteq Y$ and $i$ is an elementary embedding with $i\restriction\bar{\kappa}=\id_{\bar{\kappa}}$, $i(\bar{\kappa})=\kappa$ and $j(A)=A$, contradicting the properties of $\vartheta$ and $A$.  
 \end{proof}
 
 By elementarity, the above claim directly implies that $\kappa$ is an element of $C$. Moreover, our setup ensures that $E_{\bar{\kappa}}=j(E_{\bar{\kappa}})\cap\bar{\kappa}=E_\kappa\cap\bar{\kappa}$. 
 
 The above computations show that $\delta$ is a limit cardinal with the property that for every $\delta$-list $\seq{E_\gamma}{\gamma<\delta}$ and every closed unbounded set $C$ in $\delta$ that consists of cardinals, there are $\mu<\nu$ in $C$ with $E_\mu=E_\nu\cap\nu$. This directly implies that $\delta$ has uncountable cofinality and we can conclude that for every $\delta$-list $\seq{E_\gamma}{\gamma<\delta}$ and every closed unbounded set $C$ in $\delta$, there are $\mu<\nu$ in $C$ with $E_\mu=E_\nu\cap\nu$. 
\end{proof}

\begin{proposition}\label{proposition:Sr--C1_part2}
 If $\delta$ is an uncountable cardinal that is not an element of $C^{(1)}$, then there exists a class $\Ce$ of structures of the same type with $\Ce\subseteq V_\delta$ and the property that the principle $\SR^{--}_\Ce(\kappa)$ fails for every cardinal $\kappa\leq\delta$. 
\end{proposition}

\begin{proof}
 Our assumptions on $\delta$ directly yield an ordinal $\beta<\delta$ with the property that the set $V_\beta$ has cardinality greater than $\delta$. 
 Let $\Ce$ denote the set of all $\calL_\in$-structures of the form $\langle V_\gamma,\in\rangle$ with $\beta\leq\gamma<\delta$. Then $\emptyset\neq\Ce\subseteq V_\delta$ and $\SR^{--}_\Ce(\kappa)$ fails for every cardinal $\kappa\leq\delta$, because $\Ce$ contains no structures of cardinality less than or equal to $\delta$. 
\end{proof}

\begin{proof}[Proof of Theorem \ref{theorem:SubtleChar}]
 Let $\delta$ be an uncountable cardinal with the property that for every set $\Ce$ of structures of the same type with $\Ce\subseteq V_\delta$, there exists a cardinal $\kappa<\delta$ such that  the principle $\WSR_\Ce(\kappa)$ holds. 
 Then Proposition \ref{proposition:Sr--C1_part2} shows that $\delta$ is an element of $C^{(1)}$. 
 Assume, towards a contradiction, that $\kappa$ is neither a subtle cardinal nor a limit of subtle cardinals. 
 Pick an uncountable regular cardinal $\rho<\delta$ with the property that the interval $(\rho,{\delta}]$ contains no subtle cardinals. 
 
\begin{claim*}
  If $A\subseteq V_\delta$ and $\theta>\beth_\delta$ is a cardinal, then there exist cardinals $\rho<\bar{\nu}<\nu<\delta$, an elementary submodel $X$ of $H_\theta$ with $\bar{\nu}\cup\{\bar{\nu},A\}\subseteq X$ and an elementary embedding $\map{j}{X}{H_\theta}$ with $j\restriction\bar{\nu}=\id_{\bar{\nu}}$, $j(\bar{\nu})=\nu$ and $j(A)=A$.  
\end{claim*}

 \begin{proof}[Proof of the Claim]
  Let $\calL$ denote the first order-language that extends $\calL_\in$ by three constant symbols and let $\Ce$ denote the set of all $\calL$-structures of the form $\langle M,\in,\rho,\mu,B\rangle$ such that $\mu$ is a cardinal strictly between $\beth_\rho$ and $\delta$,  and there exists an elementary submodel $X$ of $H_\theta$ of cardinality $\mu$ with the property that   $V_\rho\cup\mu\cup\{\mu,A\}\subseteq X$, $M$ is the transitive collapse of $X$ and $\pi(A)=B$, where $\map{\pi}{X}{M}$ denotes the corresponding  collapsing map. 
 We then have $\Ce\subseteq V_\delta$ and,  
  by our assumptions on $\delta$, there exists a cardinal $\kappa<\delta$ with the property that $\WSR_\Ce(\kappa)$ holds. 
  Since every structure in $\Ce$ has cardinality greater than $\beth_\rho$, we know that $\kappa>\rho$. 
  
  Now, pick an elementary submodel $Y$ of $H_\theta$ of cardinality $\kappa$ with $V_\rho\cup\kappa\cup\{\kappa,A\}\subseteq Y$ and let $\map{\pi}{Y}{N}$ denote the corresponding transitive collapse. Then $\langle N,\in,\rho,\kappa,\pi(A)\rangle$ is a structure in $\Ce$ of cardinality $\kappa$. 
  By our assumptions, we can now find a cardinal $\bar{\kappa}$ strictly between  $\beth_\rho$ and $\kappa$, an elementary submodel $X$ of $H_\theta$ of cardinality $\bar{\kappa}$ with $V_\rho\cup\bar{\kappa}\cup\{\bar{\kappa},A\}\subseteq X$ and  an elementary embedding $i$ of $\langle M,\in,\rho,\bar{\kappa},\bar{\pi}(A)\rangle$ into  $\langle N,\in,\rho,\kappa,\pi(A)\rangle$, where $\map{\bar{\pi}}{X}{M}$ denotes the corresponding transitive collapse. 
  We define $$\map{j ~ = ~ \pi^{{-}1}\circ i\circ\bar{\pi}}{X}{H_\theta}.$$ Then $j$ is an elementary embedding with $j(\rho)=\rho$, $j(\bar{\kappa})=\kappa$ and $j(A)=A$. Moreover, since $\rho$ is an uncountable regular cardinal,  $V_\rho\subseteq X$ and $j(V_\rho)=V_\rho$, the \emph{Kunen Inconsistency} implies that $j\restriction V_\rho=\id_\rho$. 
  Let $\bar{\nu}$ denote the least ordinal in $X$ that is moved by $j$. Then $\bar{\nu}$ is a cardinal with $\rho<\bar{\nu}\leq\bar{\kappa}$. Set $\nu=j(\bar{\nu})$. We then know that $\nu$ is a cardinal with $\bar{\nu}<\nu\leq\kappa<\delta$.  
 \end{proof}

 By Lemma \ref{lemma:SubtleFromSR-SR--}, the above claim shows that the interval $(\rho,\delta]$ contains a subtle cardinal, contradicting our assumptions.

 The above computations show that  \eqref{item:SR-SR--} in Theorem \ref{theorem:SubtleChar} implies \eqref{item:SubtleLimSubtle} of the theorem. 
 In the other direction, Corollary \ref{corollary:SubtleYieldsSR-SR--} states that \eqref{item:SubtleLimSubtle} also implies \eqref{item:SR-SR--}. 
\end{proof}

Note that results of Matet show that a cardinal $\delta$ is either subtle or a limit of subtle cardinals if and only if for every $\delta$-list $\seq{E_\gamma}{\gamma<\delta}$ and every $\rho<\delta$, there exist  $\rho<\mu<\nu<\delta$ satisfying  $E_\mu=E_\nu\cap\mu$ (see {\cite[Corollary 3]{MR0930262}}). 
 An alternative proof of this equivalence can be obtained from a combination of Theorem \ref{theorem:SubtleChar} and Lemma \ref{lemma:SubtleSR-SR--}.

%%%%%%%%%%%%%%%%%%%%%%%%%%
%%%%%%%%%%%%%%%%%%%%%%%%%%

\section{Extenders deriving from Product Reflection}\label{prodref}

The Product Reflection Principle ($\PRP$) yields the existence of strong cardinals, via extenders derived from the principle holding for a single $\Pi_1$-definable class $\Ce$ of structures (see \cite[Theorem 4.1]{BagariaWilsonRefl}). In this section, we shall derive extenders in a more general context, including weak forms of Product Structural Reflection ($\WPR$) restricted to $\Pi_1$-definable classes of structures, which shall then be used in the proofs of Theorems \ref{theorem:CharWPR} and \ref{main:ProdSubtle}, given below in Sections \ref{section7} and \ref{section8}, respectively.

In the following, let $\calL^\prime$ denote a first-order language that extends the language of set theory  by constants $\dot{\kappa}$ and $\dot{u}$, plus, possibly,  finitely-many other predicate, constant, and function symbols.
Let $\calL$ denote the first-order language that extends $\calL^\prime$ by an $(n+1)$-ary predicate symbol $\dot{T}_\varphi$ for  every $\calL^\prime$-formula $\varphi(v_0,\ldots,v_n)$ with $(n+1)$-many free variables. %natural number  $n$. 
 Define $\calS_\calL$ to be the class of all $\calL$-structures $A$ such that there exists a cardinal $\kappa_A$ in $C^{(1)}$ and a limit  ordinal $\theta_A>\kappa_A$ 
 %transitive set $M_A$ with the property 
 such that: 
 \begin{enumerate}
     \item The domain of $A$ is $V_{\theta_A+1}$. 
     
     \item $\in^A={\in}\restriction{V_{\theta_A+1}}$, $\dot{\kappa}^A=\kappa_A$ and $\dot{u}^A=\theta_A$. 
     
     \item If $\varphi(v_0,\ldots,v_n)$ is an $\calL^\prime$-formula, then $$\dot{T}_\varphi^A ~ = ~ \Set{\langle x_0,\ldots,x_n\rangle\in V_{\theta_A+1}^{n+1}}{A\models\varphi(x_0,\ldots,x_n)}.$$
 \end{enumerate}
 Note that   $\calS_\calL$ is definable by a $\Pi_1$-formula  without parameters.

For the rest of the section, we assume that 
 \begin{itemize}
     \item $\Ce$ be a subclass of $\calS_\calL$, and 
     
     \item $\zeta$ is an ordinal with $\Ce\cap V_\zeta\neq\emptyset$. 
 \end{itemize}
   We define  $$\kappa ~ = ~ \sup\Set{\kappa_A}{A\in\Ce\cap V_\zeta} ~ \leq ~ \zeta.$$
 For every set $x$, we let $f_x$ denote the unique  function with domain $\Ce \cap V_\zeta$ such that  that $f_x(A)=x$ for all $A$ with $x\in V_{\kappa_A}$, and $f_x(A)=\dot{u}^A$ otherwise. 
 In addition, let $f^x$ denote the unique   function with domain $\Ce \cap V_\zeta$ and such that $f^x(A)=x\cap V_{\kappa_A}$, for all $A$ in the domain.
 For the remainder of this section, we also assume that 
 \begin{itemize}
  %\item $\calL^\prime$ is a finite first-order language extending $\calL_*$, 
  
  %\item $\Ce$ is a subclass of $\calS_\calL$, 
  
  %\item $\zeta$ is an ordinal with $\Ce\cap V_\zeta\neq\emptyset$, 
  
  %\item $\kappa=\sup\Set{\kappa_A}{A\in\Ce\cap V_\zeta}$, 
  
  \item $X$ is a substructure of $\prod(\Ce\cap V_\zeta)$ with $f_x\in X$, for all $x\in V_\kappa $, 
  
  \item $B$ is an element of $\calS_\calL$ with $\kappa_B\geq\kappa$, and 
  
  \item   $\map{h}{X}{B}$ is a homomorphism. 
 \end{itemize}
 
 \begin{lemma}\label{lemma:Prod1}
 The following hold: 
 \begin{enumerate}
     \item\label{item:Prod1} If $\varphi(v_0,\ldots,v_{n-1})$ is an $\calL^\prime$-formula and $g_0,\ldots,g_{n-1}\in X$ with   $$A\models\varphi(g_0(A),\ldots,g_{n-1}(A))$$ for all $A\in\Ce\cap V_\zeta$, then $$B\models\varphi(h(g_0),\ldots,h(g_{n-1})).$$  
     
     %\item\label{item:ProdNull}   $h(f_\kappa)=\dot{u}^B$. 
     
     \item\label{item:Prod2} If $x\in V_\kappa$ with $h(f_x)\neq \dot{u}^B$, then $h(f_x)\in V_{\kappa_B}$. 
      
     \item\label{item:Prod3} If $\alpha<\kappa$ with $h(f_\alpha)\neq \dot{u}^B$, then $h(f_\alpha)<\kappa_B$. 
     
     %\item\label{item:Prod4} If $E\subseteq V_\zeta$ with $f^E\in X$, then $h(f^E)\neq \dot{u}^B$. 
     
     \item\label{item:Prod+1} If $E_0,E_1\subseteq V_\zeta$ with $f^{E_0},f^{E_1}\in X$, then $f^{E_0\cap E_1}\in X$ implies that $h(f^{E_0})\cap h(f^{E_1})=h(f^{E_0\cap E_1})$ and $f^{E_0\cup E_1}\in X$ implies that $h(f^{E_0})\cup h(f^{E_1})=h(f^{E_0\cup E_1})$. Moreover, if $f^{E_0\setminus E_1}\in X$, then $h(F^{E_0})\setminus h(f^{E_1})=h(f^{E_0\setminus E_1})$. 
     
      \item\label{item:Prod+2} If $E_0\subseteq E_1\subseteq V_\zeta$ with $f^{E_0},f^{E_1}\in X$, then $h(f^{E_0})\subseteq h(f^{E_1})$. 
 \end{enumerate}
\end{lemma}

\begin{proof}
 \eqref{item:Prod1} 
  Fix an $\calL^\prime$-formula $\varphi(v_0,\ldots,v_{n-1})$   and $g_0,\ldots,g_{n-1}\in X$ with   $A\models\varphi(g_0(A),\ldots,g_{n-1}(A))$  for all $A\in\Ce\cap V_\zeta$. 
  % Let $\ulcorner\varphi\urcorner$ be the canonical element of $\mathsf{Fml}_\calL$ corresponding to $\varphi$. 
  %
  Then $A\models \dot{T}_\varphi(g_0(A),\ldots,g_{n-1}(A))$ holds for all $A\in\Ce\cap V_\zeta$. By the definition of the product structure, this implies that $\prod(\Ce\cap V_\zeta)\models\dot{T}_\varphi(g_0,\ldots,g_{n-1})$ and hence $X\models\dot{T}_\varphi(g_0,\ldots,g_{n-1})$. 
  Since $h$ is a homomorphism, this shows that  $B\models\dot{T}_\varphi(h(g_0),\ldots,h(g_{n-1}))$ holds, and therefore  $B\models\varphi(h(g_0),\ldots,h(g_{n-1}))$. 
  
  %\eqref{item:ProdNull} This is clear since $f_\kappa (A)=\dot{u}^A$, for all $A\in (\Ce \cap V_\zeta)$,  and $h$ is a homomorphism.
  
 \eqref{item:Prod2} Pick $x\in V_\kappa$ with $h(f_x)\neq \dot{u}^B$.  Given $A\in\Ce\cap V_\zeta$, we  have $$A\models ``f_x(A)\neq\dot{u}\longrightarrow \textit{$f_x(A)$ is an element of $V_{\dot{\kappa}}$}".$$ Using \eqref{item:Prod1}, the fact that $h(f_x)\neq \dot{u}^B$ now directly implies that $h(f_x)$ is an element of $V_{\kappa_B}$.

  \eqref{item:Prod3} This implication follows directly from the combination of \eqref{item:Prod1} and \eqref{item:Prod2}.

  %\eqref{item:Prod4} %Fix $E\subseteq V_\zeta$ with $f^E\in X$. 
   %By the  definition of $F^E$, we have $f^E(A)\neq \dot{u}^A$ for all $A\in\Ce\cap V_\zeta$.  Therefore, we can use \eqref{item:Prod1} to conclude that  $h(f^E)\neq \dot{u}^B$. 
  
  \eqref{item:Prod+1} Since the given definitions ensure that  $f^{E_0}(A)\cap f^{E_1}(A)=f^{E_0\cap E_1}(A)$, $f^{E_0}(A)\cup f^{E_1}(A)=f^{E_0\cup E_1}(A)$, and $f^{E_0}(A)\setminus f^{E_1}(A)=f^{E_0\setminus E_1}(A)$,  for all $A\in\Ce\cap V_\zeta$, the desired implications follow immediately from \eqref{item:Prod1}. 
  
  \eqref{item:Prod+2} Since our assumptions imply that $E_0\cap E_1=E_0$, a combination of \eqref{item:Prod1} and \eqref{item:Prod+1} shows that $h(f^{E_0})\cap h(f^{E_1})=h(f^{E_0})$ and hence $h(f^{E_0})\subseteq h(f^{E_1})$. 
\end{proof}

 Note that we have  $f_\kappa(A)=\dot{u}^A=\theta_A$ for all $A\in\Ce\cap V_\zeta$ and this implies that $f_\kappa\in X$ with $h(f_\kappa)=\dot{u}^B=\theta_B>\kappa_B\geq\kappa$. In particular, we know that $f_\mu\in X$ holds for all ordinals $\mu\leq\kappa$.

\begin{lemma}\label{lemma:ProdCons}
 Suppose that $\mu$ is an ordinal less than or equal to $\kappa$ such that $h(f_\mu)\neq\mu$, $f^\mu\in X$, and $h(f_x)=x$ for all $x\in V_\mu$. 
 Then the following hold: 
 \begin{enumerate}
  \item\label{item:ProdCons1} If $E\subseteq V_\mu$ with $f^E\in X$, then $h(f^E)\cap V_\mu=E$. 
  
    \item\label{item:ProdCons3} If $h(f_\mu)=\dot{u}^B$, then $h(f^\mu)=\kappa_B$. 
  
  \item\label{item:ProdCons4} If $h(f_\mu)\neq\dot{u}^B$, then $h(f^\mu)=h(f_\mu)<\kappa_B$. 
  
  \item\label{item:ProdCons2.1} $\mu$ is a strong limit cardinal. 
  
    \item\label{item:ProdCons2.2New} If there exists $E\subseteq\mu$ of order-type $\cof{\mu}$ with $f^E\in X$ and either $\mu<h(f^\mu)$ or $\kappa_A<\kappa_B$ holds for all $A\in\Ce\cap V_\zeta$, then $\mu$ is regular and, by \eqref{item:ProdCons2.1}, it follows that $\mu$ is  an inaccessible cardinal. 
  
  \item\label{item:ProdCons5} $f^\emptyset\in X$ with $h(f^\emptyset)=\emptyset$ and, if $f^{[\mu]^{{<}\omega}}\in X$, then $h(f^{[\mu]^{{<}\omega}})=[h(f^\mu)]^{{<}\omega}$. 
  %If $E\subseteq[\mu]^{{<}\omega}$ with $f^E\in X$, then $h(f^E)\subseteq[h(f^\mu)]^{{<}\omega}$. 
 \end{enumerate}
\end{lemma}

\begin{proof}
 \eqref{item:ProdCons1}   First, pick $x\in E$. For each $A\in\Ce\cap V_\zeta$,  we have 
 $$A\models ``f_x(A)\ne \dot{u}\longrightarrow f_x(A)  \in   f^E(A)".$$ 
 By Lemma \ref{lemma:Prod1}.\eqref{item:Prod1}, the fact that $h(f_x)=x\neq \dot{u}^B$  implies that $x=h(f_x)\in h(f^E)$. This shows that $E\subseteq h(f^E)\cap V_\mu$. 
    
    In the other direction, pick $x\in V_\mu\setminus E$. Then $$A\models ``f_x(A)  \in  f^E(A)\longrightarrow f_x(A)=\dot{u}"$$ holds for all $A\in\Ce\cap V_\zeta$, and since $h(f_x)=x\neq \dot{u}^B$, by  Lemma \ref{lemma:Prod1}.\eqref{item:Prod1} we must have that   $x=h(f_x)\notin h(f^E)$. This shows $h(f^E)\cap V_\mu\subseteq E$, thus proving the desired equality. 
 
   \eqref{item:ProdCons3} For each $A\in\Ce\cap V_\zeta$, we have $$A\models ``f_\mu(A)=\dot{u} ~ \longrightarrow ~ f^\mu(A)=\dot{\kappa}".$$ 
  By Lemma \ref{lemma:Prod1}.\eqref{item:Prod1}, this yields the desired implication. 
  
   \eqref{item:ProdCons4} For each $A\in\Ce\cap V_\zeta$, we have $$A\models ``f_\mu(A)\neq\dot{u} ~ \longrightarrow ~ f^\mu(A)=f_\mu(A)<\dot{\kappa}".$$ The desired implication now follows from another application of Lemma \ref{lemma:Prod1}.\eqref{item:Prod1}. 
   
  \eqref{item:ProdCons2.1} First, note that since $h(f_\mu)\neq\mu$,  Lemma \ref{lemma:Prod1}.\eqref{item:Prod1} directly implies that $\mu>\omega$. 
  Now assume, towards a contradiction, that there is a cardinal $\rho<\mu$ with $2^\rho\geq\mu$. Given $A\in\Ce\cap V_\zeta$ with $\rho<\kappa_A$, the fact that $\kappa_A\in C^{(1)}$ implies that $2^\rho<\kappa_A\leq\kappa$. 
  Using Lemma \ref{lemma:Prod1}.\eqref{item:Prod1}, we can now show that $h(f_{2^\rho})=2^\rho$, $h(f_{V_{2^\rho}})=V_{2^\rho}$ and $h(f_x)\in V_{2^\rho}$ for all $x\in V_{2^\rho}$. 
  In particular, since $h(f_\mu)\ne \mu$ and $2^\rho \geq \mu$, we must have $\mu<2^\rho$. Another application of  Lemma \ref{lemma:Prod1}.\eqref{item:Prod1}  shows that the map $$\Map{j}{V_{2^\rho}}{V_{2^\rho}}{x}{h(f_x)}$$ is an elementary embedding. 
  Since $\cof{2^\rho}>\rho \geq\omega$, the \emph{Kunen Inconsistency} implies that $j$ is the identity on $V_{2^\rho}$ and hence $h(f_\mu)=j(\mu)=\mu$, a contradiction. 
  
    \eqref{item:ProdCons2.2New} Assume, towards a contradiction, that $\mu$ is singular and pick $E\subseteq\mu$ of order-type $\cof{\mu}$ with $f^E\in X$. 
  We then have 
   \begin{equation*}
    \begin{split}
     A\models ``[ & f_{\cof{\mu}}(A)\neq\dot{u} ~ \wedge ~ f_\mu(A)=\dot{u} ~ \wedge ~ \otp{f^E(A)}\geq f_{\cof{\mu}}(A)] \\ 
     & \longrightarrow ~ \textit{$f^E(A)$ is a cofinal subset of $\dot{\kappa}$ of order-type $f_{\cof{\mu}}(A)$}"
    \end{split}
   \end{equation*} 
  for all $A\in\Ce\cap V_\zeta$. By Lemma  \ref{lemma:Prod1}.\eqref{item:Prod1}, and the assumption that  $h(f_x)=x$ for all $x\in V_\mu$, which in particular yields $h(f_{\cof{\mu}})=\cof{\mu}$, we have that
   \begin{equation*}
    \begin{split}
     B\models ``[ & \cof{\mu}\neq\dot{u} ~ \wedge ~ h(f_\mu)=\dot{u} ~ \wedge ~ \otp{h(f^E)}\geq \cof{\mu}] \\ 
     & \longrightarrow ~ \textit{$h(f^E)$ is a cofinal subset of $\dot{\kappa}$ of order-type $\cof{\mu}$}".
    \end{split}
   \end{equation*} 

  \begin{claim*}
   $h(f_\mu)\neq\dot{u}^B$. 
  \end{claim*}
  
  \begin{proof}[Proof of the Claim]
   Assume, towards a contradiction, that $h(f_\mu)=\dot{u}^B$. Since \eqref{item:ProdCons1} shows that we have  $\otp{h(f^E)}\geq\otp{E}=\cof{\mu}$, a combination of the above observation with the fact that $\cof{\mu}<\mu \leq \kappa < \dot{u}^B$ then shows that $\mu=\kappa_B=\sup(E)$. 
   By \eqref{item:ProdCons3}, this also shows that $h(f^\mu)=\mu$ and our assumptions imply that $\kappa_A<\kappa_B$ holds for all $A\in\Ce\cap V_\zeta$. But then Lemma \ref{lemma:Prod1}.\eqref{item:Prod1} allows us to find $A\in\Ce\cap V_\zeta$ with $\cof{\mu}<\kappa_A\leq\mu$ and $\kappa_A=\sup(E)=\kappa_B$, contradicting the fact that $\kappa_A<\kappa_B$ holds for all $A\in\Ce\cap V_\zeta$.  
  \end{proof}

  %
  %
  %$h(f^E)$ would be a cofinal subset of $\kappa_B$ of order-type $\cof{\mu}$. 
  %
  %Since $E\subseteq h(f^E)$, this implies that $E$ is a cofinal subset of $\kappa_B$. But this is impossible because by assumption  $h(f_\mu)\ne \mu$, and clearly $h(f_{\kappa_B})=\kappa_B$, hence  $\mu <\kappa_B$.  
  %
  %
  Now using Lemma \ref{lemma:Prod1}.\eqref{item:Prod1} again, the fact that $$A\models ``f_\mu(A)\neq\dot{u}\longrightarrow \otp{f^E(A)}=f_{\cof{\mu}}(A) ~ \wedge ~ f_\mu(A)=\sup(f^E(A))"$$ holds for all $A\in \Ce \cap V_\zeta$, implies that $\otp{h(f^E)}=\cof{\mu}$ and $h(f_\mu)=\sup(h(f^E))$. Since $E\subseteq h(f^E)$, this implies $h(f_\mu)=\sup(E)=\mu$, contradicting our initial assumptions.  
  
  \eqref{item:ProdCons5} Since $f_\emptyset(A)=\emptyset=f^\emptyset(A)$ holds for all $A\in\Ce\cap V_\zeta$, we know that $f^\emptyset\in X$ and Lemma \ref{lemma:Prod1}.\eqref{item:Prod1} directly implies that $h(f^\emptyset)=\emptyset$. 
  Next, assume that $f^{[\mu]^{{<}\omega}}\in X$ and note that we have $$f^{[\mu]^{{<}\omega}}(A) ~ = ~ [\mu]^{{<}\omega}\cap V_{\kappa_A} ~ = ~ [\mu\cap\kappa_A]^{{<}\omega} ~ = ~ [f^\mu(A)]^{{<}\omega}$$ for all $A\in\Ce\cap V_\zeta$. By Lemma \ref{lemma:Prod1}.\eqref{item:Prod1}, this implies that $h(f^{[\mu]^{{<}\omega}})=[h(f^\mu)]^{{<}\omega}$.  
\end{proof}

\begin{lemma}\label{lemma:ProdReflExtender}
Let $\mu$ be as in Lemma \ref{lemma:ProdCons} and assume, moreover, that $\mu<h(f^\mu)$ and $f^E\in X$ for all $E\subseteq V_\mu$. Set $\nu=h(f^\mu)$ and define $$\mathsf{E}_a ~ = ~ \Set{E\subseteq[\mu]^{\betrag{a}}}{a\in h(f^E)}$$ for all $a\in[\nu]^{{<}\omega}$. Then the resulting system $$\Ee ~ = ~ \seq{\mathsf{E}_a}{a\in[\nu]^{{<}\omega}}$$ is a \emph{$(\mu,\nu)$-extender} (as defined in {\cite[pp. 354--355]{MR1994835}}). 
\end{lemma}

\begin{proof}
 We prove the lemma through a series of claims. 
 
 \begin{claim*}
  For every $a\in [\nu]^{{<}\omega}$, the collection $\mathsf{E}_a$ is a ${<}\mu$-complete ultrafilter on $[\mu]^{\betrag{a}}$. 
 \end{claim*}
 
 \begin{proof}[Proof of the Claim]
  First, note that a combination of clauses \eqref{item:Prod+1} and \eqref{item:Prod+2} of Lemma \ref{lemma:Prod1} and Lemma \ref{lemma:ProdCons}.\eqref{item:ProdCons5} shows that  $\mathsf{E}_a$ is an ultrafilter on $[\mu]^{\betrag{a}}$. 
  Now, fix $\rho<\mu$ and a sequence $\seq{E_\alpha}{\alpha<\rho}$ of elements of $\mathsf{E}_a$. Define $G=\bigcap\Set{E_\alpha}{\alpha<\rho}$ and $H=\Set{\langle x,\alpha\rangle}{\alpha<\rho, ~ x\in E_\alpha}\subseteq V_\mu$. 
  Given $A\in\Ce\cap V_\zeta$ with $\rho\leq \kappa_A$, the fact that $\mu$ and $\kappa_A$ are both cardinals implies that  $$x ~ \in ~  f^{E_\alpha}(A) ~   =  ~ E_\alpha\cap[\kappa_A]^{\betrag{a}} ~ \Longleftrightarrow ~ \langle x,\alpha\rangle ~ \in ~  f^H(A) ~ = ~ H\cap V_{\kappa_A}$$ holds for all $\alpha<\rho$ and this allows us to conclude that $$f^G(A) ~ = ~ \Set{x\in[f^\mu(A)]^{\betrag{a}}}{\textit{$\langle x,\alpha\rangle\in H$ for all $\alpha<\rho$}}.$$
  Since $h(f_\alpha)=\alpha$ for all $\alpha\leq\rho$, we can use Lemma \ref{lemma:Prod1}.\eqref{item:Prod1} to show that $$h(f^{E_\alpha}) ~ = ~ \Set{x\in[\nu]^{\betrag{a}}}{\langle x,\alpha\rangle\in h(f^H)}$$ for all $\alpha<\rho$ and $$a ~ \in ~ \bigcap\Set{h(f^{E_\alpha})}{\alpha<\rho} ~ = ~ \Set{x\in[\nu]^{\betrag{a}}}{\textit{$\langle x,\alpha\rangle\in h(f^H)$ for all $\alpha<\rho$}} ~ = ~   h(f^G).$$
  This proves that $\bigcap\Set{E_\alpha}{\alpha<\rho}\in\mathsf{E}_\alpha$. 
 \end{proof}

 \begin{claim*}
  $\mathsf{E}_{\{\mu\}}$ is not ${<}\mu^+$-complete. 
 \end{claim*}
 
 \begin{proof}[Proof of the Claim]
  Given $\alpha<\mu$, set $E_\alpha=\Set{\{\beta\}}{\alpha<\beta<\mu}\subseteq[\mu]^1$.  For each $\alpha<\mu$, Lemma \ref{lemma:Prod1}.\eqref{item:Prod1} then implies that $h(f^{E_\alpha})=\Set{\{\beta\}}{\alpha<\beta<\nu}$ and this shows that $E_\alpha\in\mathsf{E}_{\{\mu\}}$. But $\bigcap\Set{E_\alpha}{\alpha<\mu}=\emptyset$ and this yields the statement of the claim. 
 \end{proof}
 
  \begin{claim*}
  If $\alpha<\mu$, then $\mathsf{E}_{\{\alpha\}}=\Set{x\subseteq[\mu]^1}{\{\alpha\}\in x}$. 
 \end{claim*}
 
  \begin{proof}[Proof of the Claim]
  Fix $x\subseteq[\mu]^1$. If $\{\alpha\}\in x$, then $$A\models ``f_\alpha(A)\neq\dot{u} ~ \longrightarrow ~ f_{\{\alpha\}}(A)\in f^x(A)"$$ holds for all $A\in\Ce\cap V_\zeta$, and hence Lemma \ref{lemma:Prod1}.\eqref{item:Prod1} implies that $\{\alpha\}=h(f_{\{\alpha\}})\in h(f^x)$, which yields  $x\in\mathsf{E}_{\{\alpha\}}$. 
  In the other direction, if $x\in\mathsf{E}_{\{\alpha\}}$, then $\{\alpha\}\in h(f^x)\cap[\mu]^1$ and Lemma \ref{lemma:ProdCons}.\eqref{item:ProdCons1} implies that $\{\alpha\}\in x$. 
 \end{proof}

 Given $a\subseteq b\in[\nu]^{{<}\omega}$, we let $\map{\pi_{b,a}}{[\mu]^{\betrag{b}}}{[\mu]^{\betrag{a}}}$ denote the canonical induced map, {i.e.} if $\beta_0<\ldots<\beta_{\betrag{b}-1}$ is the monotone enumeration of $b$ and $i_0<\ldots<i_{\betrag{a}-1}$ is the unique strictly increasing sequence of natural numbers less than $\betrag{b}$ such that $a=\{\beta_{i_0},\ldots,\beta_{i_{\betrag{a}-1}}\}$, then $$\pi_{b,a}(x) ~ = ~ \{\alpha_{i_0},\ldots,\alpha_{i_{\betrag{a}-1}}\}$$ for all $x\in[\mu]^{\betrag{b}}$ with monotone enumeration $\alpha_0<\ldots<\alpha_{\betrag{b}-1}$.

 \begin{claim*}[Coherence]
  If $a\subseteq b\in[\nu]^{{<}\omega}$, then $\mathsf{E}_a=\Set{E\subseteq[\mu]^{\betrag{a}}}{\pi_{b,a}^{{-}1}[E]\in\mathsf{E}_b}$. 
 \end{claim*}
 
 \begin{proof}[Proof of the Claim] 
  Given $A\in\Ce\cap V_\zeta$, let $\map{\pi^A_{b,a}}{[f^\mu(A)]^{\betrag{b}}}{[f^\mu(A)]^{\betrag{a}}}$ denote the induced canonical projection map. 
   
  Fix $E\subseteq[\mu]^{\betrag{a}}$. If $A\in\Ce\cap V_\zeta$ is such that  $\mu<\kappa_A$, then we have $f^\mu(A)=\mu$, $f^E(A)=E$, $f^{\pi^{{-}1}_{b,a}[E]}(A)=\pi^{{-}1}_{b,a}[E]$, $\pi^A_{b,a}=\pi_{b,a}$, and hence $$x\in f^{\pi^{{-}1}_{b,a}[E]}(A) ~ \Longleftrightarrow ~ \pi^A_{b,a}(x)\in f^E(A)$$  for all $x\in[f^\mu(A)]^{\betrag{b}}$. 
 In the other case, namely if  $\mu\geq\kappa_A$,  then we have $f^\mu(A)=\kappa_A$, $f^E(A)=E\cap V_{\kappa_A}$, $f^{\pi^{{-}1}_{b,a}[E]}(A)=(\pi^{{-}1}_{b,a}[E])\cap V_{\kappa_A}$ and therefore $$x\in f^{\pi^{{-}1}_{b,a}[E]}(A) ~ \Longleftrightarrow ~ \pi_{b,a}(x)=\pi^A_{b,a}(x)\in E ~ \Longleftrightarrow ~ \pi^A_{b,a}(x)\in f^E(A)$$  for all $x\in[f^\mu(A)]^{\betrag{b}}\subseteq V_{\kappa_A}$.  
 If we now define $\map{\tilde{\pi}_{b,a}}{[\nu]^{\betrag{b}}}{[\nu]^{\betrag{a}}}$ to be the induced canonical projection map, then an application of  Lemma \ref{lemma:Prod1}.\eqref{item:Prod1} yields $h(f^{\pi^{{-}1}_{b,a}[E]})=\tilde{\pi}_{b,a}^{{-}1}[h(f^E)]$.  
  This equality allows us now to conclude that $$E\in\mathsf{E}_a ~ \Longleftrightarrow ~ \tilde{\pi}_{b,a}(b)=a\in h(f^E) ~ \Longleftrightarrow ~ b\in h(f^{\pi^{{-}1}_{b,a}[E]}) ~ \Longleftrightarrow ~ \pi^{{-}1}_{b,a}[E]\in\mathsf{E}_b$$ holds for all $E\subseteq[\mu]^{\betrag{a}}$.  
 \end{proof}

 \begin{claim*}[Well-foundedness]
  If $\seq{a_n}{n<\omega}$ is a sequence of elements of $[\nu]^{{<}\omega}$ and $\seq{E_n}{n<\omega}$ is such that  $E_n\in\mathsf{E}_{a_n}$ for all $n<\omega$, then there exists a function $$\map{d}{\bigcup\Set{a_n}{n<\omega}}{\mu}$$ with $d[a_n]\in E_n$ for all $n<\omega$. 
 \end{claim*}
 
 \begin{proof}[Proof of the Claim]
  Assume, towards a contradiction, that there exist sequences $\seq{a_n}{n<\omega}$ and $\seq{E_n}{n<\omega}$ witnessing that the  claim fails. 
  Then by the previous Claim (Coherence) we may assume that  $\betrag{a_n}=n$, $a_{n+1}=a_n\cup\{\max(a_{n+1})\}$ and $E_{n+1}\subseteq\pi^{{-}1}_{a_{n+1},a_n}E_n$ hold for all $n<\omega$. 
  Define $T$ to be the subset of ${}^{{<}\omega}\mu$ consisting of all  strictly increasing sequences $t$  with $\ran{t}\in E_{\length{t}}$. Our assumptions then imply that $T$ is a subtree of ${}^{{<}\omega}\mu$. 
  Moreover, it follows that $T$ is well-founded, because if $\map{c}{\omega}{\mu}$ was a cofinal branch through $T$, then the map $$\Map{d}{\bigcup\Set{a_n}{n<\omega}}{\mu}{\max(a_{n+1})}{c(n)}$$ would satisfy $d[a_n]\in E_n$ for all $n<\omega$. 
 For each $A\in\Ce\cap V_\zeta$,  $$A\models\anf{\textit{$f^T(A)$ is a well-founded subtree of ${}^{{<}\omega}f^\mu(A)$}}.$$ Hence, by  Lemma \ref{lemma:Prod1}.\eqref{item:Prod1} and the fact $\POT{\mu}\subseteq V_{\theta_B}$, we have that $h(f^T)$ is a well-founded subtree of ${}^{{<}\omega}\nu$. 
  
  Define $E=\Set{\langle x,n\rangle}{n<\omega,  ~ x\in E_n}\subseteq V_\mu$. Given $n<\omega$, we have $$A\models ``f^{E_n}(A)=\Set{x\in[f^\mu]^n}{\langle x,n\rangle\in f^E(A)}"$$  for all $A\in\Ce\cap V_\zeta$, and hence $h(f^{E_n})=\Set{x\in[\nu]^n}{\langle x,n\rangle\in h(f^E)}$. 
  Moreover, since we also have that  $$A\models\anf{\textit{$f^T(A)$ consists of  all  strictly increasing  $t$ in  ${}^{{<}\omega}f^\mu(A)$ with $\langle\ran{t},\length{t}\rangle\in f^E(A)$}},$$ we can conclude that $h(f^T)$ consists of all strictly increasing sequences $t$ in ${}^{{<}\omega}\nu$ with the property that $\langle\ran{t},\length{t}\rangle\in h(f^E)$. 
  Now define $$\Map{c}{\omega}{\nu}{n}{\max(a_{n+1})}.$$ Given $n<\omega$, we  have that $\ran{c\restriction n}=a_n\in h(f^{E_n})$,  and therefore $\langle \ran{c\restriction n},n\rangle\in h(f^E)$. This shows that $c\restriction n\in h(f^T)$ for all $n<\omega$, contradicting the well-foundedness of the tree $h(f^T)$. 
 \end{proof}

 \begin{claim*}[Normality]
  If $ a\in[\nu]^{{<}\omega}$ and $\map{r}{[\mu]^{\betrag{a}}}{\mu}$ is such that  $$\Set{x\in[\mu]^{\betrag{a}}}{r(x)<\max(x)} ~ \in ~ \mathsf{E}_a,$$ then there exists $a\subseteq b\in[\nu]^{{<}\omega}$ with $\max(a)=\max(b)$ and $$\Set{x\in[\mu]^{\betrag{b}}}{(r\circ \pi_{b,a})(x)\in x} ~ \in ~ \mathsf{E}_b.$$
 \end{claim*}
 
 \begin{proof}[Proof of the Claim]
  Set $E=\Set{x\in[\mu]^{\betrag{a}}}{r(x)<\max(x)}$. Let $A\in\Ce\cap V_\zeta$. If $\mu<\kappa_A$, then $f^r(A)=r$, and therefore $$f^E(A) ~ = ~ E ~ = ~ \Set{x\in[f^\mu(A)]^{\betrag{a}}}{x\in\dom{f^r(A)}, ~ f^r(A)(x)<\max(x)}.$$ 
 Now assume   $\mu\geq\kappa_A$ and let us see that the last equality also holds in this case. 
 If $x\in f^E(A)=E\cap V_{\kappa_A}$, then  $x\in[\kappa_A]^{\betrag{a}}$ and  $r(x)<\max(x)<\kappa_A$, hence $x\in\dom{f^r(A)}$ with $f^r(A)(x)=r(x)$. In the other direction, if $x\in[f^\mu (A)]^{|a|}=[\kappa_A]^{\betrag{a}}$ with $x\in\dom{f^r(A)}$ and $f^r(A)(x)<\max(x)$, then $r(x)=f^r(A)(x)$ and $x\in E$. 
 
 An application of Lemma \ref{lemma:Prod1}.\eqref{item:Prod1} now yields $$h(f^E) ~ = ~ \Set{x\in[\nu]^{\betrag{a}}}{x\in\dom{h(f^r)}, ~ h(f^r)(x)<\max(x)}.$$ Since $E\in \mathsf{E}_a$, and so $a\in h(f^E)$, the equality above implies that $a\in\dom{h(f^r)}$ and $h(f^r)(a)<\max(a)$. Set $b=a\cup\{h(f^r)(a)\}$ and $$D ~ = ~ \Set{x\in[\mu]^{\betrag{b}}}{(r\circ\pi_{b,a})(x)\in x}.$$ 
 Letting $\map{\tilde{\pi}_{b,a}}{[\nu]^{\betrag{b}}}{[\nu]^{\betrag{a}}}$ be the induced  canonical projection, a variation of the above argument now shows that $$h(f^D) ~ = ~ \Set{x\in[\nu]^{\betrag{b}}}{\tilde{\pi}_{b,a}(x)\in\dom{h(f^r)}, ~ (h(f^r)\circ\tilde{\pi}_{b,a})(x)\in x}.$$ 
 Since $\tilde{\pi}_{b,a}(b)=a$ and $h(f^r)(a)\in b$, we can conclude that $b\in h(f^D)$ and hence $D\in\mathsf{E}_b$.  
 \end{proof}
 
 This completes the proof of the lemma. 
\end{proof}

\begin{lemma}\label{Lemma:HomStrongA}
 Under the assumptions of Lemma \ref{lemma:ProdReflExtender},  the following hold: 
 \begin{enumerate}
  \item\label{item:HomStrong1}  The cardinal $\mu$ is inaccessible and $\nu$ is an element of $C^{(1)}$.  
  
  \item\label{item:HomStrong2} Let $\Ee  =  \seq{\mathsf{E}_a}{a\in[\nu]^{{<}\omega}}$ be the \emph{$(\mu,\nu)$-extender} given by Lemma \ref{lemma:ProdReflExtender}, 
 let $$\langle\seq{M_a}{a\in[\nu]^{{<}\omega}}, ~ \seq{\map{j_a}{V}{M_a}}{a\in[\nu]^{{<}\omega}}\rangle$$ denote the induced system of ultrapowers of $V$ and ultrapower embeddings (see {\cite[p. 355]{MR1994835}}), 
 let $$\seq{\map{i_{a,b}}{M_a}{M_b}}{a\subseteq b\in[\nu]^{{<}\omega}}$$ denote the system of elementary embeddings induced by the projections $\map{\pi_{b,a}}{[\mu]^{\betrag{b}}}{[\mu]^{\betrag{a}}}$ (see {\cite[p. 354]{MR1994835}}), 
 let $$\langle M_\Ee, ~ \seq{\map{k_a}{M_a}{M_\Ee}}{a\in[\nu]^{{<}\omega}}\rangle$$ denote the direct limit of the directed system $$\langle\seq{M_a}{a\in[\nu]^{{<}\omega}}, ~ \seq{\map{i_{a,b}}{M_a}{M_b}}{a\subseteq b\in[\nu]^{{<}\omega}}\rangle$$ and let $\map{j_\Ee}{V}{M_\Ee}$ denote the induced  embedding. 
 We then have $\crit{j_\Ee}=\mu$, $j_\Ee(\mu)\geq\nu$ and $V_\nu\subseteq M_\Ee$. 
 \end{enumerate} 
\end{lemma}
 
\begin{proof}
  \eqref{item:HomStrong1} Since we assumed that $\mu<h(f^\mu)$, Lemma  \ref{lemma:ProdCons}.\eqref{item:ProdCons2.2New} directly implies that $\mu$ is inaccessible. 
  %our assumptions imply that $\mu$ is an inaccessible cardinal. 
 %
 Given $A\in\Ce\cap V_\zeta$, we have $f^\mu(A)\in\{\mu,\kappa_A\}\subseteq C^{(1)}$ and this implies that  $V_{f^\mu(A)}\preceq_{\Sigma_1}V_{\dot{\kappa}^A}$.  Using Lemma \ref{lemma:Prod1}.\eqref{item:Prod1}, we now know that $V_\nu\preceq_{\Sigma_1}V_{\kappa_B}$ and, since $\kappa_B\in C^{(1)}$, we can conclude that $\nu\in C^{(1)}$. 
 %We then have  for all $A\in\Ce\cap V_\zeta$ and, by Lemma \ref{lemma:Prod1}.\eqref{item:Prod1}, this shows that $\nu$ is an element of $C^{(1)}$.  

 \eqref{item:HomStrong2} By {\cite[Lemma 26.2.(b)]{MR1994835}}, we  have  $\crit{j_\Ee}=\mu$ and $j_\Ee(\mu)\geq\nu$. 
  Fix a bijection $\map{g}{[\mu]^1}{V_\mu}$ with the property that $g[[\rho]^1]=V_\rho$ holds for every $\rho\in C^{(1)}\cap\mu$. 
  Next, for every $A\in\Ce\cap V_\zeta$ we have that  
 $$A\models\anf{\textit{$f^g(A)$ is a bijection between $[f^\mu(A)]^1$ and $V_{f^\mu(A)}$}}.$$ Hence, by Lemma \ref{lemma:Prod1}.\eqref{item:Prod1}, we know that  $h(f^g)$ is a bijection between $[\nu]^1$ and $V_\nu$. 
 Also, elementarity ensures that $j_\Ee(g)$ is a bijection between $[j_\Ee(\mu)]^1$ and $V_{j_\Ee(\mu)}^{M_\Ee}$ with the property that $j_\Ee(g)[[\rho]^1]=V_\rho^{M_\Ee}$ holds for all $\rho$ in $(C^{(1)})^{M_\Ee}$ less than or equal to $j_\Ee(\mu)$. 
Now, since \eqref{item:HomStrong1} shows that $\nu\in C^{(1)}$ and being in $C^{(1)}$ is a $\Pi_1$-property, we know that $\nu \in (C^{(1)})^{M_\Ee}$. Thus, $j_\Ee(g)[[\nu]^1]=V_\nu^{M_\Ee}$, and the map  $$\map{\iota ~ = ~ j_\Ee(g)\circ h(f^g)^{{-}1}}{V_\nu}{V_\nu^{M_\Ee}}$$ is a bijection.

 \begin{claim*}
  $\iota=\id_{V_\nu}$. 
 \end{claim*}
 
 \begin{proof}[Proof of the Claim]
  Since $V_\nu$ and $V_\nu^{M_\Ee}$ are transitive, it is sufficient to show that $\iota$ is an $\in$-homomorphism. Thus, let $x_0,x_1\in V_\nu$.  As $\map{h(f^g)}{[\nu]^1}{V_\nu}$ is a bijection, let  $a_0,a_1\in[\nu]^1$ be the preimages under $h(f^g)$ of $x_0$ and $x_1$, respectively. Set $a=a_0\cup a_1$ and $$E ~ = ~ \Set{x\in[\mu]^{\betrag{a}}}{g(\pi_{a,a_0}(x))\in g(\pi_{a,a_1}(x))}.$$ 
  Given $A\in\Ce\cap V_\zeta$, the fact that $g[[f^\mu(A)]^1]=V_{f^\mu(A)}$ implies that $f^g(A)=g\restriction[f^\mu(A)]^1$ and hence $$f^E(A) ~ = ~ \Set{x\in[f^\mu(A)]^{\betrag{a}}}{f^g(A)(\pi_{a,a_0}(x))\in f^g(A)(\pi_{a,a_1}(x))}.$$
  By Lemma \ref{lemma:Prod1}.\eqref{item:Prod1}, this shows that $$h(f^E) ~ =\Set{x\in[\nu]^{\betrag{a}}}{h(f^g)(\pi_{a,a_0}(x))\in h(f^g)(\pi_{a,a_1}(x))}.$$ 
 Thus, we have the following equivalences:
  $$
x_0\in x_1~  \Longleftrightarrow ~ h(f^g)(a_0)\in h(f^g)(a_1) ~  \Longleftrightarrow ~ a\in h(f^E) ~ \Longleftrightarrow ~ E\in\mathsf{E}_a$$
and the latter, by the definition of $E$ and the ultrapower map $\map{j_a}{V}{M_a}$, is equivalent to 
$$j_a(g)([\pi_{a,a_0}]_{\mathsf{E}_a})\in j_a(g)([\pi_{a,a_1}]_{\mathsf{E}_a}).$$
By applying the map $k_a$ to  the last displayed sentence,  and using the fact that $j_\Ee = k_a \circ j_a$, $k_{a_0}= k_a \circ i_{a_0,a}$, and $i_{a_0, a}([{\id}_{|a_0|}]_{\mathsf{E}_{a_0}})= [{\id}_{|a_0|} \circ \pi_{a_0,a}]_{\mathsf{E}_a}=[\pi_{a_0,a}]_{\mathsf{E}_{a}}$, and similarly for $a_1$, we have that
$$j_\Ee (g)(k_{a_0}([{\id}_{|a_0|}]_{\mathsf{E}_{a_0}}))\in j_\Ee(g)(k_{a_1}([{\id}_{|a_1|}]_{\mathsf{E}_{a_1}})).$$
Now, as in {\cite[Lemma 26.2.(a)]{MR1994835}}, $k_{a_0}([{\id}_{|a_0|}]_{\mathsf{E}_{a_0}})=a_0$, and similarly for $a_1$. Thus the last displayed sentence is equivalent to the first term of the following chain of equivalences
$$j_\Ee(g)(a_0)\in j_\Ee(g)(a_1)  
     \Longleftrightarrow   \iota(h(f^g)(a_0))\in\iota(h(f^g)(a_1))\Longleftrightarrow \iota (x_0)\in \iota (x_1).$$
     We have thus shown that $x_0\in x_1$ if and only if $\iota(x_0)\in \iota(x_1)$, which proves the Claim.
 \end{proof}
 
 The above claim shows that $V_\nu ={\rm range}(\iota) \subseteq M_\Ee$, and thus $V_\nu\subseteq M_\Ee$. 
\end{proof}

We shall end this section with the following lemma, which will be used in the proofs of Theorems \ref{theorem:CharWPR} and \ref{main:ProdSubtle}, given in Sections \ref{section7} and \ref{section8} below.

\begin{lemma}\label{lemma:Prod2}
% In the situation of Lemma \ref{lemma:Prod1}, 
 %if there exists $x\in V_{\kappa+1}$ with \todo[fancyline]{This always holds. Change.} $h(f_x)=\dot{u}^B$ and 
 If we define $$\lambda ~ = ~ \min\Set{\rank{x}}{f_x\in X, ~ h(f_x)=\dot{u}^B},$$ then the following statements hold: 
  \begin{enumerate}
    \item\label{item:Prod5} $\lambda$ is a limit ordinal with $\lambda\leq\kappa$ and  $h(f_\lambda)=\dot{u}^B$. 
         
 %   \item\label{item:Prod6} If $h(f_x)=x$ for all $x\in V_\lambda$, then $h(f^E)\cap V_\lambda=E$ holds for all $E\subseteq V_\lambda$ with $f^E\in X$. 
         
    \item\label{item:Prod7} If $x\in V_\lambda$ with $h(f_x)\neq x$, then there exists $\alpha\leq\rank{x}$ with $\alpha<h(f_\alpha)$.

   \item\label{item:ProdBonus} If we define $$\chi ~ = ~ \sup\Set{h(f_\alpha)}{\alpha<\lambda},$$ then $\chi\leq\kappa_B$,  $\rank{h(f_x)}<\chi$ for all $x\in V_\lambda$ and the map $$\Map{j}{V_\lambda}{V_\chi}{x}{h(f_x)}$$ is a $\Sigma_1$-elementary embedding. 
  \end{enumerate}
\end{lemma}

\begin{proof}
 \eqref{item:Prod5} First, note that the fact that $f_\kappa\in X$ with $h(f_\kappa)=\dot{u}^B$ directly implies that $\lambda\leq\kappa$. 
 
 Next,   assume, towards a contradiction, that $h(f_\lambda)\neq\dot{u}^B$. Then the above computations show that  $\lambda <\kappa$. Pick  a set $x\in V_\kappa$ with  $\rank{x}=\lambda$ and $h(f_x)=\dot{u}^B$.
 %
 %This implies that  $$B\models h(f_\lambda)\neq\dot{u} ~ \wedge ~ h(f_x)=\dot{u}.$$
 % In this situation, 
 An application of Lemma \ref{lemma:Prod1}.\eqref{item:Prod1} now yields an element $A$ of $\Ce\cap V_\zeta$ with $f_\lambda(A)\neq\dot{u}^A$ and $f_x(A)=\dot{u}^A$. But then $\lambda<\kappa_A$ and $x\notin V_{\kappa_A}$, a contradiction. 
 
 Now, assume, towards a contradiction, that there is an ordinal $\alpha$ with $\lambda=\alpha+1$. Then $h(f_\alpha)\neq\dot{u}^B$ and $\lambda<\kappa_A$ holds for all $A\in\Ce\cap V_\zeta$ with $\alpha<\kappa_A$. By Lemma \ref{lemma:Prod1}.\eqref{item:Prod1}, this implies that $h(f_\lambda)\neq\dot{u}^B$, a contradiction.

    \eqref{item:Prod7} Assume that there exists $x\in V_\lambda$ with $h(f_x)\neq x$. 
    Let $y\in V_\lambda$ be rank-minimal with $h(f_y)\neq y$. Set $\alpha=\rank{y}\leq\rank{x}<\lambda\leq\kappa$. Then $h(f_\alpha)\neq \dot{u}^B$, and Lemma \ref{lemma:Prod1}.\eqref{item:Prod3} shows that $h(f_\alpha)<\kappa_B$. 
    Note that, since $y\in V_\lambda$, $h(f_y)\ne \dot{u}^B$. 
    So, since we have $$A\models `` f_y(A)\neq\dot{u}^A ~ \longrightarrow ~ f_\alpha(A)=\rank{f_y(A)}"$$ for all $A\in\Ce\cap V_\zeta$,  this implies that $h(f_\alpha)=\rank{h(f_y)}$. 
    
    Assume, towards a contradiction, that $h(f_\alpha)\leq\alpha$. Given $z\in y$, we then have $$A\models `` f_y(A)\neq\dot{u}^A\longrightarrow f_z(A) \in  f_y(A)"$$ for all $A\in\Ce\cap V_\zeta$, and this shows that  $h(f_z)\in h(f_y)$. By the minimality of $y$, this implies that $y\subseteq h(f_y)$ and hence there exists $w\in h(f_y)\setminus y$.  
    We now know that  $$\rank{w} ~ < ~ \rank{h(f_y)} ~ = ~ h(f_\alpha) ~ \leq ~ \alpha ~ = ~ \rank{y}$$ and therefore the minimality of $y$ implies that $h(f_w)=w\in h(f_y)$. 
    Hence there exists $A$ in $\Ce\cap V_\zeta$ with $f_y(A)\neq \dot{u}^A$ and $f_w(A)\in f_y(A)$. But then $\rank{w}<\rank{y}<\kappa_A$ and hence we can conclude that $w=f_w(A)\in f_y(A)=y$, a contradiction.

   \eqref{item:ProdBonus} First, note that  Lemma \ref{lemma:Prod1}.\eqref{item:Prod3}  directly implies that $\chi\leq\kappa_B$. 
       Since Lemma \ref{lemma:Prod1}.\eqref{item:Prod1} shows that $\rank{h(f_x)}=h(f_{\rank{x}})<\chi$ holds for all $x\in V_\lambda$,  we know that the  function $\map{j}{V_\lambda}{V_\chi}$ is well-defined.  
 
  In the following, fix a $\Sigma_1$-formula $\varphi(v)$ and an element $x$ of $V_\lambda$. 
    First, assume that $\varphi(x)$ holds in $V_\lambda$. 
    Since \eqref{item:Prod5} shows that $\lambda$ is a limit ordinal, we can use the fact that $\Delta_0$-formulas are absolute between transitive structures to find an ordinal  $\alpha<\lambda$ with the property that $x\in V_\alpha$ and $\varphi(x)$ holds in $V_\alpha$. 
      By $\Sigma_1$-upwards absoluteness,  
  this shows that we have  $$A\models ``  f_\alpha(A)\neq\dot{u}^A\longrightarrow  V_{f_\alpha(A)}\models \varphi(f_x(A))"$$  for all $A\in\Ce\cap V_\zeta$. 
  Using Lemma \ref{lemma:Prod1}.\eqref{item:Prod1}, we now know that $\varphi(j(x))$ holds in $V_{j(\alpha)}$ and therefore $\Sigma_1$-upwards absoluteness implies that $\varphi(j(x))$ holds in $V_\chi$. 
    In the other direction, assume that $\varphi(j(x))$ holds in $V_\chi$. 
   Since our computations already show that $j(\alpha)<j(\beta)$ holds for all $\alpha<\beta<\lambda$, we know that $\chi=\lub\Set{h(f_\alpha)}{\alpha<\lambda}$ and hence we can find $\alpha<\lambda$ such that $\varphi(j(x))$ holds in $V_{j(\alpha)}$. 
   Another application of Lemma \ref{lemma:Prod1}.\eqref{item:Prod1} then shows that $\varphi(x)$ holds in $V_\alpha$ and $\Sigma_1$-upwards absoluteness allows us to conclude that this statement holds in $V_\lambda$. 
\end{proof}

%%%%%%%%%%%%%%%%%%
%%%%%%%%%%%%%%%%%%

\section{Strongly unfoldable product reflection}
\label{section7}

We shall give next a proof of Theorem \ref{theorem:CharWPR}. Namely, 
we will show that 
 the following  are equivalent for every  cardinal $\kappa$: 
 \begin{enumerate}
   \item $\kappa$ is either strongly unfoldable or a limit of strong  cardinals. 
   
  \item The principle $\WPR_\Ce(\kappa)$ holds for every class $\Ce$ of structures of the same type that is definable by a $\Sigma_2$-formula with parameters in $V_\kappa$. 
 \end{enumerate}

The implication $(1)\Rightarrow (2)$ follows from a combination of the following lemma and Theorem \ref{theorem:StrongProductSigma2}.

\begin{lemma}\label{lemma:PRfromStrongUnfoldability}
 If $\kappa$ is a strongly unfoldable cardinal, then  $\WPR_\Ce(\kappa)$ holds for every  non-empty class $\Ce$ of structures of the same type that is definable by a $\Sigma_2$-formula with parameters in $V_\kappa$. 
\end{lemma}

\begin{proof}
 First, note that, since  all strongly unfoldable cardinals are elements of $C^{(2)}$ (see Theorem \ref{theorem:CharSr-Sr--}), Lemma \ref{lemma:Sigma2Reflex} shows that $\Ce\cap V_\kappa\neq\emptyset$. 
 Now, fix a substructure $X$ of $\prod(\Ce\cap V_\kappa)$ of cardinality at most $\kappa$ and a structure $B$ in $\Ce$. 
 Pick a cardinal $\delta\in C^{(2)}$ greater than $\kappa$, with $B\in V_\delta$, and an elementary submodel $M$ of $H_{\kappa^+}$ of cardinality $\kappa$ with the property that  $V_\kappa\cup\{X\}\cup{}^{{<}\kappa}M\subseteq M$. 
 Since $\kappa\in C^{(2)}$, we know that $\Ce\cap V_\kappa\in M$. 
 Using the strong unfoldability of $\kappa$, we can find a transitive set $N$ with $V_\delta\subseteq N$ and an elementary embedding $\map{j}{M}{N}$ with $\crit{j}=\kappa$ and $j(\kappa)>\delta$. 
 Now notice that since $\delta$ is in $C^{(2)}$, and therefore  in $C^{(1)}$ in the sense of $N$, and since $j(\kappa)$ is an inaccessible cardinal  in  $N$,  we have that $V_\delta\prec_{\Sigma_1} V^N_{j(\kappa)}$, which implies $\Ce\cap V_\delta\subseteq j(\Ce\cap V_\kappa)$. 
 Thus, $B\in\Ce\cap V_\delta\subseteq j(\Ce\cap V_\kappa)$, and the function $$\Map{h}{X}{B}{f}{j(f)(B)}$$ is a well-defined homomorphism. 
  \end{proof}

\begin{proof}[Proof of Theorem \ref{theorem:CharWPR}]
 Let $\kappa$ be a cardinal with the property that $\WPR_\Ce(\kappa)$ holds for every class $\Ce$ of structures of the same type that is definable by a $\Sigma_2$-formula with parameters in $V_\kappa$. 
 Then Lemma \ref{lemma:Sigma2Reflex}  shows that $\kappa$ is an element of $C^{(2)}$. 
 Assume, towards a contradiction, that $\kappa$ is neither strongly unfoldable nor a limit of strong cardinals. 
 Pick an ordinal $\alpha<\kappa$ such that the interval $[\alpha,\kappa)$ contains no strong cardinals. 
 Given a cardinal $\rho$ that is not strong, we let $\eta_\rho$ denote the least cardinal $\delta>\rho$ such that $\rho$ is not $\delta$-strong. Since the class of ordinals that are not strong cardinal is definable by a $\Sigma_2$-formula without parameters, the fact that $\kappa \in C^{(2)}$ implies that the interval $(\alpha,\kappa)$ is closed under the function $\rho\longmapsto \eta_\rho$, and therefore it contains unboundedly many cardinals $\xi$ with the property that $\eta_\rho<\xi$ holds for all cardinals $\alpha\leq\rho<\xi$.  
 Finally, since {\cite[Theorem 1.3]{luecke2021strong}}  implies  that $\kappa$ is not \emph{shrewd}, basic definability considerations allow us to  find an $\calL_\in$-formula $\Phi(v_0,v_1)$, a limit ordinal $\theta>\kappa$ and $E\subseteq V_\kappa$ with the property that $\Phi(\kappa,E)$ holds in $V_{\theta+1}$ and for all $\beta<\gamma<\kappa$, the statement $\Phi(\beta,E\cap V_\beta)$ does not hold in $V_{\gamma+1}$. 
 %
% Without loss of generality, we may assume that $\theta$ is a limit ordinal, because for all  ordinals $\beta$, the set $V_\beta$ is definable in $V_{\beta+\omega}$ by a formula without parameters. 

  Let $\calL^\prime$ be the first-order language that extends the language of set theory with  
  a binary relation symbol $\dot{S}$,  constant symbols $\dot{\kappa}$, $\dot{u}$, and  $\dot{c}$, and a unary function symbol $\dot{e}$. Let $\calL$ denote the first-order language that extends $\calL^\prime$ by an $(n+1)$-ary predicate symbol $\dot{T}_\varphi$ for  every $\calL^\prime$-formula $\varphi(v_0,\ldots,v_n)$ with $(n+1)$-many free variables, and let 
  $\calS_\calL$  be the class of structures as defined at the beginning of Section \ref{prodref}. Namely,  $\calS_\calL$ is the class of all $\calL$-structures $A$ such that there exists a cardinal $\kappa_A$ in $C^{(1)}$ and a limit  ordinal $\theta_A>\kappa_A$ 
 %transitive set $M_A$ with the property 
 such that the following  hold: 
 \begin{enumerate}
     %\item %$M_A$ has cardinality $2^{\kappa_A}$ \todo[fancyline]{Needed?} and 
     %$V_{\kappa_A+1}\subseteq M_A$ and for some ordinal $\vartheta>\kappa_A$, there exists an elementary  embedding $\map{j}{M_A}{V_\vartheta}$ with $j\restriction V_{\kappa_A+1}=\id_{V_{\kappa_A+1}}$. 
     
     \item The domain of $A$ is $V_{\theta_A +1}$. 
     
     \item $\in^A={\in}\restriction{V_{\theta_A+1}}$, $\dot{\kappa}^A=\kappa_A$ and $\dot{u}^A=\theta_A$. 
     
     \item If $\varphi(v_0,\ldots,v_n)$ is an $\calL^\prime$-formula, then $$\dot{T}_\varphi^A ~ = ~ \Set{\langle x_0,\ldots,x_n\rangle\in V_{\theta_A +1}^{n+1}}{A\models\varphi(x_0,\ldots,x_n)}.$$
     %$n$ is a natural number and $A^-$ denotes the $\calL$-retract of $A$, then  $$\dot{T}_n^{A} ~ = ~ \Set{\langle a, x_0,\ldots,x_{n-1}\rangle\in M_A^{n+1}}{a\in\mathsf{Fml}_\calL, ~ \mathsf{Sat}_\calL(A^-,a,\langle x_0,\ldots,x_{n-1}\rangle)},$$ where $\mathsf{Fml}_\calL\subseteq V_\omega$ denotes the set of formalized $\calL$-formulas and $\mathsf{Sat_\calL}$ denotes the corresponding formalized satisfaction relation.  
 \end{enumerate}
 
 Let $\Ce$ denote the class of all $A\in\calS_\calL$ such that the following  hold: 
 \begin{itemize}
  \item $\dot{c}^A=\alpha<\kappa_A$. 
  
  \item The interval $[\alpha,\kappa_A)$ contains no strong cardinals. 
  
  %\item $\dot{E}^A\subseteq V_{\kappa_A}$. 
  
  \item $\dot{S}^A=\Set{\langle\rho,\gamma\rangle\in\kappa_A\times\kappa_A}{\textit{$\rho$ is a $\gamma$-strong cardinal}}$. 
  
  \item If $\alpha<\delta<\kappa_A$ is a cardinal, then $\dot{e}^A(\delta)$ is a cardinal below $\kappa_A$ and  is the smallest cardinal $\xi$ greater than $\delta$ that has the property that $\eta_\rho<\xi$ holds for all cardinals $\alpha\leq\rho<\xi$. 
   %there exists a sequence $\seq{\rho_n}{n<\omega}$ with $\rho_0=\rho$, $\dot{e}^A(\rho)=\sup_{n<\omega}\rho_n<\kappa_A$ and $\rho_{n+1}=\eta_{\rho_n}$ for all $n<\omega$. 
 \end{itemize}
 
 It is easily seen that the class $\Ce$ is definable by a $\Sigma_2$-formula with parameter $\alpha$. 
 In addition, the fact that  $\kappa$ is an element of $C^{(2)}$ implies that $\sup\Set{\kappa_A}{A\in\Ce\cap V_\kappa}=\kappa$ and there exists a structure $B$ in $\Ce$ with $\kappa_B=\kappa$ and $\theta_B=\theta$. % and $\dot{E}^B=E$. 
 Let $C$ be a cofinal subset of $\kappa$ of order-type $\cof{\kappa}$. 
 Given a set $x$, we define  functions $f_x$ and $f^x$ with domain $\Ce \cap V_\kappa$ as in Section \ref{prodref}. Namely, we have $f^x(A)=x\cap V_{\kappa_A}$ for all $A\in\Ce\cap V_\kappa$, $f_x(A)=x$ for all $A\in\Ce\cap V_\kappa$ with $x\in V_{\kappa_A}$ and $f_x(A)=\dot{u}^A$ for all $A\in\Ce\cap V_\kappa$ with $x\notin V_{\kappa_A}$. 
 Since $\kappa$ is an element of $C^{(1)}$, we can find a substructure  $X$ of $\prod(\Ce\cap V_\kappa)$ of cardinality $\kappa$ with the property that $f^E\in X$, $f^C\in X$,  and $f_x,f^x\in X$ for all $x\in V_\kappa$. 
 By our assumptions, there exists a homomorphism $\map{h}{X}{B}$ and we can  define $$\lambda ~ = ~ \min\Set{\rank{x}}{f_x\in X, ~ h(f_x)=\dot{u}^B}$$ and $$\chi ~ = ~ \sup\Set{h(f_\beta)}{\beta<\lambda}.$$
 Lemma \ref{lemma:Prod2} then shows  that both $\lambda$ and $\chi$ are less than or equal to $\kappa$. Moreover, Lemma \ref{lemma:Prod2}.\eqref{item:Prod5} shows that $h(f_\lambda)=\dot{u}^B\neq\lambda$, and this implies that   $\alpha<\lambda$, because  $f_\alpha(A)\neq\dot{u}^A$ holds for all $A\in\Ce\cap V_\kappa$ and, by Lemma \ref{lemma:Prod1}.\eqref{item:Prod1}, this implies that $h(f_\alpha)\neq\dot{u}^B$. 
 Now, let $$\mu ~ = ~ \min\Set{\beta\leq\lambda}{h(f_\beta)\neq\beta}.$$
 Then Lemma  \ref{lemma:Prod2}.\eqref{item:Prod7} implies that $h(f_x)=x$ holds for all $x\in V_\mu$. 
 Moreover, since $\kappa_A<\kappa$ holds for all $A\in\Ce\cap V_\kappa$, 
 we can apply Lemma \ref{lemma:ProdCons}.\eqref{item:ProdCons2.2New} to show that $\mu$ is an inaccessible cardinal.

 \begin{claim*}
  $\mu<\kappa$. 
 \end{claim*}
 
 \begin{proof}[Proof of the Claim]
  Assume, towards a contradiction, that $\mu=\lambda=\kappa$ holds. 
  Since Lemma \ref{lemma:Prod1}.\eqref{item:Prod1} implies  $h(f^E)\subseteq V_\kappa$, we can apply Lemma \ref{lemma:ProdCons}.\eqref{item:ProdCons1} to conclude that $h(f^E)=E$ and hence $$B\models \Phi(\dot{\kappa}^B,h(f^E)).$$ 
  Another application of Lemma \ref{lemma:Prod1}.\eqref{item:Prod1} then yields $A\in\Ce\cap V_\kappa$ with $$A\models\Phi(\dot{\kappa}^A,f^E(A))$$ and this shows that $\Phi(\kappa_A,E\cap V_{\kappa_A})$ holds in $V_{\theta_A+1}$. 
  Since $\kappa_A<\theta_A<\kappa$, this contradicts the fact that $E$ witnesses that $\kappa$ is not a shrewd cardinal. 
  \end{proof}

    \begin{claim*}
   $\mu<\lambda$. 
  \end{claim*}
  
  \begin{proof}[Proof of the Claim]
   Assume, towards a contradiction, that $\mu=\lambda<\kappa$. 
   Then Lemma \ref{lemma:ProdCons}.\eqref{item:ProdCons3} shows that $h(f^\mu)=\kappa$ and we can apply Lemma \ref{Lemma:HomStrongA} to show that $\mu$ is a $\kappa$-strong cardinal. 
 But this contradicts the fact that $\alpha<\mu<\eta_\mu<\kappa$. 
  \end{proof}

 The above claim shows that $\mu<\lambda\leq\kappa$ and $h(f_\mu)\neq\dot{u}^B$. 
 We let $$\Map{j}{V_\lambda}{V_\chi}{x}{h(f_x)}$$ denote the non-trivial $\Sigma_1$-elementary embedding with $\crit{j}=\mu$ that was introduced in Lemma \ref{lemma:Prod2}.\eqref{item:ProdBonus}. 
 As in the proof of Theorem \ref{theorem:StrongProductSigma2}, we can now use the $\Sigma_1$-elementarity of $j$ and the \emph{Kunen Inconsistency} to conclude that    $\alpha<\mu$. 
 Since $\alpha<\mu<\lambda\leq\kappa$, we can now pick a cardinal $\mu<\xi<\kappa$ that is the minimal cardinal above $\mu$ with the property that $\eta_\rho<\xi$ holds for all cardinals $\alpha\leq\rho<\xi$. 
  Given $A\in\Ce\cap V_\kappa$ with $\mu<\kappa_A$, we then have $\xi=\dot{e}^A(\mu)<\kappa_A$. Using Lemma \ref{lemma:Prod1}.\eqref{item:Prod1}, this shows that $h(f_\xi)\neq\dot{u}^B$ and  $\xi<\lambda$.

  \begin{claim*}
   If $n<\omega$, then $j^n(\mu)<j^{n+1}(\mu)<\xi$ and $j^n(\mu)$ is a $j^{n+1}(\mu)$-strong cardinal. 
  \end{claim*}
  
   \begin{proof}[Proof of the Claim]
    Since $\mu<\kappa$ and Lemma \ref{lemma:ProdCons}.\eqref{item:ProdCons4} shows that $j(\mu)=h(f^\mu)$,  we can apply Lemma \ref{Lemma:HomStrongA} to conclude that $\mu$ is $j(\mu)$-strong and this implies that $\mu<j(\mu)<\eta_\mu<\xi$. 
   Now, assume that for some $n<\omega$, we have $j^n(\mu)<j^{n+1}(\mu)<\xi$ and $j^n(\mu)$ is a $j^{n+1}(\mu)$-strong cardinal. Then $\langle j^n(\mu),j^{n+1}(\mu)\rangle\in\dot{S}^A$ for all $A\in\Ce\cap V_\kappa$ with $\xi<\kappa_A$ and the fact that $h(f_\xi)\neq\dot{u}^B$ allows us to use Lemma \ref{lemma:Prod1}.\eqref{item:Prod1} to show that $\langle j^{n+1}(\mu),j^{n+2}(\mu)\rangle\in\dot{S}^B$. Hence, we know that $j^{n+1}(\mu)$ is a $j^{n+2}(\mu)$-strong cardinal and $j^{n+1}(\mu)<j^{n+2}(\mu)<\eta_{j^{n+1}(\mu)}<\xi$. 
   \end{proof}
   
   We can now define $$\tau ~ = ~ \sup_{n<\omega}j^n(\mu) ~ \leq ~ \xi ~ < ~ \lambda,$$ apply Lemma  \ref{lemma:Prod2}.\eqref{item:Prod5} to show that $\lambda$ is a limit ordinal, and use  the $\Sigma_1$-elementarity of $j$ to conclude that $j(V_{\tau+2})=V_{\tau+2}$.  
  Since this entails  that $\map{j\restriction V_{\tau+2}}{V_{\tau+2}}{V_{\tau+2}}$ is a non-trivial elementary embedding, we again derived a contradiction to the \emph{Kunen Inconsistency}.

  The above computations yield a proof of the implication $(2)\Rightarrow (1)$ of Theorem \ref{theorem:CharWPR}. The converse implication $(1)\Rightarrow (2)$  follows directly from a combination of  Theorem \ref{theorem:StrongProductSigma2} and Lemma \ref{lemma:PRfromStrongUnfoldability}.  
\end{proof}

%%%%%%%%%%%%%%%%%%%%%%%%%
%%%%%%%%%%%%%%%%%%%%%%%%%

\section{A Lemma about Weak Product Structural Reflection}\label{section:7}

Recall (see Definition \ref{wpsr}) that for a class $\Ce$ of structures of the same type and a cardinal $\kappa$, the principle  $\WPR_\Ce(\kappa)$ asserts that  $\Ce\cap V_\kappa\neq\emptyset$ and for every substructure $X$ of $\prod(\Ce\cap V_\kappa)$ of cardinality at most  $\kappa$ and every $B\in\Ce$, there exists a homomorphism from $X$ to $B$. We now prove a lemma that will be used in the proof of Theorem \ref{main:ProdSubtle}, given in the next section.

\begin{lemma}\label{lemma:WPRimpliesInaccessible}
 Let $\delta$ be an uncountable cardinal with the property that for every set $\Ce$ of structures of the same type with $\Ce\subseteq V_\delta$, there exists a cardinal $\kappa<\delta$ such that $\WPR_\Ce(\kappa)$ holds. Then $\delta$ is inaccessible. 
\end{lemma}

\begin{proof}
 We start by proving a series of claims.
 
 \begin{claim*}
  $\delta$ is a limit cardinal. 
 \end{claim*}
 
 \begin{proof}[Proof of the Claim]
  Assume, towards a contradiction, that there exists a cardinal $\gamma<\delta$ satisfying $\gamma^+=\delta$. Let $\calL$ denote the trivial first-order language, let $A$ be the $\calL$-structure with domain $V_\gamma$ and set $\Ce=\{A\}$. Then $\Ce\cap V_\kappa=\emptyset$ for all cardinals $\kappa<\delta$, contradicting our assumption. 
 \end{proof}

 \begin{claim*}
  $\cof{\delta}>\omega$. 
 \end{claim*}
 
 \begin{proof}[Proof of the Claim]
   Assume, towards a contradiction, that $\cof{\delta}$ is countable. 
 Pick a strictly increasing sequence $\seq{\delta_n}{n<\omega}$ of cardinals that is cofinal in $\delta$ and let  $\calL$ denote the first-order language that extends the language of group theory by a constant symbol $\dot{g}$. 
 Given $1<n<\omega$, fix an $\calL$-structure $G_n$ such that $\delta_n<\rank{G_n}<\delta_{n+1}$ and the reduct of $G_n$ to the language of group theory is the sum  of $\delta_n$-many copies of the cyclic group of order $n$ and $\dot{g}^{G_n}$ is an element of order $n$ in this group. Set $\Ce=\Set{G_n}{1<n<\omega}\subseteq V_\delta$. 
 Then there exists a cardinal $\kappa<\delta$ with the property that $\WPR_\Ce(\kappa)$ holds. 
 Let $X$ be a substructure of $\prod(\Ce\cap V_\kappa)$ of cardinality at most $\kappa$. Pick a prime number $p$ with $\delta_p>\kappa$. Then our assumption yields a homomorphism $\map{h}{X}{G_p}$ and our setup ensures that $(\dot{g}^X)^{(p-1)!}$ is the neutral element of $X$. But this implies that $(\dot{g}^{G_p})^{(p-1)!}$ is the neutral element of $G_p$, contradicting the fact that $\dot{g}^{G_p}$ has order $p$ in $G_p$.  
 \end{proof}

 \begin{claim*}
  $\betrag{V_{\cof{\delta}}}\geq\delta$. 
 \end{claim*}

 \begin{proof}[Proof of the Claim]
  Assume, towards a contradiction, that $\betrag{V_{\cof{\delta}}}<\delta$. Then $\cof{\delta}<\delta$ and we can pick a strictly increasing sequence $\seq{\delta_\xi}{\xi<\cof{\delta}}$ of cardinals greater than $\betrag{V_{\cof{\delta}}}$ that is cofinal in $\delta$. 
  Let $\calL$ denote the first-order language that extends $\calL_\in$ by a constant symbol $\dot{c}$, a unary relation symbol $\dot{M}$ and an $(n+1)$-ary relation symbol $\dot{R}_\varphi$ for every $\calL_\in$-formula $\varphi\equiv\varphi(v_0,\ldots,v_n)$ with $n+1$ free variables. 
  Given $\xi<\cof{\delta}$, let $A_\xi$ denote  the unique $\calL$-structure with $\calL_\in$-reduct $\langle V_{\delta_\xi},\in\rangle$ that satisfies $\dot{c}^{A_\xi}=\xi$, $\dot{M}^{A_\xi}=V_{\cof{\delta}}$ and $$\dot{R}_\varphi^{A_\xi} ~ = ~ \Set{\langle x_0,\ldots,x_n\rangle\in V_{\cof{\delta}}^{n+1}}{V_{\cof{\delta}}\models\varphi(x_0,\ldots,x_n)}$$ for every $\calL_\in$-formula $\varphi\equiv\varphi(v_0,\ldots,v_n)$. 
    Set $\Ce=\Set{A_\xi}{\xi<\cof{\delta}}\subseteq V_\delta$. Then there exists a cardinal $\kappa<\delta$ with the property that $\WPR_\Ce(\kappa)$ holds and  the fact that $\Ce\cap V_\kappa\neq\emptyset$ directly implies that $\kappa>\betrag{V_{\cof{\delta}}}$. 
   Fix $\zeta<\cof{\delta}$ with the property that $\delta_\zeta>\kappa$. 
  Given $x\in V_{\cof{\delta}}$, let $f_x$ denote the unique function with domain $\Ce\cap V_\kappa$ and $f_x(A)=x$ for all $A\in\Ce\cap V_\kappa$. 
  Since $\betrag{V_{\cof{\delta}}}<\kappa$, we can find a substructure $X$ of $\prod(\Ce\cap V_\kappa)$ with $f_x\in X$ for all $x\in V_{\cof{\delta}}$. 
   Our assumptions then yield a homomorphism $\map{h}{X}{A_\zeta}$ and, as in the proof of Lemma \ref{lemma:Prod1}.\eqref{item:Prod1}, one can now prove the following statement. 
  
 \begin{subclaim*}
  If $\varphi(v_0,\ldots,v_{n-1})$ is an $\calL_\in$-formula and $g_0,\ldots,g_{n-1}\in X$ with the property that $g_0(A),\ldots,g_{n-1}(A)\in V_{\cof{\delta}}$ and $V_{\cof{\delta}}\models\varphi(g_0(A),\ldots,g_{n-1}(A))$ hold for all $A\in\Ce\cap V_\kappa$, then $h(g_0),\ldots,h(g_{n-1})\in V_{\cof{\delta}}$ and $V_{\cof{\delta}}\models\varphi(h(g_0),\ldots,h(g_{n-1}))$. \qed
 \end{subclaim*}

  If we now define $$\Map{j}{V_{\cof{\delta}}}{V_{\cof{\delta}}}{x}{h(f_x)},$$ then the above claim shows that $j$ is an elementary embedding. 
 Moreover, since $$V_{\cof{\delta}}\models f_\zeta(A)\neq\dot{c}^A$$ holds for all $A\in\Ce\cap V_\kappa$, the subclaim shows that $j(\zeta)\neq\dot{c}^{A_\zeta}=\zeta$ and this allows us to conclude that $j$ is a non-trivial embedding. 
 But this contradicts the \emph{Kunen Inconsistency}, because $\cof{\delta}$ is an uncountable regular cardinal.  
 \end{proof}

  \begin{claim*}
  $\betrag{V_\xi}<\delta$ for all $\xi<\cof{\delta}$. 
 \end{claim*}

 \begin{proof}[Proof of the Claim]
  Assume, towards a contradiction, that the statement of the claim fails and let $\xi<\cof{\delta}$ be minimal with $\betrag{V_\xi}\geq\delta$. 
  The minimality of $\xi$ then yields an ordinal $\eta$ with $\xi=\eta+1$ and $\betrag{V_\eta}<\delta$. Fix an injective enumeration $\seq{x_\gamma}{\gamma<\delta}$ of subsets of $V_\eta$. 
  Moreover, if there exists a limit cardinal $\lambda$ of countable cofinality with $\eta\in\{\lambda,\lambda+1\}$, then we also fix a cofinal function $\map{d}{\omega}{\lambda}$. 
    Let $\calL$ denote the first-order language that extends $\calL_\in$ by a constant symbol $\dot{c}$, a constant symbol $\dot{d}_n$ for every natural number $n$, a unary relation symbol $\dot{M}$ and an $(n+1)$-ary relation symbol $\dot{R}_\varphi$ for every $\calL_\in$-formula $\varphi\equiv\varphi(v_0,\ldots,v_n)$ with $n+1$ free variables. 
    Given $\gamma<\delta$, let $A_\gamma$ denote an $\calL$-structure such that the following statements hold: 
    \begin{itemize}
     \item The $\calL_\in$-reduct of $A_\gamma$ is of the form $\langle V_\rho,\in\rangle$ for some cardinal $\max(\gamma,\betrag{V_\eta})<\rho<\delta$. 
      
     \item $\dot{c}^{A_\gamma}=x_\gamma$ and $\dot{M}^{A_\gamma}=V_\eta$. 
     
     \item If $\eta\in\{\lambda,\lambda+1\}$ for a limit cardinal $\lambda$ of countable cofinality, then $\dot{d}_n^{A_\gamma}=d(n)$. 
     
     \item If $\varphi\equiv\varphi(v_0,\ldots,v_n)$ is an $\calL_\in$-formula, then $$\dot{R}_\varphi^{A_\gamma} ~ = ~ \Set{\langle z_0,\ldots,z_n\rangle\in V_{\eta+1}^{n+1}}{V_{\eta+1}\models\varphi(z_0,\ldots,z_n)}.$$ 
    \end{itemize}
      Define $\Ce=\Set{A_\gamma}{\gamma<\delta}\subseteq V_\delta$ and pick a cardinal $\kappa<\delta$ with the property that $\WPR_\Ce(\kappa)$ holds. 
      Then $\kappa>\betrag{V_\eta}$, because $\Ce \cap V_\kappa$ is non-empty and the domain of every element of $\Ce$ is some $V_\rho$ with $\rho$ greater than $|V_\eta|$. Given $x\in V_{\eta+1}$, we define $f_x$ to be the unique function with domain $\Ce\cap V_\kappa$ and $f_x(A)=x$ for all $A\in\Ce\cap V_\kappa$. Then there exists a substructure $X$ of $\prod(\Ce\cap V_\kappa)$ of cardinality at most $\kappa$ with $f_{V_\eta},f_{x_\kappa}\in X$ and $f_x\in X$ for all $x\in V_\eta$. Moreover, our assumption yields a homomorphism $\map{h}{X}{A_\kappa}$. 
Then we again know that for every $\calL_\in$-formula $\varphi(v_0,\ldots,v_{n-1})$ and all $g_0,\ldots,g_{n-1}\in X$ such that $g_0(A),\ldots,g_{n-1}(A)\in V_{\eta+1}$ and $V_{\eta+1}\models\varphi(g_0(A),\ldots,g_{n-1}(A))$ hold for all $A\in\Ce\cap V_\kappa$, we have $h(g_0),\ldots,h(g_{n-1})\in V_{\eta+1}$ and $V_{\eta+1}\models\varphi(h(g_0),\ldots,h(g_{n-1}))$. 
In particular, we know that $h(f_{V_\eta})=V_\eta$ and the induced map $$\Map{j}{V_\eta}{V_\eta}{x}{h(f_x)}$$ is an elementary embedding.

    \begin{subclaim*}
     The embedding $j$ is non-trivial. 
    \end{subclaim*}
    
    \begin{proof}[Proof of the Subclaim]
     Assume, towards a contradiction, that $j=\id_{V_\eta}$. We then have $\dot{c}^A\neq x_\kappa=f_{x_\kappa}(A)$ for all $A\in\Ce\cap V_\kappa$ and hence the above observations show that $h(f_{x_\kappa})\neq \dot{c}^{A_\kappa}=x_\kappa$. Pick an element $x\in V_\eta$ that is contained in the symmetric difference of $x_\kappa$ and $h(f_{x_\kappa})$. Since $h(f_x)=x$ holds, our earlier computations imply that $x$ is an element of $x_\kappa$ if and only if $x$ is an element of $h(f_{x_\kappa})$, a contradiction. 
    \end{proof}
    
  If we now define $\lambda=\sup_{n<\omega} j^n(\crit{j})$, then $\lambda$ is a strong limit cardinal of countable cofinality and the \emph{Kunen Inconsistency} implies that $\eta\in\{\lambda,\lambda+1\}$. 
  But this shows that for some cofinal function $\map{d}{\omega}{\lambda}$, we have $\dot{d}_n^A=d(n)$ for all $n<\omega$ and $A\in\Ce$. In particular, this implies that $j(d(n))=\dot{d}_n^{A_\kappa}=d(n)$ holds for all $n<\omega$, contradicting the fact that $\eta\leq\lambda+1$.  
 \end{proof}

 \begin{claim*}
  The cardinal $\delta$ is regular. 
 \end{claim*}
 
 \begin{proof}[Proof of the Claim]
  Assume, towards a contradiction, that $\delta$ is singular. Since our first claim shows that  $\betrag{V_{\cof{\delta}}}\geq\delta>\cof{\delta}$,  we can now find an ordinal $\eta<\cof{\delta}$ with $\betrag{V_\eta}\geq\cof{\delta}$. 
  Fix an injective sequence $\seq{x_\xi}{\xi<\cof{\delta}}$ of elements of $V_\eta$. By our second claim, we can also pick a strictly increasing sequence $\seq{\delta_\xi}{\xi<\cof{\delta}}$ of cardinals greater than $\betrag{V_{\eta+2}}$ that is cofinal in $\delta$. 
 We let  $\calL$ denote the first-order language extending $\calL_\in$ by a constant symbol $\dot{c}$, a unary relation symbol $\dot{M}$ and an $(n+1)$-ary relation symbol $\dot{R}_\varphi$ for every $\calL_\in$-formula $\varphi\equiv\varphi(v_0,\ldots,v_n)$ with $n+1$ free variables. 
  For every $\xi<\cof{\delta}$, we let $A_\xi$ denote  the unique $\calL$-structure with $\calL_\in$-reduct $\langle V_{\delta_\xi},\in\rangle$ such that $\dot{c}^{A_\xi}=x_\xi$, $\dot{M}^{A_\xi}=V_{\eta+2}$ and $$\dot{R}_\varphi^{A_\xi} ~ = ~ \Set{\langle y_0,\ldots,y_n\rangle\in V_{\eta+2}^{n+1}}{V_{\eta+2}\models\varphi(y_0,\ldots,y_n)}$$ for every $\calL_\in$-formula $\varphi\equiv\varphi(v_0,\ldots,v_n)$. 
 Set $\Ce=\Set{A_\xi}{\xi<\cof{\delta}}\subseteq V_\delta$, pick a cardinal $\kappa<\delta$ with the property that $\WPR_\Ce(\kappa)$ holds and fix $\zeta<\cof{\delta}$ with $\delta_\zeta>\kappa$. 
 Given $x\in V_{\eta+2}$, we let $f_x$ denote the unique function with domain $\Ce\cap V_\kappa$ and $f_x(A)=x$ for all $A\in\Ce\cap V_\kappa$. 
 Since our setup ensures that  $\kappa>\betrag{V_{\eta+2}}$, we can now find a homomorphism $\map{h}{X}{A_\zeta}$ for some substructure $X$ of $\prod(\Ce\cap V_\kappa)$ with $f_{x_\zeta}\in X$ and $f_x\in X$ for all $x\in V_{\eta+2}$. 
 As above, we  know that $$\Map{j}{V_{\eta+2}}{V_{\eta+2}}{x}{h(f_x)}$$ is an elementary embedding and, by the \emph{Kunen Inconsistency}, this map is trivial. 
 But then $$h(f_{x_\zeta}) ~ = ~ j(x_\zeta) ~ = ~ x_\zeta ~ = ~ \dot{c}^{A_\zeta}$$ and there exists $A\in\Ce\cap V_\kappa$ with $x_{\zeta}=f_{x_\zeta}(A)=\dot{c}^A$, a contradiction.  
 \end{proof}
 
 The above arguments show that $\delta$ is a regular cardinal with $\betrag{V_\gamma}<\delta$ for all $\gamma<\delta$ and this directly implies that $\delta$ is inaccessible. 
\end{proof}

We show how the above lemma can be  combined with results in \cite{BagariaWilsonRefl} to prove that the following statements are equivalent for every uncountable cardinal $\delta$: 
\begin{enumerate}
  \item  $\delta$ is a Woodin cardinal. 
  
  \item For every set $\Ce$ of structures of the same type with $\Ce\subseteq V_\delta$, there exists a cardinal $\kappa<\delta$ with the property that the principle $\PR_\Ce(\kappa)$ holds. 
 \end{enumerate}

\begin{proof}[Proof of Theorem \ref{main:Woodin}]
 $(1) \Rightarrow (2)$: Let $\delta$ be a Woodin cardinal and let  $\Ce\subseteq V_\delta$ be a set of structures of the same type.  
  Using {\cite[Theorem 26.14]{MR1994835}}, we can  find a cardinal $\kappa<\delta$ that is \emph{$\gamma$-strong for $\Ce$} for all $\kappa<\gamma<\delta$, {i.e.} for all $\kappa<\gamma<\delta$, there exists a transitive class $M$ with $V_\gamma\subseteq M$ and an elementary embedding $\map{j}{V}{M}$ with $\crit{j}=\kappa$, $j(\kappa)>\gamma$ and $j(\Ce)\cap V_\gamma=\Ce\cap V_\gamma$. 
 Fix $B\in\Ce$ and pick an inaccessible cardinal $\kappa<\gamma<\delta$ with $B\in V_\gamma$ and an elementary embedding  $\map{j}{V}{M}$ with $\crit{j}=\kappa$, $j(\kappa)>\gamma$ and $j(\Ce)\cap V_\gamma=\Ce\cap V_\gamma$. 
 Then we have $B\in \Ce\cap V_\gamma=j(\Ce)\cap V_\gamma\subseteq j(\Ce)\cap V_{j(\kappa)}^M=j(\Ce\cap V_\kappa)$ and the map $$\Map{h}{\prod(\Ce\cap V_\kappa)}{B}{f}{j(f)(B)}$$ is a well-defined homomorphism. 
 
 $(2) \Rightarrow (1)$: Assume that $\delta$ is an uncountable cardinal with the property that for every set $\Ce\subseteq V_\delta$ of structures of the same type, there exists a cardinal $\kappa<\delta$ with the property that the principle $\PR_\Ce(\kappa)$ holds. Then Lemma \ref{lemma:WPRimpliesInaccessible} implies that $\delta$ is inaccessible. A direct adaptation of the proof of \cite[Theorem 5.13]{BagariaWilsonRefl} then shows that $\delta$ is a Woodin cardinal. 
\end{proof}

%%%%%%%%%%%%%%%%%%%%%%%%%%
%%%%%%%%%%%%%%%%%%%%%%%%%%

\section{Subtle product reflection}\label{section8}

We will next give a proof of Theorem \ref{main:ProdSubtle}. Namely, we will show that the following statements are equivalent for every uncountable cardinal $\delta$: 
 \begin{enumerate}
  \item $\delta$ is a subtle cardinal. 
  
  \item For every set $\Ce$ of structures of the same type with $\Ce\subseteq V_\delta$, there exists a cardinal $\kappa<\delta$ with the property that the principle $\WPR_\Ce(\kappa)$ holds. 
 \end{enumerate}

The implication $(1)\Rightarrow (2)$ is given by the following lemma:

\begin{lemma}\label{lemma:SubtleYieldsProdRefl}
If $\delta$ is a subtle cardinal and  $\Ce\subseteq V_\delta$ is a non-empty set of structures of the same type, then there exists an inaccessible cardinal $\kappa<\delta$ such that $\WPR_\Ce(\kappa)$ holds. 
\end{lemma}

\begin{proof}
 Assume, towards a contradiction, that the above conclusion fails. Given an inaccessible cardinal $\kappa<\delta$, fix a substructure $X_\kappa$ of $\prod(\Ce\cap V_\kappa)$ of cardinality $\kappa$ and an ordinal $\xi_\kappa<\delta$ such that there exists a structure $B_\kappa\in\Ce\cap V_{\xi_\kappa}$ with the property that there is no homomorphism from $X_\kappa$ into $B_\kappa$. 
 Then there exists a closed unbounded subset $C$ of $\delta$ that consists of strong limit cardinals and has the property that whenever $\kappa$ is an inaccessible cardinal in $C$, then $\xi_\kappa<\min(C\setminus(\kappa+1))$. 
 In addition, fix a bijection $\map{b_\kappa}{\kappa}{X_\kappa}$ for every inaccessible cardinal $\kappa<\delta$. 
 
 Let $\calL$ denote the signature of $\Ce$ 
 %, let $\mathsf{Sat}_\calL$ denote the  formalized satisfaction relation for $\calL$ 
 and fix an enumeration  $\seq{\varphi_k}{k<\omega}$ of all   $\calL$-formulas. 
 Pick a $\delta$-list $\seq{D_\alpha}{\alpha<\delta}$ such that the following statements hold for all $\alpha<\delta$: 
 \begin{itemize}
  \item If $\alpha$ is inaccessible, then $D_\alpha$ is the set of all elements of $\alpha$ of the form\footnote{Here, we let $\map{\goedel{\cdot}{\ldots,\cdot}}{\On^{n+1}}{\On}$ denote the iterated \emph{G\"odel Pairing Function}.} $\goedel{0,k}{\alpha_0,\ldots,\alpha_{n-1}}$, where $\varphi_k$ is an $n$-ary $\calL$-formula, $\alpha_0,\ldots,\alpha_{n-1}<\alpha$ and $$X_\alpha\models\varphi_k( b_\alpha(\alpha_0),\ldots,b_\alpha(\alpha_{n-1})).$$ 
  
  \item If $\alpha$ is a singular limit cardinal of cofinality $\lambda$, then there exists a cofinal function $\map{c_\alpha}{\lambda}{\alpha}$ with $c_\alpha(0)=0$ and $D_\alpha ~ = ~ \Set{\goedel{1,\lambda}{\xi,c(\xi)}}{\xi<\lambda}$.
 \end{itemize}

 Using the subtleness of $\delta$, we can now find ordinals  $\kappa<\theta$ in $C$ with $D_\kappa=D_\theta\cap\kappa$.

 \begin{claim*}
  The ordinals $\kappa$ and $\theta$ are both inaccessible cardinals. 
 \end{claim*}
 
 \begin{proof}[Proof of the Claim]
  %Assume, towards a contradiction, that the above statement fails. 
  %
  %Then the definition of $C$ implies that one of these  ordinals is a singular strong limit cardinal. 
  %
  First, note that the definition of $C$ implies that both $\kappa$ and $\theta$ are strong limit cardinals. 
  Now, assume that $\kappa$ is a singular strong limit cardinal. Then we have $$\goedel{1}{\cof{\kappa},0,0} ~ \in ~ D_\kappa ~ \subseteq  ~ D_\theta$$ and this shows that $\theta$ is also singular with $\cof{\kappa}=\cof{\theta}$. 
  Moreover, it follows that $c_\kappa=c_\theta$ and hence $\ran{c_\theta}\subseteq\kappa<\theta$, a contradiction. 
  Thus we have shown  that $\kappa$ is inaccessible, and therefore there is some $k<\omega$ such that $\goedel{0}{k}\in D_\kappa\subseteq D_\theta$. 
  By the definition of the list $\seq{D_\alpha}{\alpha<\delta}$ this shows that $\theta$ is also an inaccessible cardinal.  
 \end{proof}

  If we now define  $$\map{b ~ = ~ b_\theta\circ b_\kappa^{{-}1}}{X_\kappa}{X_\theta},$$ then the above setup ensures that $b$ is an elementary embedding. 
  Note that the definition of $C$  implies that $B_\kappa\in\Ce\cap V_\theta$ and therefore we know that the map $$\Map{p}{X_\kappa}{B_\theta}{f}{b(f)(B_\kappa)}$$ is a homomorphism of $\calL$-structures. 
  But this allows us to conclude that $\map{p\circ b}{X_\kappa}{B_\kappa}$ is a homomorphism, contradicting our initial assumptions.  
\end{proof}

\begin{proof}[Proof of Theorem \ref{main:ProdSubtle}]
 Let $\delta$ be an uncountable cardinal with the property that for every set $\Ce$ of structures of the same type with $\Ce\subseteq V_\delta$, there exists a cardinal $\kappa<\delta$ with the property that the principle $\WPR_\Ce(\kappa)$ holds. 
 Lemma \ref{lemma:WPRimpliesInaccessible} then shows that $\delta$ is inaccessible. 
  Assume, towards a contradiction, that $\delta$ is not a subtle cardinal and fix a closed unbounded subset $C$ of $\delta$ and a $\delta$-list $\vec{E}=\seq{E_\gamma}{\gamma<\delta}$ with the property that $E_\gamma\cap\beta\neq E_\beta$ holds for all $\beta<\gamma$ in $C$. 
  Let $\alpha$ denote the least  cardinal of uncountable cofinality in $C$.  
  
     Let $\calL^\prime$ denote the first-order language that extends the language of set theory by a unary predicate symbol $\dot{C}$,  constant symbols $\dot{\kappa}$, $\dot{u}$, $\dot{c}$, and  unary function symbols $\dot{e}$ and $\dot{s}$. 
     Let $\calL$ denote the first-order language that extends $\calL^\prime$ by an $(n+1)$-ary predicate symbol $\dot{T}_\varphi$ for  every $\calL^\prime$-formula $\varphi(v_0,\ldots,v_n)$ with $(n+1)$-many free variables, and let $\calS_\calL$  be the class of $\calL$-structures as defined at the beginning of Section \ref{prodref}. Namely, the class of $\calL$-structures $A$  such that there exists a cardinal $\kappa_A$ in $C^{(1)}$ and a limit  ordinal $\theta_A>\kappa_A$ 
 %transitive set $M_A$ with the property 
 such that: 
 \begin{enumerate}
     %\item %$M_A$ has cardinality $2^{\kappa_A}$ \todo[fancyline]{Needed?} and 
     %$V_{\kappa_A+1}\subseteq M_A$ and for some ordinal $\vartheta>\kappa_A$, there exists an elementary  embedding $\map{j}{M_A}{V_\vartheta}$ with $j\restriction V_{\kappa_A+1}=\id_{V_{\kappa_A+1}}$. 
     
     \item The domain of $A$ is $V_{\theta_A +1}$. 
     
     \item $\in^A={\in}\restriction{V_{\theta_A}+1}$, $\dot{\kappa}^A=\kappa_A$ and $\dot{u}^A=\theta_A$. 
     
     \item If $\varphi(v_0,\ldots,v_n)$ is an $\calL^\prime$-formula, then $$\dot{T}_\varphi^A ~ = ~ \Set{\langle x_0,\ldots,x_n\rangle\in V_{\theta_A +1}^{n+1}}{A\models\varphi(x_0,\ldots,x_n)}.$$
     %$n$ is a natural number and $A^-$ denotes the $\calL$-retract of $A$, then  $$\dot{T}_n^{A} ~ = ~ \Set{\langle a, x_0,\ldots,x_{n-1}\rangle\in M_A^{n+1}}{a\in\mathsf{Fml}_\calL, ~ \mathsf{Sat}_\calL(A^-,a,\langle x_0,\ldots,x_{n-1}\rangle)},$$ where $\mathsf{Fml}_\calL\subseteq V_\omega$ denotes the set of formalized $\calL$-formulas and $\mathsf{Sat_\calL}$ denotes the corresponding formalized satisfaction relation.  
 \end{enumerate}
 Let $\Ce$ denote the set of all $A\in\calS_\calL\cap V_\delta$ such that the following  hold: 
  \begin{itemize}
   \item $\kappa_A$ is a limit point of $C$ above $\alpha$. 
   
   \item $\dot{c}^A=\alpha$ and $\dot{C}^A=C\cap\kappa_A$. 
  
  \item If $\gamma<\kappa_A$, then $\dot{e}^A(\gamma)=E_\gamma$ and $\dot{s}^A(\gamma)=\min(C\setminus(\gamma+1))$.  
 \end{itemize}

 The fact that $\delta$ is inaccessible implies that $\Ce$ is non-empty and hence there exists a cardinal $\zeta<\delta$ with the property that $\WPR_\Ce(\zeta)$ holds. 
 Define $$\kappa ~ = ~ \sup\Set{\kappa_A}{A\in\Ce\cap V_\zeta} ~ \in ~ (\alpha,\zeta]$$ 
 and notice that $\kappa\in C^{(1)}\cap\Lim(C)$.
 Let $D$ be some cofinal subset of $\kappa$ of order-type $\cof{\kappa}$. Given a set $x$, we again define functions $f_x$ and $f^x$ with domain $\Ce\cap V_\zeta$ as in Section \ref{prodref}. Namely, we have $f^x(A)=x\cap V_{\kappa_A}$ for all $A\in\Ce\cap V_\zeta$, $f_x(A)=x$ for all $A\in\Ce\cap V_\zeta$ with $x\in V_{\kappa_A}$ and $f_x(A)=\dot{u}^A$ for all $A\in\Ce\cap V_\zeta$ with $x\notin V_{\kappa_A}$.  
 Then there is a substructure $X$ of $\prod(\Ce\cap V_\zeta)$ of cardinality $\kappa$ with the property that $f^D,f^{E_\kappa}\in X$ and $f_x,f^x\in X$   for all $x\in V_\kappa$. 
 Moreover, since $\delta$ is inaccessible, we can find a structure  $B$ in $\Ce$ with $\kappa_B>\zeta$.  By $\WPR_\Ce(\zeta)$, there is a homomorphism $\map{h}{X}{B}$. Using the results of Section \ref{prodref}, we can define  $$\lambda ~ = ~ \min\Set{\rank{x}}{f_x\in X, \, h(f_x)=\dot{u}^B} ~ \leq ~ \kappa$$ as well as $$\chi ~ = ~ \sup\Set{h(f_\beta)}{\beta<\lambda} ~ \leq ~ \kappa_B.$$ 
 Since $h(f_\alpha)=\dot{c}^B=\alpha\neq\dot{u}^B$, we have that $\alpha<\lambda$. 
 
 \begin{claim*}
  $\lambda\in C$.
 \end{claim*}
 
 \begin{proof}[Proof of the Claim]
  Assume, towards a contradiction, that $\lambda\notin C$. We then have $\alpha\in C\cap\lambda\neq\emptyset$ and, if we define $\beta=\sup(C\cap\lambda)$, then $\beta<\lambda$ and  $h(f_\beta)\neq\dot{u}^B$. 
  Set $\gamma=\min(C\setminus(\beta+1))$. If $A\in\Ce\cap V_\zeta$ is such that $f_\beta(A) \ne \dot{u}^A$, then $\beta<\kappa_A$,  $\gamma=\dot{s}^A(\beta)<\kappa_A$, and hence  $f_\gamma (A)=\gamma \ne \dot{u}^A$. Thus, for every $A\in \Ce \cap V_\zeta$, we have 
  $$A\models ``f_\beta \ne \dot{u}^A \longrightarrow f_\gamma \ne \dot{u}^A".$$
  As $h(f_\beta )\ne \dot{u}^B$,  Lemma \ref{lemma:Prod1}.\eqref{item:Prod1} shows that $h(f_\gamma)\neq\dot{u}^B$ and hence $\gamma<\lambda$, which yields a contradiction to the fact that $\gamma$ is the least element of $C$ greater than $\beta$, and therefore must be greater than $\lambda$.  
 \end{proof}

   Now  let  $$\Map{j}{V_\lambda}{V_\chi}{x}{h(f_x)}$$ be the $\Sigma_1$-elementary embedding given by  Lemma \ref{lemma:Prod2}.\eqref{item:ProdBonus}. 
   Lemma \ref{lemma:Prod2}.\eqref{item:Prod5} shows that $h(f_\lambda)=\dot{u}^B\neq\lambda$ and we can define $\mu\leq\lambda$ to be the minimal ordinal with $h(f_\mu)\neq\mu$. 
 %
%An application of Lemma \ref{lemma:ProdCons}.\eqref{item:ProdCons2.2} now shows that $\mu$ is an inaccessible cardinal smaller than $\kappa_B$. 
%
Then  Lemma  \ref{lemma:Prod2}.\eqref{item:Prod7} implies that $h(f_x)=x$ holds for all $x\in V_\mu$.

  \begin{claim*}
   $\mu<\kappa$. 
  \end{claim*}
  
  \begin{proof}[Proof of the Claim]
   Assume, towards a contradiction, that $\mu=\lambda=\kappa$ holds. 
   Since $f^{E_\kappa}\in X$, Lemma \ref{lemma:ProdCons}.\eqref{item:ProdCons1} shows that $h(f^{E_\kappa})\cap\kappa=E_\kappa=\dot{e}^B(\kappa)$. Moreover, since $\kappa\in C\cap\kappa_B=\dot{C}^B$, the cardinal $\kappa$   witnesses that $$B\models\anf{\exists\beta<\dot{\kappa} ~ [\beta\in\dot{C}\wedge h(f^{E_\kappa})\cap\beta=\dot{e}(\beta)]}.$$
   Using Lemma \ref{lemma:Prod1}.\eqref{item:Prod1}, we can now find $A\in\Ce\cap V_\zeta$ and $\beta\in C\cap\kappa_A$ with $E_\kappa\cap\beta=f^{E_\kappa}(A)\cap\beta=E_\beta$.  But this yields a contradiction, because we have $\beta<\kappa_A\leq\kappa$, and hence $\beta$ and $\kappa$ are distinct  elements of $C$.  
  \end{proof}

  \begin{claim*}
   $\mu<\lambda$. 
  \end{claim*}
  
  \begin{proof}[Proof of the Claim]
   Assume, towards a contradiction, that $\mu=\lambda$ holds. The above claims then show that $\mu\in C\cap\kappa$.  
   Since $f^{E_\mu}\in X$, Lemma \ref{lemma:ProdCons}.\eqref{item:ProdCons1} implies that $h(f^{E_\mu})\cap \mu=E_\mu$ and Lemma \ref{lemma:Prod2}.\eqref{item:Prod5} shows that   $h(f_\lambda)=\dot{u}^B$, the ordinal $\mu$ witnesses that   $$B\models\anf{h(f_\mu)=\dot{u} ~ \wedge ~ \exists\beta<\dot{\kappa} ~ [\beta\in\dot{C}\wedge h(f^{E_\mu})\cap\beta=\dot{e}(\beta)]}.$$
   Using Lemma \ref{lemma:Prod1}.\eqref{item:Prod1}, this yields $A\in\Ce\cap V_\zeta$ with $\kappa_A\leq\mu$ and $\beta\in C\cap\kappa_A$ satisfying  $E_\mu\cap\beta  =  f^{E_\mu}(A)\cap\beta  =    E_\beta$. Since $\beta<\mu$ are both elements of $C$,  this yields a contradiction. 
  \end{proof}

 By the above claims, we now know that $\mu<\lambda\leq\kappa$.

 \begin{claim*}
  $\alpha<\mu$. 
 \end{claim*}
 
   \begin{proof}[Proof of the Claim]
   Assume, towards a contradiction, that $\mu\leq\alpha$ holds. Since we have $$j(\alpha) ~ = ~ h(f_\alpha) ~ = ~ \dot{c}^B ~ = ~ \alpha$$ we must have $\mu<\alpha$.
   But this implies that $\map{j\restriction V_\alpha}{V_\alpha}{V_\alpha}$ is a non-trivial elementary embedding and, since $\alpha$ has uncountable cofinality, this yields a contradiction via the \emph{Kunen Inconsistency}. 
  \end{proof}

 \begin{claim*}
  $\mu\in C$. 
 \end{claim*}

  \begin{proof}[Proof of the Claim]
   Assume, towards a contradiction, that $\mu$ is not an element of $C$. Since the previous claim shows that $\alpha\in C\cap\mu$, we know that $\beta=\sup(C\cap\mu)<\mu$ and $\gamma=\min(C\setminus(\beta+1))>\mu$. 
   We then have $\gamma=\dot{s}^A(\beta)<\kappa_A$ for all $A\in\Ce\cap V_\zeta$ with $\mu<\kappa_A$. Thus for every $A\in \Ce \cap V_\zeta$, we have 
   $$A\models ``f_\mu(A)\ne \dot{u}^A \longrightarrow  \dot{s}^A (f_\beta(A))<\dot{\kappa}^A".$$ 
   Thus, since an earlier claim shows $\mu<\lambda$ and therefore Lemma \ref{lemma:Prod1}.\eqref{item:Prod1} implies that $h(f_\mu)\neq \dot{u}^B$, we know that $\dot{s}^B(h(f_\beta))<\kappa_B$.
   Hence, we can conclude that 
   $$j(\gamma)=h(f_\gamma))= \dot{s}^B(h(f_\beta))=\dot{s}^B(j(\beta))=\dot{s}^B(\beta)=\gamma.$$
   Since Lemma \ref{lemma:Prod2}.\eqref{item:Prod5} ensures that $\gamma+2<\lambda$, the $\Sigma_1$-elementarity of $j$ implies that $j(V_{\gamma+2})=V_{\gamma+2}$ and we can conclude that $\map{j\restriction V_{\gamma+2}}{V_{\gamma+2}}{V_{\gamma+2}}$ is a non-trivial elementary embedding, which yields a contradiction via \emph{Kunen Inconsistency}. 
  \end{proof}

  Now notice that since $\mu <\lambda$, Lemma \ref{lemma:Prod2}.\eqref{item:Prod7} yields that $h(f_\mu)\ne \dot{u}^B$, and since  $\mu<j(\mu)=h(f_\mu)$, we have that $$B\models\anf{h(f_\mu)\neq\dot{u} ~ \wedge ~ \exists\beta< h(f_\mu) ~ [\beta\in\dot{C}\wedge h(f^{E_\mu})\cap\beta=\dot{e}(\beta)]}$$
  as witnessed by $\mu$.
  An application of  Lemma \ref{lemma:Prod1}.\eqref{item:Prod1} then yields an $A\in\Ce\cap V_\zeta$ with $\kappa_A>\mu$ and $\beta\in C\cap\mu$ with the property that $$E_\mu\cap\beta ~ = ~ f^{E_\mu}(A)\cap\beta ~ = ~ \dot{e}^A(\beta) ~ = ~ E_\beta.$$
  Since $\beta$ and $\mu$ are distinct elements of $C$, this yields a final contradiction.  
  
  The above computations yield the implication $(2)\Rightarrow (1)$ of the theorem. Since the implication $(1)\Rightarrow (2)$ follows from Lemma \ref{lemma:SubtleYieldsProdRefl}, this completes the proof of the theorem. 
\end{proof}

%%%%%%%%%%%%%
%%%%%%%%%%%%%

\section{Between strongly unfoldable and subtle cardinals}
\label{section9}

We shall give in this section a proof of Theorem \ref{C(n)StronglyUnfoldable}. But first we shall prove some properties of $C^{(n)}$-strongly unfoldable cardinals, and give other equivalent reformulations  of these cardinals in terms of elementary embeddings.
Recall (see Definition \ref{defCnsf}) that an inaccessible cardinal $\kappa$ is \emph{$C^{(n)}$-strongly unfoldable} if for every ordinal $\lambda\in C^{(n)}$ greater than $\kappa$ and every transitive $\mathrm{ZF}^-$-model $M$ of cardinality $\kappa$ with $\kappa\in M$ and ${}^{{<}\kappa}M\subseteq M$, there is a transitive set $N$ with $V_\lambda\subseteq N$ and an elementary embedding $\map{j}{M}{N}$ with $\crit{j}=\kappa$,  $j(\kappa)>\lambda$ and $V_\lambda\prec_{\Sigma_n}V_{j(\kappa)}^N$. 
%Note that every strongly unfoldable cardinal is $C^{(1)}$-strongly unfoldable.

\begin{proposition}\label{proposition:CnSUNF-Cn+1}
 $C^{(n)}$-strongly unfoldable cardinals are elements of $C^{(n+1)}$. 
\end{proposition}

\begin{proof}
 Let $\kappa$ be a $C^{(n)}$-strongly unfoldable cardinal. 
 Pick a $\Sigma_{n+1}$-formula $\varphi(v)$ and $z\in V_\kappa$ with the property that $\varphi(z)$ holds in $V$. %
 Fix $\lambda\in C^{(n+1)}$ greater than $\kappa$, so that $\varphi(z)$ holds in $V_\lambda$, and fix an elementary submodel $M$ of $H_{\kappa^+}$ of cardinality $\kappa$ with $(\kappa+1)\cup{}^{{<}\kappa}M\subseteq M$. By our assumption, there exists a transitive set $N$ with $V_\lambda\subseteq N$ and an elementary embedding $\map{j}{M}{N}$ with  $\crit{j}=\kappa$,  $j(\kappa)>\lambda$ and $V_\lambda\prec_{\Sigma_n}V_{j(\kappa)}^N$. 
Thus, since $\varphi(z)$ holds in $V_\lambda$, it also holds in $V_{j(\kappa)}^N$,  and hence elementarity implies that $\varphi(z)$ holds in $V_\kappa$. 
\end{proof}

Clearly, a cardinal is strongly unfoldable if and only if it is $C^{(0)}$-strongly unfoldable. But more is true:

\begin{proposition}\label{proposition:SunfC1Sunf}
 A cardinal is strongly unfoldable if and only if it is $C^{(1)}$-strongly unfoldable. 
\end{proposition}

\begin{proof}
 Let $\kappa$ be an inaccessible cardinal, let $\lambda\in C^{(1)}$ be greater than $\kappa$, let $M$ be a transitive $\mathrm{ZF}^-$-model  of cardinality $\kappa$ with  $\{\kappa\}\cup{}^{{<}\kappa}M\subseteq M$, let $N$ be a transitive set with $V_\lambda\subseteq N$, and let  $\map{j}{M}{N}$ be an elementary embedding with $\crit{j}=\kappa$ and   $j(\kappa)>\lambda$. 
 Then, in $N$,  $\lambda$ is also in $C^{(1)}$. So, since $j(\kappa)$ is an inaccessible cardinal in $N$,  hence also in $C^{(1)}$ in the sense of $N$, we have that $V_\lambda\prec_{\Sigma_1}V_{j(\kappa)}^N$. 
\end{proof}

\begin{proposition}\label{proposition:CnExtendsibleCn+1Sunf}
 Given a natural number $n$, every $C^{(n)}$-extendible cardinal is $C^{(n+1)}$-strongly unfoldable. 
\end{proposition}

\begin{proof}
 Since $C^{(0)}$-extendibility coincides with $C^{(1)}$-extendibility and $C^{(0)}$-strong unfoldability coincides with $C^{(1)}$-strong unfoldability, we may assume that $n$ is greater than $0$. 
 Let $\kappa$ be a $C^{(n)}$-extendible cardinal, let $\lambda>\kappa$ be an element of $C^{(n+1)}$ and let $M$ be a transitive $\mathrm{ZF}^-$-model  of cardinality $\kappa$ with $\kappa\in M$ and ${}^{{<}\kappa}M\subseteq M$. 
 Then there is an ordinal  $\mu$ and an elementary embedding $\map{j}{V_\lambda}{V_\mu}$ with $\crit{j}=\kappa$, $j(\kappa)>\lambda$ and $j(\kappa)\in C^{(n)}$. 
 Elementarity now implies that $N=j(M)$ is a transitive set with $V_{j(\kappa)}\subseteq N$, and our setup ensures that $V_\lambda\prec_{\Sigma_{n+1}}V_{j(\kappa)}=V^N_{j(\kappa)}$. 
 Moreover, the map $\map{i=j\restriction M}{M}{N}$ is an elementary embedding with $\crit{i}=\kappa$, $i(\kappa)>\lambda$ and $V_\lambda\prec_{\Sigma_{n+1}}V^N_{i(\kappa)}$.  
\end{proof}

\begin{theorem}\label{theorem:SigmaNsunf}
 Given a natural number $n>0$, the following statements are equivalent for every cardinal $\kappa$: 
 \begin{enumerate}
     \item\label{item:sunf}  $\kappa$ is $C^{(n)}$-strongly unfoldable. 
     
     \item\label{item:shrewd} For every $\calL_\in$-formula $\varphi(v_0,v_1)$, every ordinal $\gamma\in C^{(n)}$ greater than $\kappa$, and every subset $A$ of $V_\kappa$ with the property that $\varphi(\kappa,A)$ holds in $V_\gamma$, there exist ordinals $\alpha<\beta<\kappa$ with $\beta\in C^{(n)}$ and the property that $\varphi(\alpha,A\cap V_\alpha)$ holds in $V_\beta$. 
     
    \item\label{item:Magidor} For every ordinal $\gamma\in C^{(n)}$ greater than $\kappa$ and every $z\in V_\gamma$ there exist an ordinal $\bar{\gamma}\in C^{(n)}\cap\kappa$, a cardinal $\bar{\kappa}<\bar{\gamma}$, an elementary submodel $X$ of $V_{\bar{\gamma}}$ with $V_{\bar{\kappa}}\cup\{\bar{\kappa}\}\subseteq X$, and an elementary embedding $\map{j}{X}{V_\gamma}$ with $j\restriction\bar{\kappa}=\id_{\bar{\kappa}}$, $j(\bar{\kappa})=\kappa$ and $z\in\ran{j}$.

    \item\label{item:sunf-superalmosthuge}
    For every ordinal $\lambda\in C^{(n)}$ greater than $\kappa$, every transitive $\mathrm{ZF}^-$-model $M$ of cardinality $\kappa$ with $\kappa\in M$ and ${}^{{<}\kappa}M\subseteq M$, and every $\Pi_{n-1}$-formula $\psi(v_0,v_1)$ with the property that $\psi(M,\kappa)$ holds, there is a transitive set $N$ with $V_\lambda\subseteq N$ and an elementary embedding $\map{j}{M}{N}$ such that  $\crit{j}=\kappa$,  $j(\kappa)>\lambda$,  $V_\lambda\prec_{\Sigma_n}V_{j(\kappa)}^N$, and $\psi(N,j(\kappa))$ holds.  
 \end{enumerate}
\end{theorem}

\begin{proof}
 First, note that, by a well-known result of Levy (see \cite[Proposition 6.2]{MR1994835}), the assumption \eqref{item:shrewd} implies that $\kappa$ is an inaccessible cardinal. And also each of $\eqref{item:Magidor}$ and $\eqref{item:sunf-superalmosthuge}$ implies that $\kappa$ is inaccessible, since so are the critical points of elementary embeddings of sufficiently reach transitive models.

 $\eqref{item:sunf}\Rightarrow \eqref{item:shrewd}$: Assume that \eqref{item:sunf} holds and fix an $\calL_\in$-formula $\varphi(v_0,v_1)$, an ordinal $\gamma\in C^{(n)}$ greater than $\kappa$, and a subset $A$ of $V_\kappa$ such that $\varphi(\kappa,A)$ holds in $V_\gamma$. %
 Pick an elementary submodel $M$ of $H_{\kappa^+}$ of cardinality $\kappa$ with $\{\kappa,A\}\cup{}^{{<}\kappa}M\subseteq M$ and use  our assumption to find a transitive set $N$ with $V_\gamma\subseteq N$ and an elementary embedding $\map{j}{M}{N}$ with $\crit{j}=\kappa$, $j(\kappa)>\gamma$ and $V_\gamma\prec_{\Sigma_n}V_{j(\kappa)}^N$. 
 The ordinals $\kappa$ and $\gamma$ then witness that, in $N$, there are ordinals $\alpha<\beta<j(\kappa)$ such that $V_\beta\prec_{\Sigma_n}V_{j(\kappa)}$ and $\varphi(\alpha,j(A)\cap V_\alpha)$ holds in $V_\beta$. 
 Using the elementarity of $j$, we find ordinals $\alpha<\beta<\kappa$ with the property that  $V_\beta\prec_{\Sigma_n}V_\kappa$ and $\varphi(\alpha,A\cap V_\alpha)$ holds in $V_\beta$. 
 Since Proposition \ref{proposition:CnSUNF-Cn+1} shows that $\kappa\in C^{(n+1)}$, we also have that $\beta$ is in $C^{(n)}$.

 $\eqref{item:shrewd}\Rightarrow \eqref{item:Magidor}$: Assume that \eqref{item:shrewd} holds. 
 Pick an ordinal $\gamma\in C^{(n)}$ greater than $\kappa$, and an element $z$ of $V_\gamma$. 
 Let $\varphi(v_0,v_1)$ be an $\calL_\in$-formula  with the property that for every ordinal $\lambda$ and all $A,\delta\in V_\lambda$, the statement $\varphi(A,\delta)$ holds in $V_\lambda$ if and only if $\lambda$ is a limit ordinal and  there exists an ordinal  $\varepsilon>\delta$, a subset $X$ of $V_\varepsilon$ and a bijection $\map{b}{\delta}{X}$ such that the following statements hold: 
    \begin{itemize}  
     \item $V_\varepsilon$ is $\Sigma_n$-correct. 
     
     \item $V_\delta\cup\{\delta\}\subseteq X$.
     
     \item $b(0)=\delta$ and $b(\omega\cdot(1+\alpha))=\alpha$ for all $\alpha<\delta$. 
    
     \item If  $\alpha_0,\ldots,\alpha_{n-1}<\delta$ and $a\in\mathsf{Fml}$  represents a formula with $n$ free variables, then 
         \begin{equation*}
          \begin{split}
            \langle a,\alpha_0,\ldots,\alpha_{n-1}\rangle\in A ~ & \Longleftrightarrow ~ \mathsf{Sat}(X,\langle b(\alpha_0),\ldots,b(\alpha_{n-1})\rangle,a) \\
            & \Longleftrightarrow ~ \mathsf{Sat}(V_\varepsilon,\langle b(\alpha_0),\ldots,b(\alpha_{n-1})\rangle,a),  
           \end{split}
         \end{equation*}
      where $\mathsf{Fml}$ denotes the set of  formalized $\calL_\in$-formulas and $\mathsf{Sat}$ denotes the  formalized satisfaction relation for $\calL_\in$-formulas.\footnote{See {\cite[Section I.9]{devlin_2017}}. Note that the classes $\mathsf{Fml}$ and $\mathsf{Sat}$ are defined by $\Sigma_1$-formulas without parameters.}  
    \end{itemize}
  Pick an ordinal $\lambda\in C^{(n)}$ greater than $\gamma$, an  elementary submodel $Y$ of $V_\gamma$ of cardinality $\kappa$ with $V_\kappa\cup\{\kappa,z\}\subseteq Y$ and a bijection $\map{b}{\kappa}{Y}$ with  $b(0)=\kappa$, $b(1)=z$ and $b(\omega\cdot(1+\alpha))=\alpha$ for all $\alpha<\kappa$. 
 Let $A$ denote the set of all tuples $\langle a,\alpha_0,\ldots,\alpha_{n-1}\rangle$ with the property that $a\in\mathsf{Fml}$ represents an  $\calL_\in$-formula with $n$ free variables,  $\alpha_0,\ldots,\alpha_{n-1}<\kappa$ and $\mathsf{Sat}(Y,\langle b(\alpha_0),\ldots,b(\alpha_{n-1})\rangle,a)$.  
 Then $\gamma$, $Y$ and $b$ witness that the statement $\varphi(A,\kappa)$ holds in $V_\lambda$, and we can use our assumption to find ordinals $\bar{\kappa}<\beta<\kappa$ such that $\beta\in C^{(n)}$ and $\varphi(A\cap V_{\bar{\kappa}},\bar{\kappa})$ holds in $V_\beta$. 
 Then there are  ordinals  $\bar{\kappa}<\bar{\gamma}<\beta$, a subset $X$ of $V_{\bar{\gamma}}$ and a bijection $\map{\bar{b}}{\bar{\kappa}}{X}$ such that $V_{\bar{\gamma}}\prec_{\Sigma_n}V_\beta$,  $V_{\bar{\kappa}}\cup\{\bar{\kappa}\}\subseteq X$, $\bar{b}(0)=\bar{\kappa}$, $\bar{b}(\omega\cdot(1+\alpha))=\alpha$ for all $\alpha<\bar{\kappa}$ and 
  \begin{equation*}
          \begin{split}
            \langle a,\alpha_0,\ldots,\alpha_{n-1}\rangle\in A\cap V_{\bar{\kappa}} ~ & \Longleftrightarrow ~ \mathsf{Sat}(X,\langle \bar{b}(\alpha_0),\ldots,\bar{b}(\alpha_{n-1})\rangle,a) \\
            & \Longleftrightarrow ~ \mathsf{Sat}(V_{\bar{\gamma}},\langle\bar{b}(\alpha_0),\ldots,\bar{b}(\alpha_{n-1})\rangle,a)  
           \end{split}
         \end{equation*}
         for all $\alpha_0,\ldots,\alpha_{n-1}<\bar{\kappa}$ and every $a\in\mathsf{Fml}$ representing a formula with $n$ free variables. 
    Since $\beta \in C^{(n)}$ and $V_{\bar{\gamma}}\prec_{\Sigma_n}V_\beta$, we have that $\bar{\gamma}\in C^{(n)}$. Also, the displayed equivalences above show that  $X$ is an elementary submodel of $V_{\bar{\gamma}}$ and, for each $\calL_\in$-formula $\varphi(v_0,\ldots,v_{n-1})$ and  $\alpha_0,\ldots,\alpha_{n-1}<\bar{\kappa}$, we  have 
    \begin{equation*}
     \begin{split}
      X\models\varphi(\bar{b}(\alpha_0),\ldots,\bar{b}(\alpha_{n-1})) ~ & \Longleftrightarrow ~   \langle \ulcorner\varphi\urcorner,\alpha_0,\ldots,\alpha_{n-1}\rangle\in A\cap V_{\bar{\kappa}} \\ 
      ~ \Longleftrightarrow  ~ \langle \ulcorner\varphi\urcorner,\alpha_0,\ldots,\alpha_{n-1}\rangle\in A ~ & \Longleftrightarrow  ~  V_\gamma\models\varphi(b(\alpha_0),\ldots,b(\alpha_{n-1})),   
   \end{split}
  \end{equation*}
  where $\ulcorner\varphi\urcorner$ denotes the canonical element of $\mathsf{Fml}$ representing  $\varphi$. 
 In particular, we know that 
   the map $\map{j=b\circ\bar{b}^{{-}1}}{X}{V_\gamma}$ is an elementary embedding that satisfies  $j\restriction\bar{\kappa}=\id_{\bar{\kappa}}$, $j(\bar{\kappa})=\kappa$ and $z\in\ran{j}$. This shows that \eqref{item:Magidor} holds.

 $\eqref{item:Magidor}\Rightarrow \eqref{item:sunf-superalmosthuge}$: Assume that \eqref{item:Magidor} holds. 
 Fix  ordinals $\lambda<\gamma$ in $C^{(n)}$ and greater than $\kappa$, a transitive $\mathrm{ZF}^-$-model $M$ of cardinality $\kappa$ with $\{\kappa\}\cup{}^{{<}\kappa}M\subseteq M$, and a $\Pi_{n-1}$-formula $\psi(v_0,v_1)$ such that $\psi(M,\kappa)$ holds. 
 Then there exists an ordinal $\bar{\gamma}\in C^{(n)}\cap\kappa$, a cardinal $\bar{\kappa}<\bar{\gamma}$, an elementary submodel $X$ of $V_{\bar{\gamma}}$ with $V_{\bar{\kappa}}\cup\{\bar{\kappa}\}\subseteq X$, and an elementary embedding $\map{j}{X}{V_\gamma}$ with $j\restriction\bar{\kappa}=\id_{\bar{\kappa}}$, $j(\bar{\kappa})=\kappa$ and $M,\lambda\in\ran{j}$. 
 Pick $\bar{M},\bar{\lambda}\in X$ with $j(\bar{M})=M$ and $j(\bar{\lambda})=\lambda$. 
 Then $\bar{M}\subseteq X$ and the fact that the class $C^{(n)}$ is $\Pi_n$-definable implies that $\bar{\lambda}$ is an also element of $C^{(n)}$. In particular, we know that $V_{\bar{\lambda}}\prec_{\Sigma_n}V_\kappa$. 
 Thus, $\kappa$, $M$ and $j\restriction\bar{M}$ witness that there exists a transitive set $N$ and an elementary embedding $\map{k}{\bar{M}}{N}$ such that 
 $\bar{\lambda}<k(\bar{\kappa})$, 
 $V_{\bar{\lambda}}\subseteq N$, 
 $V_{\bar{\lambda}}\prec_{\Sigma_n}V_{k(\bar{\kappa})}^N$, 
 $\crit{k}=\bar{\kappa}$  and $\psi(N,k(\bar{\kappa}))$ holds.  
 Since this statement can be formulated by a $\Sigma_n$-formula with parameters $\bar{\kappa}$, $\bar{M}$ and $V_{\bar{\lambda}}$, the fact that all these parameters are contained in $X$ implies that the  statement holds in $X$. 
 The elementarity of $j$ and the $\Sigma_n$-correctness of $V_\gamma$ now yield a transitive set $N$ with $V_\lambda\subseteq N$ and an elementary embedding $\map{k}{M}{N}$ such that  $\crit{k}=\kappa$, $k(\kappa)>\lambda$, $V_\lambda\prec_{\Sigma_n}V_\rho^N$ and $\psi(N,k(\kappa))$ holds. This shows that \eqref{item:sunf-superalmosthuge} holds.

 Since the implication $\eqref{item:sunf-superalmosthuge}\Rightarrow \eqref{item:sunf}$ is trivial, this completes the proof that all four of the listed statements  are equivalent. 
\end{proof}

\begin{remark}\label{remarkCnsf}
 Note that, in the definition of $\kappa$ being $C^{(n)}$-strongly unfoldable (Definition \ref{defCnsf}), we may only require, and obtain an equivalent definition, that for a \emph{tail} of ordinals   $\lambda\in C^{(n)}$ greater than $\kappa$ and every transitive $\mathrm{ZF}^-$-model $M$ of cardinality $\kappa$ with $\kappa\in M$ and ${}^{{<}\kappa}M\subseteq M$, there is a transitive set $N$ with $V_\lambda\subseteq N$ and an elementary embedding $\map{j}{M}{N}$ with $\crit{j}=\kappa$,  $j(\kappa)>\lambda$ and $V_\lambda\prec_{\Sigma_n}V_{j(\kappa)}^N$. For given \emph{any} $\lambda'\in C^{(n)}$ greater than $\kappa$, and given $M$ as before, we can pick $\lambda \geq \lambda'$ in the tail so that the required $N$ exists. Then $N$ also works for $\lambda$, because $V_{\lambda'}\prec_{\Sigma_n}V_\lambda \prec_{\Sigma_n} V^N_{j(\kappa)}$, and therefore $V_{\lambda'}\prec_{\Sigma_n}  V^N_{j(\kappa)}$. 

 The same applies to the equivalent reformulations of $C^{(n)}$-strong unfoldability given in Theorem \ref{theorem:SigmaNsunf}, namely one gets equivalent statements by only requiring that they hold for a \emph{tail} of $\lambda$ in $C^{(n)}$. Let us see this for \eqref{item:shrewd}: so suppose \eqref{item:shrewd} holds for a tail of $\lambda\in C^{(n)}$. Given  $\varphi(v_0,v_1)$, $\gamma\in C^{(n)}$ greater than $\kappa$, and  $A\subseteq V_\kappa$ such that $\varphi(\kappa,A)$ holds in $V_\gamma$, pick $\lambda \in C^{(n)}$ in the tail such that the sentence 
$$\exists X,\delta ~ [X=V_\delta ~ \wedge ~ \delta \in C^{(n)} ~ \wedge ~  V_\delta \models \varphi(\kappa,A)]$$ holds in $V_\lambda$ (such a $\lambda$ exists by the Reflection Theorem, since the sentence  is true in $V$, and therefore true in  $V_\beta$ for a closed unbounded class  class of $\beta$). So, there exist ordinals $\alpha<\beta<\kappa$ with $\beta\in C^{(n)}$ and the property that the sentence $$\exists X,  \delta 
~ [X=V_\delta ~ \wedge ~  \delta \in C^{(n)} ~ \wedge ~ V_\delta \models \varphi(\alpha, A\cap \alpha)]$$ holds in $V_\beta$. If $\delta$ witnesses this, then $\delta \in C^{(n)}$, because $V_\beta$ is correct about this. 
\end{remark}

\begin{corollary}
 Given a nautural number $n>1$, the following statements are equivalent for every cardinal $\kappa$: 
 \begin{enumerate}
   \item  $\kappa$ is $C^{(n)}$-strongly unfoldable. 
     
   \item For every cardinal $\lambda\in C^{(n)}$ greater than $\kappa$, and every transitive $\mathrm{ZF}^-$-model $M$ of cardinality $\kappa$ with $\kappa\in M$ and ${}^{{<}\kappa}M\subseteq M$, there exists an inaccessible cardinal $\rho\in C^{(n-1)}$ greater than $\lambda$,\footnote{Note that these assumptions imply that $V_\lambda\prec_{\Sigma_n}V_\rho$.} a transitive set $N$ with $\rho\in N$ and ${}^{{<}\rho}N\subseteq N$,  and an elementary embedding $\map{j}{M}{N}$ with $\crit{j}=\kappa$ and   $j(\kappa)=\rho$. \qed 
\end{enumerate}
\end{corollary}

\begin{proof}
 Assuming $(1)$, the statement $(2)$ follows immediately from \eqref{item:sunf-superalmosthuge} of Theorem \ref{theorem:SigmaNsunf}  by taking $\varphi(v_0,v_1)$ to be the formula that asserts that $v_1\in C^{(n-1)}$ and $v_0$ is a transitive $\mathrm{ZF}^-$-model  of cardinality $v_1$ with $v_1 \in v_0$ and ${}^{{<}v_1}v_0\subseteq v_0$, and by taking $\rho=j(\kappa)$. 
 That $(2)$ implies $(1)$ is immediate.
 \end{proof}

In particular, this shows that all $C^{(2)}$-strongly unfoldable cardinals are \emph{almost-hugely unfoldable}, as introduced in   {\cite[Definition 4]{HJ:SU}}.

%%%%
%%%%

The characterizations and properties of $C^{(n)}$-strongly unfoldable cardinals given in Section \ref{section3} can now be put to use to extend the equivalences between strong unfoldability and weak forms of Structural Reflection given by Theorem \ref{theorem:CharSr-Sr--} to any classes of structures of any degree of definitional complexity.

\begin{lemma}\label{lemmaCnSR}
 If $\kappa$ is $C^{(n)}$-strongly unfoldable cardinal, then $\HSR^-_\Ce(\kappa)$ holds for every class $\Ce$ of structures of the same type that is definable by a $\Sigma_{n+1}$-formula with parameters in $V_\kappa$. 
\end{lemma}

\begin{proof}
 Pick a $\Pi_n$-formula $\varphi(v_0,v_1,v_2)$ and an element $z$ of $V_\kappa$ with the property that $$\Ce ~ = ~ \Set{A}{\exists x ~ \varphi(A,x,z)}.$$ 
 Fix a structure $B$ in $\Ce$ of cardinality $\kappa$ and an ordinal $\kappa<\gamma\in C^{(n)}$ with the property that $\varphi(B,y,z)$ holds for some $y\in V_\gamma$. 
 By Theorem \ref{theorem:SigmaNsunf}, we can now find an ordinal $\bar{\gamma}\in C^{(n)}\cap\kappa$, a cardinal $\bar{\kappa}<\bar{\gamma}$, an elementary submodel $X$ of $V_{\bar{\gamma}}$ with $V_{\bar{\kappa}}\cup\{\bar{\kappa}\}\subseteq X$ and an elementary embedding $\map{j}{X}{V_\gamma}$ with $j\restriction\bar{\kappa}=\id_{\bar{\kappa}}$, $j(\bar{\kappa})=\kappa$ and $B,z\in\ran{j}$. 
 Pick $A\in V_{\bar{\gamma}}$ with $j(A)=B$. Our setup then ensures that the domain of $A$ is a subset of $X$, $\map{j\restriction A}{A}{B}$ is an elementary embedding and $A$ is an element of $\Ce\cap H_\kappa$. 
\end{proof}

We shall next prove Theorem \ref{C(n)StronglyUnfoldable}, which is a generalization of Theorem \ref{theorem:CharSr-Sr--} to arbitrary classes of structures. Namely, we will show that, for every $n>1$, the following  are equivalent for every  cardinal $\kappa$: 
 \begin{enumerate} 
  \item $\kappa$ is either $C^{(n)}$-strongly unfoldable or a limit of $C^{(n-1)}$-extendible cardinals. 
   
  \item  The principle $\WSR_\Ce(\kappa)$ holds for every class $\Ce$ of structures of the same type that is definable by a $\Sigma_{n+1}$-formula with parameters in $V_\kappa$. 
   \item $\kappa \in C^{(n+1)}$ and the principle $\HSR_\Ce^- (\kappa)$ holds for every class $\Ce$ of structures of the same type that is definable by a $\Sigma_{n+1}$-formula with parameters in $V_\kappa$. 
 \end{enumerate}

The proof follows similar arguments as that  of Theorem \ref{theorem:CharSr-Sr--} (given in Section  \ref{proof1.7}), but now using Proposition \ref{proposition:CnSUNF-Cn+1}, Lemma \ref{lemmaCnSR},  and the characterization of $C^{(n)}$-strongly unfoldable cardinals given in Theorem \ref{theorem:SigmaNsunf} \ref{item:Magidor}.

\begin{proof}[Proof of Theorem \ref{C(n)StronglyUnfoldable}]
  Assume that $\kappa$ is not $C^{(n)}$-strongly unfoldable and the principle $\WSR_\Ce(\kappa)$ holds for every class $\Ce$ of structures of the same type that is definable by a $\Sigma_{n+1}$-formula with parameters in $V_\kappa$. 
  Lemma \ref{lemma:Sigma2Reflex} then shows that $\kappa\in C^{(n+1)}$. In particular, we know that $\kappa$ is a limit cardinal and $\betrag{V_\kappa}=\kappa$.

 \begin{claim*}
  If $\theta$ is a cardinal in $C^{(n)}$ greater than $\kappa$, $y\in V_\kappa$, and $z\in H_\theta$, then there are cardinals $\bar{\kappa}<\bar{\theta}<\kappa$ with $y\in V_{\bar{\kappa}}$ and $\bar{\theta}\in C^{(n)}$, an elementary submodel $X$ of $H_{\bar{\theta}}$ with $V_{\bar{\kappa}}\cup\{\bar{\kappa}\}\subseteq X$, and an elementary embedding $\map{j}{X}{H_\theta}$ such that  $j(\bar{\kappa})=\kappa$, $j(y)=y$ and $z\in\ran{j}$. 
 \end{claim*}
 
 \begin{proof}[Proof of the Claim]
  Let $\calL$ denote the first-order language that extends $\calL_\in$ by three constant symbols and let $\Ce$ denote the class of all $\calL$-structures of the form $\langle M,\in,\mu,a,y\rangle$ such that $\mu$ is a cardinal in $C^{(n)}$, $y\in V_\mu$, and there exists a cardinal $\nu$ in $C^{(n)}$ greater than $\mu$ and   an elementary submodel $X$ of $H_{\nu}$ with $V_\mu\cup\{\mu\}\subseteq X$ and the property that $M$ is the transitive collapse of $X$. 
  Since the class $C^{(n)}$ is $\Pi_{n}$-definable, the class $\Ce$ is  definable by a $\Sigma_{n+1}$-formula with parameter $y$. 
  Now, let $Y$ be an elementary submodel of $H_{\theta}$ of cardinality $\kappa$ with $V_\kappa\cup\{\kappa,z\}$ and let $\map{\tau}{Y}{N}$ denote the corresponding transitive collapse. Thus $\theta$ and $Y$ witness that $B=\langle N,\in,\kappa,\tau(z),y\rangle$ is an element of $\Ce$ of cardinality $\kappa$. 
  Our assumption $\WSR_\Ce(\kappa)$ now yields cardinals $\bar{\kappa}<\bar{\theta}$ with $\bar{\kappa}\in C^{(n)}\cap\kappa$, an elementary submodel $X$ of $H_{\bar{\theta}}$ of cardinality less than $\kappa$ with $V_{\bar{\kappa}}\cup\{\bar{\kappa}\}\subseteq X$ and an elementary embedding $\map{i}{M}{N}$, with $M$ being the transitive collapse of $X$, and with $i(\bar{\kappa})=\kappa$, $i(y)=y$ and $\tau(z)\in\ran{i}$. 
  Since $\kappa\in C^{(n+1)}$ and $M\in V_\kappa$, we may assume that $\bar{\theta}<\kappa$. Letting  $\map{\pi}{X}{M}$ denote the  transitive collapse, define $$\map{j ~ = ~ \tau^{{-}1}\circ i\circ\pi}{X}{H_{\theta}}.$$ Then $j$ is an elementary embedding with $j(\bar{\kappa})=\kappa$, $j(y)=y$ and $z\in\ran{j}$.  
 \end{proof}

 \begin{claim*}
 If $\theta$ is a cardinal in $C^{(n)}$ greater than $\kappa$, $y\in V_\kappa$, and $z\in H_\theta$, then there are cardinals $\bar{\kappa}<\bar{\theta}<\kappa$ with $y\in V_{\bar{\kappa}}$ and $\bar{\theta}\in C^{(n)}$,  an elementary submodel $X$ of $H_{\bar{\theta}}$ with $V_{\bar{\kappa}}\cup\{\bar{\kappa}\}\subseteq X$, and an elementary embedding $\map{j}{X}{H_\theta}$ such that  $j(\bar{\kappa})=\kappa$, $j(y)=y$, $z\in\ran{j}$, and $j\restriction\bar{\kappa}\neq\id_{\bar{\kappa}}$. 
 \end{claim*}
 
 \begin{proof}[Proof of the Claim]
  %First, assume that $\kappa$ is singular. By the previous claim, there exist cardinals $\bar{\kappa}<\bar{\theta}<\kappa$ with $\cof{\kappa},y\in V_{\bar{\kappa}}$, an elementary submodel $X$ of $H_{\bar{\theta}}$ with $V_{\bar{\kappa}}\cup\{\bar{\kappa}\}\subseteq X$ and an elementary embedding $\map{j}{X}{H_\theta}$ with $j(\bar{\kappa})=\kappa$, $j(\cof{\kappa})=\cof{\kappa}$, $j(y)=y$ and $z\in\ran{j}$. 
  %
  %Then elementarity implies that $\cof{\kappa}=\cof{\bar{\kappa}}$ and $X$ contains a cofinal function $\map{c}{\cof{\kappa}}{\bar{\kappa}}$. 
  %
  %We then know that $\map{j(c)}{\cof{\kappa}}{\kappa}$ is cofinal and there exists $\xi<\cof{\kappa}$ with $j(c)(\xi)>\bar{\kappa}$. But then either $j(\xi)\neq \xi$ or $j(c(\xi))\neq c(\xi)$, because otherwise we would have $$j(c)(\xi) ~ = ~ j(c)(j(\xi)) ~ = ~ j(c(\xi)) ~ = ~ c(\xi) ~ < ~ \bar{\kappa}.$$ In particular, we know that $j\restriction\bar{\kappa}\neq\id_{\bar{\kappa}}$ in this case. 

  %Now, assume that $\kappa$ is regular. The fact that $\kappa$ is in $C^{(2)}$ implies that $\kappa^{{<}\kappa}=\kappa$. 
  %
  Since we assumed $\kappa$ is not $C^{(n)}$-strongly unfoldable,  Theorem \ref{theorem:SigmaNsunf}.\ref{item:Magidor}  and  Remark \ref{remarkCnsf} show that there exists a cardinal $\vartheta$ in $C^{(n)}$ greater than $\theta$, and  $z'\in V_\vartheta$ such that for all cardinals $\bar{\kappa}<\bar{\vartheta}$, with $\bar{\vartheta}\in C^{(n)}\cap \kappa$, and all elementary submodels $X$ of $H_{\bar{\vartheta}}$ with $V_{\bar{\kappa}}\cup \{ \bar{\kappa}\}\subseteq X$, there is no elementary embedding $\map{j}{X}{H_\vartheta}$ such that  $j\restriction\bar{\kappa}=\id_{\bar{\kappa}}$,  $j(\bar{\kappa})=\kappa$, and $z, z', \theta \in \ran{j}$.
  %, because otherwise {\cite[Lemma 3.1]{SRminus}} would show that $\kappa$ is a weakly shrewd cardinal with $\kappa^{{<}\kappa}=\kappa$, a combination of {\cite[Proposition 3.4]{SRminus}} and  {\cite[Lemma 3.5]{SRminus}} would then show that $\kappa$ is  \emph{shrewd} (see {\cite[Definition 1.4]{SRminus}}) and this would allow us to apply {\cite[Theorem 1.3]{luecke2021strong}} to show that $\kappa$ is strongly unfoldable. 
 %
  An application of our previous claim now yields cardinals $\bar{\kappa}<\bar{\vartheta}<\kappa$ with $\bar{\vartheta}\in C^{(n)}$ and $y\in V_{\bar{\kappa}}$, an elementary submodel $Y$ of $H_{\bar{\vartheta}}$ with $V_{\bar{\kappa}}\cup\{\bar{\kappa}\}\subseteq Y$, and an elementary embedding $\map{i}{Y}{H_\vartheta}$ such that  $i(\bar{\kappa})=\kappa$, $i(y)=y$ and $z,z',\theta\in\ran{i}$. 
 Therefore, we must have $i\restriction\bar{\kappa}\neq\id_{\bar{\kappa}}$. Pick $\bar{\theta}\in Y$ with $i(\bar{\theta})=\theta$. Then elementarity implies that $\bar{\theta}$ is a cardinal. Set $X=Y\cap H_{\bar{\theta}}$ and $j=i\restriction X$. 
 We can then conclude that $\bar{\kappa}<\bar{\theta}<\kappa$, $X$  is an elementary submodel of $H_{\bar{\theta}}$ with $V_{\bar{\kappa}}\cup\{\bar{\kappa}\}\subseteq X$ and $\map{j}{X}{H_\theta}$ is an elementary embedding  with $j(\bar{\kappa})=\kappa$, $j(y)=y$, $z\in\ran{j}$, and $j\restriction\bar{\kappa}\neq\id_{\bar{\kappa}}$. 
  \end{proof}

 \begin{claim*}
  There are unboundedly many cardinals below $\kappa$ that are $C^{(n-1)}$-extendible. 
 \end{claim*}
 
 \begin{proof}[Proof of the Claim]
  Fix an uncountable regular cardinal $\rho<\kappa$ and a cardinal $\theta$ in $C^{(n)}$ above $\kappa$. By our previous claim, we can find cardinals $\rho<\bar{\kappa}<\bar{\theta}<\kappa$, with $\bar{\theta}\in C^{(n)}$, an elementary submodel $X$ of $H_{\bar{\theta}}$ with $V_{\bar{\kappa}}\cup\{\bar{\kappa}\}\subseteq X$, and an elementary embedding $\map{i}{X}{H_\theta}$ such that $i(\bar{\kappa})=\kappa$, $i(\rho)=\rho$ and $i\restriction\bar{\kappa}\neq\id_{\bar{\kappa}}$. 
  Set $\map{j=i\restriction V_{\bar{\kappa}}}{V_{\bar{\kappa}}}{V_\kappa}$. Since $j$ is  non-trivial, the \emph{Kunen's Inconsistency} implies that $j\restriction V_\rho=\id_\rho$, and so $\crit{j}>\rho$. 
    Define $\lambda=\crit{j}$. 
    
   \begin{subclaim*}
       $\lambda \in C^{(n)}$. 
   \end{subclaim*} 

   \begin{proof}[Proof of the Subclaim]
     Fix a $\Pi_{n-1}$-formula $\varphi(v_0,v_1)$ that is absolute between $V$ and $V_\lambda$, and a parameter $a\in V_\lambda$. In one direction, if $\exists x ~ \varphi(a,x)$ holds in $V_\lambda$, then our assumptions on $\varphi$ directly ensure that this statement also holds in $V$. 
  In the other direction, assume, towards a contradiction, that the statement $\exists x ~ \varphi(a,x)$ holds in $V$ and fails in $V_\lambda$. 
  First,  notice that $\bar{\kappa}\in C^{(n)}$, because as $\theta \in C^{(n)}$, $H_\theta$ satisfies ``$\kappa \in C^{(n)}$", and, by elementarity, the model $X$, and therefore also $H_{\bar{\theta}}$, satisfies  ``$\bar{\kappa}\in C^{(n)}$",  and since $\bar{\theta }\in C^{(n)}$, this is true. 
  In particular,  the statement $\exists x ~ \varphi(a,x)$ holds in $V_{\bar{\kappa}}$ and we can find $b\in V_{\bar{\kappa}}$ such that the statement $\varphi(a,b)$ holds in $V_{\bar{\kappa}}$, $V_\kappa$ and $V$. 
  Now, note that there exists an ordinal $0<\eta\leq\omega$ with the property that $j^m(\lambda)\leq\bar{\kappa}$  for all $m<\eta$ and  $\bar{\kappa}\leq\sup_{m<\eta}j^m(\lambda)$, because otherwise we would have $j^m(\lambda)<\bar{\kappa}$ for all $m<\omega$ and   $\sup_{m<\omega}j^m(\lambda)$ would be a fixed point of $j$ below $\bar{\kappa}$, which is impossible by the \emph{Kunen  Inconsistency}. 
   Hence, we can  pick $m<\omega$ minimal with the property that $b\in V_{j^m(\lambda)}$. Our assumptions then imply that $m>0$. Moreover, the minimality of $m$ ensures that $j^m(\lambda)<\kappa$.  
  Thus, we know that $$V_\kappa\models \exists x\in V_{j^m(\lambda)} ~ \varphi(x,a),$$ which, by elementarity, directly implies that 
  $$V_{\bar{\kappa}}\models \exists x\in V_{j^{m-1}(\lambda)} ~ \varphi(x,a).$$ 
  Since $\kappa,\bar{\kappa}\in C^{(n)}$, and therefore $V_{\bar{\kappa}}\prec_{\Sigma_n} V_\kappa$, this shows that $$V_\kappa\models \exists x\in V_{j^{m-1}(\lambda)} ~ \varphi(x,a).$$ 
 By iterating this argument $m$-many times, we can conclude that 
  $$V_{\bar{\kappa}}\models \exists x\in V_\lambda ~ \varphi(x,a)$$  and our assumptions on $\varphi$ then show that the statement $\exists x ~ \varphi(a,x)$ holds in $V_\lambda$, contradicting our initial assumption. 
  \end{proof}

  For each ordinal $\alpha$ in the interval $( \lambda ,\bar{\kappa})$, the restricted map  $\map{j\restriction V_\alpha}{V_\alpha}{V_{j(\alpha)}}$ is an elementary embedding with critical point $\lambda$. 
  Note that  $j(\lambda)\in C^{(n)}$, because  $\lambda \in C^{(n)}$ implies that $\lambda \in (C^{(n)})^{V_{\bar{\kappa}}}$, elementarity ensures that $j(\lambda)\in (C^{(n)})^{V_\kappa}$ and the fact that $\kappa\in C^{(n)}$ allows us to conclude that $j(\lambda)\in C^{(n)}$. 
  Now, for each  $\alpha\in(\lambda,\bar{\kappa})$, the statement  \anf{\emph{There exists an ordinal $\beta$ and an elementary embedding $\map{i}{V_\alpha}{V_\beta}$  with $\crit{i}=\lambda$ and $i(\lambda) \in C^{(n-1)}$}} can be formulated by a $\Sigma_n$-formula with  parameters $\alpha$ and $\lambda$, and it holds in $V$. Hence, it also holds in $V_{\bar{\kappa}}$. Thus, in $V_{\bar{\kappa}}$, for every ordinal $\alpha$ greater than $\lambda$, there is an ordinal $\beta$ and   an elementary embedding $\map{i}{V_\alpha}{V_\beta}$ with $\crit{i}=\lambda$ and $i(\lambda)\in C^{(n-1)}$. 
  Elementarity then implies that, in $V_\kappa$,  for every ordinal $\alpha$ greater than $j(\lambda)$, there is an ordinal $\beta$ and an elementary embedding $\map{i}{V_\alpha}{V_\beta}$ with  $\crit{i}=j(\lambda)$ and  $i(j(\lambda))\in C^{(n-1)}$. Since $\kappa$ is an element of $C^{(n+1)}$, this statement also holds  in $V$ and we can conclude that  $j(\lambda)$ is a cardinal in the interval $(\rho,\kappa)$ that is $C^{(n-1)}$-extendible.  
 \end{proof}
 
 The above computations directly yield the implication $(2) \Rightarrow (1)$ of the theorem. Moreover, the implication $(1) \Rightarrow (3)$ follows directly from a combination of Theorem \ref{theorem:JoanBoldSupercompactSigma22}, Proposition \ref{proposition:CnSUNF-Cn+1} and Lemma \ref{lemmaCnSR}. This completes the proof of the theorem, because the implication $(3) \Rightarrow (2)$ is immediate. 
\end{proof}

We end this section by determining the class-principle corresponding to weak SR, i.e., we show that the following schemes are equivalent over $\ZFC$: 
\begin{enumerate}
  \item $\On$ is essentially subtle. 

  \item For every natural number $n$, there exists a $C^{(n)}$-strongly unfoldable cardinal. 
  
  \item For every natural number $n$, there exists a proper class of $C^{(n)}$-strongly unfoldable cardinals. 
    
  \item  For every natural number $n$ and every class $\Ce$ of structures of the same type, there exists a cardinal $\kappa\in C^{(n)}$ with the property that  that $\HSR^-_\Ce(\kappa)$ holds. 
 \end{enumerate}

\begin{proof}[Proof of Theorem \ref{theorem:OrdSubtle1}]
 $(1)\Rightarrow(2)$: Assume that for some natural number $m>1$, there are no $C^{(m)}$-strongly unfoldable cardinals. 
 Using a canonical coding of both G\"odel numbers of $\calL_\in$-formulas and subsets of $V_\kappa$ into subsets of $\kappa$ for $\kappa\in C^{(1)}$, our assumption allows us to apply  Theorem \ref{theorem:SigmaNsunf} to find an $\calL_\in$-formula $\varphi(v_0,v_1)$ with the property that for every cardinal $\kappa$ in $C^{(1)}$, there exists a subset $E$ of $\kappa$ and an ordinal $\kappa<\gamma\in C^{(m)}$ with the property that $\varphi(\kappa,E)$ holds in $V_\gamma$ and $\neg\varphi(\alpha,E\cap\alpha)$ holds in $V_\beta$ for all $\alpha<\beta<\kappa$ with $\alpha\in C^{(1)}$ and $\beta\in C^{(m)}$. 
 Now, let $\calE$ be the unique  class function with domain $\On$ such that the following statements hold:
 \begin{itemize}
     \item If $\gamma\in\On\setminus C^{(1)}$, then $\calE(\gamma)=\POT{\gamma}$. 
     
     \item If $\kappa\in C^{(1)}$, then $\calE(\kappa)$ consists of all subsets $E$ of $\kappa$ with the property that there exists an ordinal $\kappa<\gamma\in C^{(m)}$ with the property that $\varphi(\kappa,E)$ holds in $V_\gamma$ and $\neg\varphi(\alpha,E\cap\alpha)$ holds in $V_\beta$ for all $\alpha<\beta<\kappa$ with $\alpha\in C^{(1)}$ and $\beta\in C^{(m)}$. 
 \end{itemize}  
 
 Assume, towards a contradiction, that there are  $\rho<\kappa$ in $C^{(m+1)}$ and $E\in\calE(\kappa)$ with $E\cap\rho\in\calE(\rho)$. Then there exists an ordinal $\rho<\gamma\in C^{(m)}$ with the property that $\varphi(\rho,E\cap\rho)$ holds in $V_\gamma$. Since $\kappa$ is an element of $C^{(m+1)}$, we can find such a $\gamma$ that is smaller than $\kappa$. But this contradicts the fact that $E$ is an element of $\calE(\kappa)$. These arguments show that $\On$ is not essentially subtle in this case.

 $(2)\Rightarrow(3)$: Assume that for some natural number $m$, there are only boundedly many $C^{(m)}$-strongly unfoldable cardinals, and let $\lambda$ denote the least upper bound of the set of all $C^{(m)}$-strongly unfoldable cardinals. Since the set $\{\lambda\}$ is definable by an $\calL_\in$-formula without parameter, Proposition \ref{proposition:CnSUNF-Cn+1} implies that for all sufficiently large natural numbers $n$, there is no $C^{(n)}$-strongly unfoldable cardinal. 
 
 $(3)\Rightarrow(4)$: This implication follows directly from Proposition \ref{proposition:CnSUNF-Cn+1} and Lemma \ref{lemmaCnSR}. 
 
 $(4)\Rightarrow(1)$: Assume that for every natural number $n$ and every class $\Ce$ of structures of the same type, there exists a cardinal $\kappa\in C^{(n)}$ with the property  that $\HSR^-_\Ce(\kappa)$ holds. 
 Let $C$ be a closed unbounded class of ordinals and let $\calE$ be a class function on the ordinals with the property that $\emptyset\neq\calE(\gamma)\subseteq\POT{\gamma}$ holds for all $\gamma\in\On$. Then there is a natural number $n>0$ such that every element $\kappa$ of $C^{(n)}$ with the property that $V_\kappa$ contains the parameters used in the definition of $C$ is a limit point of $C$. 
 Let $\calL$ denote the first-order language that extends $\calL_\in$ by a constant symbol $\dot{d}_n$ for every natural number $n$, unary function symbols $\dot{e}$ and $\dot{s}$, and  unary relation symbols $\dot{C}$ and $\dot{E}$. Define $\Ce$ to be the class of $\calL$-structures of the form $\langle V_\rho,\in,\seq{d_n}{n<\omega},e,s,C\cap\rho,E\rangle$ such that the following statements hold: 
 \begin{itemize}
     \item $\rho\in C^{(n)}$ and $V_\rho$ contains the parameters used in the definition of $C$.  
     
     \item $E\in\calE(\rho)$. 
     
     \item If $\cof{\rho}=\omega$, then $\seq{d_n}{n<\omega}$ is a strictly increasing cofinal sequence in $\rho$. 
     
     \item $e(\gamma)\in\calE(\gamma)$ for all $\gamma<\rho$. 
     
     \item $s(\gamma)=\min(C\setminus(\gamma+1))$ for all $\gamma<\rho$. 
 \end{itemize}
 
 By our assumption, there exists $\kappa\in C^{(n)}$ with the property that  $\HSR^-_\Ce(\kappa)$ holds. Then there exists $B$ in $\Ce$ with domain $V_\kappa$ and, since $\betrag{V_\kappa}=\kappa$, we find $A\in \Ce\cap H_\kappa$ such that there exists an elementary embedding $j$ of $A$ into $B$. Pick $\rho\in C^{(n)}$ with the property that $V_\rho$ is the domain of $A$. If $j$ is the trivial embedding, then $\rho<\kappa$ are elements of $C$ and $\dot{E}^B$ is an element of $\calE(\kappa)$ with $\dot{E}^B\cap\rho=\dot{E}^A\in\calE(\rho)$. In the following, assume that the embedding $j$ is non-trivial and set $\lambda=\crit{j}<\rho$. 
 
 \begin{claim*}
  $\lambda\in C$. 
 \end{claim*}
 
 \begin{proof}[Proof of the Claim]
  Assume, towards a contradiction, that  $\lambda\notin C$. Since elementarity implies that $j(\min(C))=\min(C)$, the \emph{Kunen Inconsistency} ensures that $\lambda>\min(C)$. Set $\beta=\sup(C\cap\lambda)<\lambda$ and $\gamma=\min(C\setminus(\beta+1))\in(\lambda,\rho)$. Then $\dot{s}^A(\beta)=\gamma$ and hence we have $j(\gamma)=\gamma$ and $j(V_{\gamma+2})=V_{\gamma+2}$, contradicting the  \emph{Kunen Inconsistency}. 
 \end{proof}
 
 By elementarity, we now know that $\lambda<j(\lambda)$ are elements of $C$ and $j(\dot{e}^A(\lambda))=\dot{e}^B(\lambda)$ is an element of $\calE(j(\lambda))$ with $j(\dot{e}^A(\lambda))\cap\lambda=\dot{e}^A(\lambda)\in\calE(\lambda)$. These computations show that (1) holds in this case. 
\end{proof}

%%%%%%%%%%%%%%%%%%
%%%%%%%%%%%%%%%%%%

\section{On $\WPR$ for arbitrary classes of structures}\label{section10}

We shall deal  next with the extension of Theorem \ref{theorem:CharWPR} to arbitrary classes of structures. This is given by Theorem \ref{WPSRC(n)}. Namely, we will show that for every natural number $n>1$, the following statements are equivalent for every cardinal $\kappa$:
\begin{enumerate}
 \item $\kappa$ is either $C^{(n)}$-strongly unfoldable or a limit of $\Sigma_{n+1}$-strong cardinals.

 \item The principle $\WPR_\Ce(\kappa)$ holds for every class $\Ce$ of structures of the same type that is definable by a $\Sigma_{n+1}$-formula with parameters in $V_\kappa$.
\end{enumerate}

The proof will use similar arguments to those in the proof of Theorem \ref{theorem:CharWPR}, given in section \ref{section7}.  The implication $(1) \Rightarrow (2)$ follows from the following  $C^{(n)}$-version of Lemma \ref{lemma:PRfromStrongUnfoldability}, which can be proved in the same way, using Proposition \ref{proposition:CnSUNF-Cn+1}.

\begin{lemma}\label{lemma:PRfromStrongUnfoldability2}
 If $n>0$ is a natural number and $\kappa$ is a $C^{(n)}$-strongly unfoldable cardinal, then  $\WPR_\Ce(\kappa)$ holds for every  non-empty class $\Ce$ of structures of the same type that is definable by a $\Sigma_{n+1}$-formula with parameters in $V_\kappa$. 
\end{lemma}

\begin{proof}
 Since Proposition \ref{proposition:CnSUNF-Cn+1} ensures that $\kappa\in C^{(n+1)}$, we have $\emptyset\neq\Ce\cap V_\kappa\in H_{\kappa^+}$. 
 Pick a substructure $X$ of $\prod(\Ce\cap V_\kappa)$ of cardinality at most $\kappa$, a structure $B$ in $\Ce$, a cardinal $\kappa<\delta\in C^{(n+1)}$  with $B\in V_\delta$, and an elementary submodel $M$ of $H_{\kappa^+}$ of cardinality $\kappa$ with  $V_\kappa\cup\{\Ce\cap V_\kappa,X\}\cup{}^{{<}\kappa}M\subseteq M$. 
 Fix  a transitive set $N$ with $V_\delta\subseteq N$ and an elementary embedding $\map{j}{M}{N}$ with $\crit{j}=\kappa$, $j(\kappa)>\delta$ and $V_\delta\prec_{\Sigma_n}V_{j(\kappa)}^N$. 
 This setup ensures that $B\in\Ce\cap V_\delta\subseteq j(\Ce\cap V_\kappa)$ and hence the function $$\Map{h}{X}{B}{f}{j(f)(B)}$$ is a well-defined homomorphism.  
  \end{proof}

\begin{proof}[Proof of Theorem \ref{WPSRC(n)}]
 Let $n>0$ be a natural number and let $\kappa$ be a cardinal with the property that $\WPR_\Ce(\kappa)$ for every class $\Ce$ of structures of the same type that is definable by a $\Sigma_{n+1}$-formula with parameters in $V_\kappa$. 
 Then Lemma \ref{lemma:Sigma2Reflex} shows that $\kappa\in C^{(n+1)}$. 
 Assume, towards a contradiction, that $\kappa$ is neither $C^{(n)}$-strongly unfoldable nor a limit of $\Sigma_{n+1}$-strong cardinals. 
 Pick an ordinal $\alpha<\kappa$ such that the interval $[\alpha,\kappa)$ contains no $\Sigma_{n+1}$-strong cardinals. 
 Given a cardinal $\rho$ that is not $\Sigma_{n+1}$-strong, we let $\eta_\rho$ denote the least cardinal $\delta>\rho$ such that $\rho$ is not $\delta$-$\Sigma_{n+1}$-strong. Since the property of not being a $\Sigma_{n+1}$-strong cardinal can be defined by a $\Sigma_{n+1}$-formula, and since $\kappa \in C^{(n+1)}$,  the interval $(\alpha,\kappa)$ is closed under the function $\rho\longmapsto \eta_\rho$, and therefore it contains unboundedly many cardinals $\xi$ with the property that $\eta_\rho<\xi$ holds for all cardinals $\alpha\leq\rho<\xi$.  
 Finally, using Theorem \ref{theorem:SigmaNsunf}.\eqref{item:shrewd} and some standard fact about definability in models of the form $V_\eta$, we can find an $\calL_\in$-formula $\Phi(v_0,v_1)$, an ordinal $\theta>\kappa$ in $C^{(n)}$, and $E\subseteq V_\kappa$ with the property that $\Phi(\kappa,E)$ holds in $V_{\theta+1}$ and for all $\beta<\gamma<\kappa$ with $\gamma \in C^{(n)}$ the statement $\Phi(\beta,E\cap V_\beta)$ does not hold in $V_{\gamma+1}$. 
 
  Let $\calL^\prime$ be the first-order language that extends the language of set theory with  
  a binary relation symbol $\dot{S}$,  constant symbols $\dot{\kappa}$, $\dot{u}$, and  $\dot{c}$, and a unary function symbol $\dot{e}$. Let $\calL$ denote the first-order language that extends $\calL^\prime$ by an $(n+1)$-ary predicate symbol $\dot{T}_\varphi$ for  every $\calL^\prime$-formula $\varphi(v_0,\ldots,v_n)$ with $(n+1)$-many free variables, and let 
  $\calS_\calL$  be the class of structures  $A$ such that there exists a cardinal $\kappa_A$ in $C^{(n)}$ and a cardinal $\theta_A>\kappa_A$ also in $C^{(n)}$ 
 %transitive set $M_A$ with the property 
 such that the following  hold: 
 \begin{enumerate}
     %\item %$M_A$ has cardinality $2^{\kappa_A}$ \todo[fancyline]{Needed?} and 
     %$V_{\kappa_A+1}\subseteq M_A$ and for some ordinal $\vartheta>\kappa_A$, there exists an elementary  embedding $\map{j}{M_A}{V_\vartheta}$ with $j\restriction V_{\kappa_A+1}=\id_{V_{\kappa_A+1}}$. 
     
     \item The domain of $A$ is $V_{\theta_A +1}$. 
     
     \item $\in^A={\in}\restriction{V_{\theta_A+1}}$, $\dot{\kappa}^A=\kappa_A$ and $\dot{u}^A=\theta_A$. 
     
     \item If $\varphi(v_0,\ldots,v_n)$ is an $\calL^\prime$-formula, then $$\dot{T}_\varphi^A ~ = ~ \Set{\langle x_0,\ldots,x_n\rangle\in V_{\theta_A +1}^{n+1}}{A\models\varphi(x_0,\ldots,x_n)}.$$
     %$n$ is a natural number and $A^-$ denotes the $\calL$-retract of $A$, then  $$\dot{T}_n^{A} ~ = ~ \Set{\langle a, x_0,\ldots,x_{n-1}\rangle\in M_A^{n+1}}{a\in\mathsf{Fml}_\calL, ~ \mathsf{Sat}_\calL(A^-,a,\langle x_0,\ldots,x_{n-1}\rangle)},$$ where $\mathsf{Fml}_\calL\subseteq V_\omega$ denotes the set of formalized $\calL$-formulas and $\mathsf{Sat_\calL}$ denotes the corresponding formalized satisfaction relation.  
 \end{enumerate}
 
 Let $\Ce$ denote the class of all $A\in\calS_\calL$ such that the following statements hold: 
 \begin{itemize}
  \item $\dot{c}^A=\alpha<\kappa_A$. 
  
  \item The interval $[\alpha,\kappa_A)$ contains no $\Sigma_{n+1}$-strong cardinals. 
  
  %\item $\dot{E}^A\subseteq V_{\kappa_A}$. 
  
  \item $\dot{S}^A=\Set{\langle\rho,\gamma\rangle\in\kappa_A\times\kappa_A}{\textit{$\rho$ is a $\gamma$-$\Sigma_{n+1}$-strong cardinal}}$. 
  
  \item If $\alpha<\delta<\kappa_A$ is a cardinal, then $\dot{e}^A(\delta)$ is a cardinal below $\kappa_A$ and  is the smallest cardinal $\xi$ greater than $\delta$ that has the property that $\eta_\rho<\xi$ holds for all cardinals $\alpha\leq\rho<\xi$. 
   %there exists a sequence $\seq{\rho_n}{n<\omega}$ with $\rho_0=\rho$, $\dot{e}^A(\rho)=\sup_{n<\omega}\rho_n<\kappa_A$ and $\rho_{n+1}=\eta_{\rho_n}$ for all $n<\omega$. 
 \end{itemize}
 
 It is easily seen that the class $\Ce$ is  $\Sigma_{n+1}$-definable  with parameter $\alpha$. 
 In addition, the fact that  $\kappa$ is an element of $C^{(n+1)}$ implies that $\sup\Set{\kappa_A}{A\in\Ce\cap V_\kappa}=\kappa$ and there exists a structure $B$ in $\Ce$ with $\kappa_B=\kappa$ and $\theta_B=\theta$. % and $\dot{E}^B=E$. 
 Let $C$ be a cofinal subset of $\kappa$ of order-type $\cof{\kappa}$.  Let $X$ be  a substructure   of $\prod(\Ce\cap V_\kappa)$ of cardinality $\kappa$ with the property that $f^C\in X$, $f^E\in X$,  and $f_x,f^x\in X$ for all $x\in V_\kappa$. 
 By our assumption, there is a homomorphism $\map{h}{X}{B}$ and we   define $$\lambda ~ = ~ \min\Set{\rank{x}}{f_x \{ \kappa \}, ~ h(f_x)=\dot{u}^B}$$ and $$\chi ~ = ~ \sup\Set{h(f_\beta)}{\beta<\lambda}.$$
 Note that the results of Section \ref{prodref} imply that $\lambda, \chi\leq \kappa$. Moreover, Lemma \ref{lemma:Prod2}.\eqref{item:Prod7} shows that $h(f_\lambda)=\dot{u}^B\neq\lambda$ and this implies that   $\alpha<\lambda$ (see the proof of Theorem \ref{theorem:CharWPR}). 
 Let $$\mu ~ = ~ \min\Set{\beta\leq\lambda}{h(f_\beta)\neq\beta}.$$
 and apply  Lemma \ref{lemma:ProdCons}.\eqref{item:ProdCons2.2New} to show that $\mu$ is an inaccessible cardinal. Moreover, Lemma  \ref{lemma:Prod2}.\eqref{item:Prod7} implies that $h(f_x)=x$ holds for all $x\in V_\mu$.

 \begin{claim*}
  $\mu<\kappa$. 
 \end{claim*}
 
 \begin{proof}[Proof of the Claim]
  Assume, towards a contradiction, that $\mu=\lambda=\kappa$ holds. 
  Since Lemma \ref{lemma:Prod1}.\eqref{item:Prod1} implies  $h(f^E)\subseteq\kappa$, we can apply Lemma \ref{lemma:ProdCons}.\eqref{item:ProdCons1} to conclude that $h(f^E)=E$ and hence $$B\models \Phi(\dot{\kappa},h(f^E)).$$ 
  Another application of Lemma \ref{lemma:Prod1}.\eqref{item:Prod1} then yields $A\in\Ce\cap V_\kappa$ with $$A\models\Phi(\dot{\kappa},f^E(A))$$ and this shows that $\Phi(\kappa_A,E\cap V_{\kappa_A})$ holds in $V_{\theta_A+1}$. 
  Since $\kappa_A<\theta_A<\kappa$ and $\theta_A \in C^{(n)}$, this contradicts the choice of $E$. 
  \end{proof}

  \begin{claim*}
   $\mu<\lambda$. 
  \end{claim*}
  
  \begin{proof}[Proof of the Claim]
   Assume, towards a contradiction, that $\mu=\lambda<\kappa$. 
   Then Lemma \ref{lemma:ProdCons}.\eqref{item:ProdCons3} shows that $h(f^\mu)=\kappa$ and we can apply Lemmas \ref{lemma:ProdReflExtender} and  \ref{Lemma:HomStrongA} to show that $\mu$ is a $\kappa$-strong cardinal. 
 The claim will be proved once we show that $\mu$ is a $\kappa$-$\Sigma_{n+1}$-strong cardinal, for this contradicts the fact that $\alpha<\mu<\eta_\mu<\kappa$. 
 
 Now, \cite[Proposition 5.5]{BagariaWilsonRefl} shows that  $\mu$ is a $\kappa$-$\Sigma_{n+1}$-strong cardinal if and only if there exists a  transitive $M$ and an elementary embedding $\map{j}{V}{M}$ with  $\crit{j}=\mu$, $j(\mu)>\kappa$, $V_\kappa \subseteq M$, and $M\models ``\kappa \in C^{(n)}"$. 
  Thus, letting $\mathcal{E}$ be the $(\mu ,\kappa)$-extender given by Lemma \ref{lemma:ProdReflExtender}, and letting $\map{j_{\mathcal{E}}}{V}{M_{\mathcal{E}}}$ be as in Lemma \ref{Lemma:HomStrongA}, it only remains  to show that $\bar{M}_{\mathcal{E}}\models ``\kappa \in C^{(n)}"$, where $\bar{M}_\Ee$ is the transitive collapse of $M_\Ee$. 
  Since $\kappa$ is a limit point of $C^{(n)}$, it suffices to show that   if $\gamma <\kappa$ belongs to  $C^{(n)}$, then $\bar{M}_\Ee\models ``\gamma \in C^{(n)}"$. So, fix such an ordinal  $\gamma$.
  %
  %For every $a\in [\kappa]^{<\omega}$, the map $k_a:M_a\to M_\Ee$  given by:
  %$$k_a([f]_{E_a})=[a, [f]_{E_a}]$$
  %is an elementary embedding such that $f_\Ee =k_a \circ j_a$.
  
  Let $\map{f}{[\mu]^1}{\mu}$ be such that $f(\{x\})=x$. It is well known that $k_{\{ \gamma \}}( [f]_{E_{\{\gamma\}}})=\gamma$, where $\map{k_{\{ \gamma\}}}{M_{\{\gamma\}}}{\bar{M}_\Ee}$ is the standard map given by $k_{\{ \gamma\}}([g]_{E_{\{\gamma\}}})=\pi([\{ \gamma\}, [g]_{E_{\{\gamma\}}}])$, for any $\map{g}{[\mu]^1}{V}$, and where $\map{\pi}{M_\Ee}{\bar{M}_\Ee}$ is the transitive collapse (see \cite[Lemma 26.2(a)]{MR1994835}). 

 Since being a singleton whose unique element belongs to  $C^{(n)}$ is a predicate in the language of every structure $A\in \Ce \cap V_\kappa$, letting $Z=\Set{ \{ x\}\in [\mu]^1}{x\in C^{(n)}}$, by Lemma \ref{lemma:Prod1}.\eqref{item:Prod1} and, since $h(f^\mu)=\kappa$ and $\kappa \in C^{(n)}$, we have that  $h$ maps $f^Z$ to the set $\Set{\{ x\}\in [\kappa]^1}{x\in C^{(n)}\cap V_{\kappa}}$. Thus, $\{ \gamma \} \in h(Z)$, and therefore $Z\in E_{\{ \gamma\}}$. 
 Hence, $M_{\{ \gamma\}}\models ``[f]_{E_{\{ \gamma \}}}\in C^{(n)}"$, and therefore we can conclude that $M_\Ee \models ``[\{ \gamma\}, [f]_\Ee]\in C^{(n)}"$, which yields $\bar{M}_\Ee \models ``\gamma \in C^{(n)}"$, as desired. 
  \end{proof}

 The claim above shows that $\mu<\lambda\leq\kappa$ and $h(f_\mu)\neq\dot{u}^B$. 
 So let $$\Map{j}{V_\lambda}{V_\chi}{x}{h(f_x)}$$ be the  non-trivial $\Sigma_1$-elementary embedding with $\crit{j}=\mu$  given by Lemma \ref{lemma:Prod2}.\eqref{item:ProdBonus}. 
 Like  in the proof of Theorem \ref{theorem:StrongProductSigma2} we may now use the $\Sigma_1$-elementarity of $j$ and the \emph{Kunen Inconsistency} to conclude that    $\alpha<\mu$. 
 Since $\alpha<\mu<\lambda\leq\kappa$, we can pick a cardinal $\mu<\xi<\kappa$ that is  minimal   above $\mu$ with the property that $\eta_\rho<\xi$ holds for all cardinals $\alpha\leq\rho<\xi$. 
  Given $A\in\Ce\cap V_\kappa$ with $\mu<\kappa_A$, we then have $\xi=\dot{e}^A(\mu)<\kappa_A$. Using Lemma \ref{lemma:Prod1}.\eqref{item:Prod1}, this shows that $h(f_\xi)\neq\dot{u}^B$ and  $\xi<\lambda$.

  \begin{claim*}
   If $i<\omega$, then $j^i(\mu)<j^{i+1}(\mu)<\xi$ and $j^i(\mu)$ is a $j^{i+1}(\mu)$-$\Sigma_{n+1}$-strong cardinal. 
  \end{claim*}
  
   \begin{proof}[Proof of the Claim]
    Since $\mu<\kappa$ and Lemma \ref{lemma:ProdCons}.\eqref{item:ProdCons4} shows that $j(\mu)=h(f^\mu)$,  we can apply Lemma \ref{Lemma:HomStrongA} and argue as in the last claim above to conclude that $\mu$ is $j(\mu)$-$\Sigma_{n+1}$-strong,  and this implies that $\mu<j(\mu)<\eta_\mu<\xi$. 
   Now, assume that for some $i<\omega$, we have $j^i(\mu)<j^{i+1}(\mu)<\xi$ and $j^i(\mu)$ is a $j^{i+1}(\mu)$-$\Sigma_{n+1}$-strong cardinal. Then $\langle j^i(\mu),j^{i+1}(\mu)\rangle\in\dot{S}^A$ for all $A\in\Ce\cap V_\kappa$ with $\xi<\kappa_A$ and the fact that $h(f_\xi)\neq\dot{u}^B$ allows us to use Lemma \ref{lemma:Prod1}.\eqref{item:Prod1} to show that $\langle j^{i+1}(\mu),j^{i+2}(\mu)\rangle\in\dot{S}^B$. Hence, we know that $j^{i+1}(\mu)$ is a $j^{i+2}(\mu)$-$\Sigma_{n+1}$-strong cardinal and $j^{i+1}(\mu)<j^{i+2}(\mu)<\eta_{j^{i+1}(\mu)}<\xi$. 
   \end{proof}
   
   Now, define $$\tau ~ = ~ \sup_{i<\omega}j^i(\mu) ~ \leq ~ \xi ~ < ~ \lambda,$$ and apply Lemma  \ref{lemma:Prod2}.\eqref{item:Prod7} to show that $\lambda$ is a limit ordinal and use  the $\Sigma_1$-elementarity of $j$ to conclude that $j(V_{\tau+2})=V_{\tau+2}$.  
  Since this entails  that $\map{j\restriction V_{\tau+2}}{V_{\tau+2}}{V_{\tau+2}}$ is a non-trivial elementary embedding, we obtain a contradiction to the \emph{Kunen Inconsistency}.
  
  The above computations prove the implication $(2)\Rightarrow(1)$ of the theorem. The implication $(1) \Rightarrow (2)$ follows directly from Lemma \ref{lemma:PRfromStrongUnfoldability2}. 
  \end{proof}

  We conclude this section with the proof of Theorem \ref{theorem:Ordsubtle2}, i.e., we prove that the following schemas are equivalent over $\ZFC$: 
  \begin{enumerate}
   \item $\On$ is essentially subtle. 
  
   \item  For every non-empty class $\Ce$ of structures of the same type, there exists a cardinal $\kappa$ such that $\WPR_\Ce(\kappa)$ holds. 
 \end{enumerate}
  
  \begin{proof}[Proof of Theorem \ref{theorem:Ordsubtle2}]
  The implication $(1)\Rightarrow (2)$ follows from a combination of Theorem \ref{theorem:OrdSubtle1} and Lemma \ref{lemma:PRfromStrongUnfoldability2}.
  
   To show that $(2)\Rightarrow (1)$, assume, aiming for a contradiction, that there is a closed unbounded class $C$ of ordinals and a class function $\calE$ with domain $\On$ and the property that $\emptyset\neq\calE(\gamma)\subseteq\POT{\gamma}$ holds for all $\gamma\in\On$ and $E\cap\beta\notin\calE(\beta)$ holds for all $\beta<\gamma$ in $C$ and $E\in\calE(\gamma)$. 
   As in the proof of Theorem \ref{main:ProdSubtle}, let $\alpha$ denote the least cardinal of uncountable cofinality in $C$ and let $\calL^\prime$ denote the first-order language that extends $\calL_\in$ by a unary predicate symbol $\dot{C}$,  constant symbols $\dot{\kappa}$, $\dot{c}$ and $\dot{u}$, and  unary function symbols $\dot{e}$ and $\dot{s}$. Then let $\calL$ denote the first-order language extending  $\calL^\prime$ by an $(n+1)$-ary predicate symbol $\dot{T}_\varphi$ for  every $\calL^\prime$-formula $\varphi(v_0,\ldots,v_n)$ with $(n+1)$-many free variables, and let $\calS_\calL$  be the class of $\calL$-structures as defined at the beginning of Section \ref{prodref}. 
   Now, let $\Ce$ denote the class of all $A\in\calS_\calL$ such that the following statements hold: 
  \begin{itemize}
   \item $\kappa_A$ is a limit point of $C$ above $\alpha$. 
   
   \item $\dot{c}^A=\alpha$ and  $\dot{C}^A=C\cap\kappa_A$. 
  
  \item If $\gamma<\kappa_A$, then $\dot{e}^A(\gamma)\in\calE(\gamma)$ and $\dot{s}^A(\gamma)=\min(C\setminus(\gamma+1))$.  
 \end{itemize}
 
 Then $\Ce\neq\emptyset$. Assume, towards a contradiction, that there is a cardinal $\zeta$ with the property that $\WPR_\Ce(\zeta)$ holds. 
 Set $$\kappa ~ = ~ \sup\Set{\kappa_A}{A\in\Ce\cap V_\zeta} ~ \in ~ \Lim(C) \cap C^{(1)}\cap (\alpha,\zeta].$$ Let $D$ be a cofinal subset of $\kappa$ of order-type $\cof{\kappa}$, and let $E$ be an element of $\calE(\kappa)$.

 For each set $x$, define functions $f_x$ and $f^x$ as in Section \ref{prodref}. 
 Pick $B\in \Ce$ with $\kappa_B>\zeta$ and $\dot{e}^B(\kappa)=E$. Fix a substructure $X$ of $\prod(\Ce\cap V_\zeta)$ of cardinality $\kappa$ such that $f^D,f^E\in X$ and $f_x,f^x\in X$   for all $x\in V_\kappa$. 
 Then there is a homomorphism $\map{h}{X}{B}$ and we can define  $$\lambda ~ = ~ \min\Set{\rank{x}}{f_x\in X, \, h(f_x)=\dot{u}^B} ~ \leq ~ \kappa$$ as well as $$\chi ~ = ~ \sup\Set{h(f_\alpha)}{\alpha<\lambda} ~ \leq ~ \kappa_B.$$ 
 Our setup then ensures that $\alpha<\lambda$ and an argument presented  in the proof of Theorem \ref{main:ProdSubtle} shows that $\lambda\in C$. Using Lemma \ref{lemma:Prod2}.\eqref{item:ProdBonus}, we know that  $$\Map{j}{V_\lambda}{V_\chi}{x}{h(f_x)}$$ is a $\Sigma_1$-elementary embedding. Next, Lemma \ref{lemma:Prod2}.\eqref{item:Prod5} allows us to define $\mu\leq\lambda$ to be the minimal ordinal with $h(f_\mu)\neq\mu$. By Lemma  \ref{lemma:Prod2}.\eqref{item:Prod7}, we then have $h(f_x)=x$  for all $x\in V_\mu$.

  \begin{claim*}
   $\mu<\kappa$. 
  \end{claim*}
  
  \begin{proof}[Proof of the Claim]
   Assume, towards a contradiction, that $\mu=\lambda=\kappa$ holds. 
   Lemma \ref{lemma:ProdCons}.\eqref{item:ProdCons1} now shows that $h(f^E)\cap\kappa=E=\dot{e}^B(\kappa)$ and hence the cardinal  $\kappa\in C\cap\kappa_B=\dot{C}^B$   witnesses that $$B\models\anf{\exists\beta<\dot{\kappa} ~ [\beta\in\dot{C}\wedge h(f^E)\cap\beta=\dot{e}(\beta)]}.$$
   Then Lemma \ref{lemma:Prod1}.\eqref{item:Prod1} yields $A\in\Ce\cap V_\zeta$ and $\beta\in C\cap\kappa_A$ with $E\cap\beta=f^E(A)\cap\beta=\dot{e}^A(\beta)\in\calE(\beta)$. Since  $\beta<\kappa_A\leq\kappa$, this yields a contradiction.   
  \end{proof}
  
  Define $E_0=\dot{e}^B(\mu)\in\calE(\mu)$. 
  
  \begin{claim*}
   $\mu<\lambda$. 
  \end{claim*}
  
  \begin{proof}[Proof of the Claim]
   Assume, towards a contradiction, that $\mu=\lambda$ holds. Then $\mu\in C\cap\kappa$ and $f^{E_0}\in X$. Lemma \ref{lemma:ProdCons}.\eqref{item:ProdCons1} now  implies that $h(f^{E_0})\cap \mu=E_0$ and, since Lemma \ref{lemma:Prod2}.\eqref{item:Prod5} shows that   $h(f_\lambda)=\dot{u}^B$, we have $$B\models\anf{h(f_\mu)=\dot{u} ~ \wedge ~ \exists\beta<\dot{\kappa} ~ [\beta\in\dot{C}\wedge h(f^{E_0})\cap\beta=\dot{e}(\beta)]}.$$
   By Lemma \ref{lemma:Prod1}.\eqref{item:Prod1}, there is $A\in\Ce\cap V_\zeta$ with $\kappa_A\leq\mu$ and $\beta\in C\cap\kappa_A$ with   $E_0\cap\beta  =  f^{E_0}(A)\cap\beta  = \dot{e}^A(\beta)\in\calE(\beta)$, a contradiction. 
  \end{proof}

 This shows that $\mu<\lambda\leq\kappa$. In addition, arguments already  contained in the proof of Theorem \ref{main:ProdSubtle} prove that   $\alpha<\mu\in C$.     Then  $\mu<j(\mu)=h(f_\mu)\neq\dot{u}^B$ and $\mu$ witnesses that $$B\models\anf{h(f_\mu)\neq\dot{u} ~ \wedge ~ \exists\beta< h(f_\mu) ~ [\beta\in\dot{C}\wedge h(f^{E_0})\cap\beta=\dot{e}(\beta)]}.$$
  By Lemma \ref{lemma:Prod1}.\eqref{item:Prod1}, there is $A\in\Ce\cap V_\zeta$ with $\kappa_A>\mu$ and $\beta\in C\cap\mu$ with $E_0\cap\beta = f^{E_0}(A)\cap\beta = \dot{e}^A(\beta) \in \calE(\beta)$, a contradiction.  
  \end{proof}

%%%%%%%%%%%%%%%%%%%
%%%%%%%%%%%%%%%%%%%

\bibliographystyle{amsplain} 
\bibliography{masterbiblio}

\providecommand{\bysame}{\leavevmode\hbox to3em{\hrulefill}\thinspace}
\providecommand{\MR}{\relax\ifhmode\unskip\space\fi MR }
\providecommand{\MRhref}[2]{%
  \href{http://www.ams.org/mathscinet-getitem?mr=#1}{#2}
}
\providecommand{\href}[2]{#2}
\begin{thebibliography}{10}

\bibitem{zbMATH06029795}
Joan Bagaria, \emph{{$C^{(n)}$}-cardinals}, Archive for Mathematical Logic
  \textbf{51} (2012), no.~3-4, 213--240.

\bibitem{Bag-SR}
\bysame, \emph{Large cardinals as principles of structural reflection}, Bull.
  Symb. Log. \textbf{29} (2023), no.~1, 19--70. \MR{4560533}

\bibitem{BCMR}
Joan Bagaria, Carles Casacuberta, A.~R.~D. Mathias, and Ji\v{r}\'{\i}
  Rosick\'y, \emph{Definable orthogonality classes in accessible categories are
  small}, Journal of the European Mathematical Society \textbf{17} (2015),
  no.~3, 549--589.

\bibitem{BGS}
Joan Bagaria, Victoria Gitman, and Ralf~D. Schindler, \emph{Generic
  {V}op\v{e}nka's {P}rinciple, remarkable cardinals, and the weak {P}roper
  {F}orcing {A}xiom}, Archive for Mathematical Logic \textbf{56} (2017),
  no.~1-2, 1--20.

\bibitem{BALU1}
Joan Bagaria and Philipp L\"{u}cke, \emph{Huge reflection}, Ann. Pure Appl.
  Logic \textbf{174} (2023), no.~1, Paper No. 103171, 32.

\bibitem{Ba-Ma:GR}
Joan Bagaria and Menachem Magidor, \emph{Group radicals and strongly compact
  cardinals}, Transactions of the {A}merican {M}athematical {S}ociety
  \textbf{366} (2014), 1857--1877.

\bibitem{BM2}
\bysame, \emph{On $\omega_1$-strongly compact cardinals}, J. Symb. Log.
  \textbf{79} (2014), no.~1, 266--278.

\bibitem{MR3519447}
Joan Bagaria and Jouko V\"{a}\"{a}n\"{a}nen, \emph{On the symbiosis between
  model-theoretic and set-theoretic properties of large cardinals}, J. Symb.
  Log. \textbf{81} (2016), no.~2, 584--604.

\bibitem{BagariaWilsonRefl}
Joan Bagaria and Trevor~M. Wilson, \emph{The weak {V}op\v{e}nka principle for
  definable classes of structures}, The Journal of Symbolic Logic (2022),
  1–27.

\bibitem{boney2022model}
Will Boney, Stamatis Dimopoulos, Victoria Gitman, and Menachem Magidor,
  \emph{Model theoretic characterizations of large cardinals revisited},
  Preprint, 2022.

\bibitem{MR4474130}
Sean Cox, \emph{Salce's problem on cotorsion pairs is undecidable}, Bull. Lond.
  Math. Soc. \textbf{54} (2022), no.~4, 1363--1374.

\bibitem{devlin_2017}
Keith~J. Devlin, \emph{Constructibility}, Perspectives in Logic, Cambridge
  University Press, 2017.

\bibitem{MR2279655}
Mirna D\v{z}amonja and Joel~David Hamkins, \emph{Diamond (on the regulars) can
  fail at any strongly unfoldable cardinal}, Ann. Pure Appl. Logic \textbf{144}
  (2006), no.~1-3, 83--95.

\bibitem{Fl:NMSC}
William~G. Fleissner, \emph{The normal {M}oore space conjecture and large
  cardinals}, Handbook of set-theoretic topology, North-Holland, Amsterdam,
  1984, pp.~733--760.

\bibitem{For:denseideal}
Matthew Foreman, \emph{An $\aleph_1$-dense ideal on $\aleph_2$}, Israel Journal
  of Mathematics \textbf{108} (1998), no.~1, 253.

\bibitem{FoMa:MutStat}
Matthew Foreman and Menachem Magidor, \emph{Mutually stationary sequences of
  sets and the non-saturation of the non-stationary ideal on
  $\mathcal{P}_\kappa (\lambda)$}, Acta Mathematica \textbf{186} (2001), no.~2,
  271--300.

\bibitem{GitPrikry}
Moti Gitik, \emph{Prikry-type forcings}, Handbook of set theory, Springer,
  2010, pp.~1351--1447.

\bibitem{hamkins2016vopvenka}
Joel~David Hamkins, \emph{The {V}op\v{e}nka principle is inequivalent to but
  conservative over the {V}op\v{e}nka scheme}, Preprint, 2016.

\bibitem{HJ:SU}
Joel~David Hamkins and Thomas~A. Johnstone, \emph{Strongly uplifting cardinals
  and the boldface resurrection axioms}, Archive for Mathematical Logic
  \textbf{56} (2017), no.~7, 1115--1133.

\bibitem{MR1221741}
Wilfrid Hodges, \emph{Model theory}, Encyclopedia of Mathematics and its
  Applications, vol.~42, Cambridge University Press, Cambridge, 1993.

\bibitem{jensennotes}
Ronald~B. Jensen and Kenneth Kunen, \emph{Some combinatorial properties of {L}
  and {V}}, Handwritten notes, available at
  {\url{https://www.mathematik.hu-berlin.de/~raesch/org/jensen/pdf/CPLV_2.pdf}},
  1969.

\bibitem{MR1994835}
Akihiro Kanamori, \emph{The {H}igher {I}nfinite}, second ed., Springer
  Monographs in Mathematics, Springer-Verlag, Berlin, 2003.

\bibitem{luecke2021strong}
Philipp L\"{u}cke, \emph{Strong unfoldability, shrewdness and combinatorial
  consequences}, Proc. Amer. Math. Soc. \textbf{150} (2022), no.~9, 4005--4020.

\bibitem{SRminus}
\bysame, \emph{Structural reflection, shrewd cardinals and the size of the
  continuum}, J. Math. Log. \textbf{22} (2022), no.~2, Paper No. 2250007, 43.

\bibitem{MR295904}
Menachem Magidor, \emph{On the role of supercompact and extendible cardinals in
  logic}, Israel J. Math. \textbf{10} (1971), 147--157.

\bibitem{MS:PPD}
Donald~A. Martin and John~R. Steel, \emph{A proof of projective determinacy},
  J. Amer. Math. Soc. \textbf{2} (1989), no.~1, 71--125.

\bibitem{MR0930262}
Pierre Matet, \emph{Some aspects of {$n$}-subtlety}, Z. Math. Logik Grundlag.
  Math. \textbf{34} (1988), no.~2, 189--192. \MR{930262}

\bibitem{MR3324126}
Norman~Lewis Perlmutter, \emph{The large cardinals between supercompact and
  almost-huge}, Arch. Math. Logic \textbf{54} (2015), no.~3-4, 257--289.

\bibitem{MR1649079}
Andr\'{e}s Villaveces, \emph{Chains of end elementary extensions of models of
  set theory}, J. Symbolic Logic \textbf{63} (1998), no.~3, 1116--1136.

\bibitem{Woodin:PNAS}
W.~Hugh Woodin, \emph{Supercompact cardinals, sets of reals, and weakly
  homogeneous trees}, Proceedings of the National Academy of Sciences
  \textbf{85} (1988), no.~18, 6587--6591.

\end{thebibliography}

\end{document}